\documentclass{article}

\usepackage[utf8]{inputenc}
\usepackage[utf8]{inputenc}
\usepackage[top=1.25in, bottom=1.25in, left=1.27in, right=1.27in]{geometry}
\usepackage{mathtools}
\usepackage[all]{xy}
\usepackage{amsthm, amsmath, amssymb, amscd, latexsym, multicol, verbatim, enumerate, graphicx,xy, color}
\usepackage{mathrsfs}
\usepackage{tikz,stackrel}
\usepackage[all]{xy}
\usetikzlibrary{calc}
\usetikzlibrary{matrix,arrows,decorations.pathmorphing}
\usetikzlibrary{intersections}
\usetikzlibrary{decorations.markings}
\usetikzlibrary{decorations.text}
\usepackage{tikz-cd}
\usepackage[english]{babel}
\usepackage{stmaryrd}
\usepackage{float}
\usepackage{comment}
\usepackage[colorlinks=true, pdfstartview=FitV, linkcolor=blue, citecolor=blue, urlcolor=blue, breaklinks=true]{hyperref}
\usepackage[algonl,boxed,norelsize,lined]{algorithm2e}
\usepackage{xcolor}
\usepackage{theoremref}
\usepackage{caption}
\usepackage{youngtab}
\usepackage{ragged2e}
\usepackage{relsize}
\usepackage{anyfontsize}
\usepackage[utf8]{inputenc}
\usepackage{bm}
\usepackage{scalerel}
\usepackage{marvosym}
\usepackage[normalem]{ulem}

\DeclareFontFamily{U}{mathx}{\hyphenchar\font45}
\DeclareFontShape{U}{mathx}{m}{n}{
      <5> <6> <7> <8> <9> <10>
      <10.95> <12> <14.4> <17.28> <20.74> <24.88>
      mathx10
      }{}
\DeclareSymbolFont{mathx}{U}{mathx}{m}{n}
\DeclareFontSubstitution{U}{mathx}{m}{n}
\DeclareMathAccent{\widecheck}{0}{mathx}{"71}
\DeclareMathAlphabet{\mathbbm}{U}{bbm}{m}{n}
\usepackage{verbatim}

\makeatother
\theoremstyle{remark}
\numberwithin{equation}{section}

\theoremstyle{definition}
\newtheorem{theorem*}{Theorem}
\newtheorem{definition*}{Definition}
\newtheorem{theorem}{Theorem}[section]
\newtheorem{definition}[theorem]{Definition}
\newtheorem{proposition}[theorem]{Proposition}
\newtheorem{proposition*}{Proposition}
\newtheorem{lemma}[theorem]{Lemma}

\newtheorem{remark}[theorem]{Remark}
\newtheorem{notation}[theorem]{Notation}
\newtheorem{example}[theorem]{Example}
\newtheorem{example*}{Example}

\newtheorem{algo}[theorem]{Algorithm}

\newcommand{\tc}[2]{\textcolor{#1}{#2}}
\newcommand{\lrp}[1]{\left(#1\right)}
\newcommand{\lrb}[1]{\left[#1\right]}
\newcommand{\lrm}[1]{\left|#1\right|}
\newcommand{\lrc}[1]{\left\{#1\right\}}
\newcommand{\lra}[1]{\langle{#1}\rangle}

\newcommand{\N}{\mathbb{N} }
\newcommand{\Q}{\mathbb{Q} }
\newcommand{\R}{\mathbb{R} }
\newcommand{\C}{\mathbb{C} }

\newcommand{\Z}{\mathbb{Z} }

\makeatletter
\newcommand{\oversetcustom}[3][0ex]{%
  \mathrel{\mathop{#3}\limits^{
    \vbox to#1{\kern-0.5\ex@
    \hbox{$\scriptstyle#2$}\vss}}}}
\makeatother

\newcommand{\gv}{\mathbf{g} }

\newcommand{\vb}[1]{\mathbf{#1}}
\newcommand{\cA}{\mathcal{A} }
\newcommand{\cAp}{\mathcal{A}_{\mathrm{prin}}}

\newcommand{\prin}{{\mathrm{prin}} }
\newcommand{\cX}{\mathcal{X} }

\newcommand{\trop}{\mathrm{trop} }
\newcommand{\tf}{\vartheta }

\newcommand{\eq}[2]{\begin{equation}\label{#2} \begin{split} #1  \end{split} \end{equation}}
\newcommand{\eqn}[1]{\begin{equation*} \begin{split} #1 \end{split} \end{equation*}}

\newcommand{\Nuf}{N_{\text{uf}}}
\newcommand{\Nuft}{\widetilde{N}_{\text{uf}}}
\newcommand{\Iuf}{I_{\text{uf}}}
\newcommand{\Iuft}{\widetilde{I}_{\text{uf}}}

\newcommand{\wN}{\widetilde{N}}
\newcommand{\wM}{\widetilde{M}}
\newcommand{\wpp}{\widetilde{p}}
\newcommand{\bL}{\vb{L}}

\newcommand{\sk}[2]{\lrc{ #1 , #2 } }

\newcommand{\sgn}{\operatorname*{sgn}}

\newcommand{\wall}{\mathfrak{d}}
\newcommand{\mono}{\mathfrak{m}}

\newcommand{\frakp}{\mathfrak{p}}
\newcommand{\ray}{\mathfrak{r}}

\newcommand{\seed}{\textbf{s}}

\DeclareMathOperator{\Supp}{Supp}
\DeclareMathOperator{\Par}{\overrightarrow{\Supp}}
\DeclareMathOperator{\Sing}{Sing}

\DeclareMathOperator{\Hom}{Hom}

\DeclareMathOperator{\Spec}{Spec}

\newcommand{\scat}{\mathfrak{D}}

\newcommand{\rank}{\operatorname{rank}}

\newcommand{\red}[1]{{ #1}}
\newcommand{\change}[1]{{ #1}}

\title{Compactifications of cluster varieties and convexity}
\author{Man-Wai Cheung, Timothy Magee and Alfredo N\'ajera Ch\'avez}
\date{}

\begin{document}

\maketitle

\begin{abstract}
    
    In \cite{GHKK}, Gross-Hacking-Keel-Kontsevich discuss compactifications of cluster varieties from {\it{positive subsets}} in the real tropicalization of the mirror. To be more precise, let $\scat$ be the scattering diagram of a cluster variety $V$ (of either type-- $\cA$ or $\cX$), and let $S$ be a closed subset of $\lrp{V^\vee}^\trop\lrp{\R}$-- the ambient space of $\scat$.
     The set $S$ is positive if the theta functions corresponding to the integral points of $S$ and its $\N$-dilations define an $\N$-graded subalgebra of $\Gamma(V, \mathcal{O}_V)\change{[x]}$. In particular, a positive set $S$ defines a compactification of $V$ through a Proj construction applied to the corresponding $\N$-graded algebra.
    In this paper we give a natural convexity notion for subsets of $\lrp{V^\vee}^\trop\lrp{\R}$, called {\it{broken line convexity}}, and show that a set is positive if and only if it is broken line convex.
    The combinatorial criterion of broken line convexity provides  a tractable way  to  construct positive subsets of $\lrp{V^\vee}^\trop\lrp{\R}$, or to check positivity of a given subset.
\end{abstract}

\section{Introduction}

Cluster algebras were introduced by Fomin and Zelevinsky in \cite{FZ_I} with the initial aim of providing an algebraic/combinatorial approach to the subjects of total positivity in semisimple algebraic groups and canonical bases for quantum groups. Shortly after, the connection between cluster algebras and algebraic, hyperbolic and Poisson geometry started to be developed by Gekhtman, Shapiro and Vainshtein in \cite{GSV05,GSV10} and by Fock and Goncharov in \cite{FG_Moduli,FG_Quantization}. This geometric perspective gave rise to the notion of a cluster variety and cluster duality for cluster varieties. 
In this article we treat cluster varieties from view points of birational and toric geometry. This perspective was developed by Gross, Hacking and Keel in \cite{GHK_birational} and, together with Kontsevich, remarkably exploited in \cite{GHKK}.

From the point of view of birational geometry, cluster varieties and their (partial) minimal models generalize the notion of toric varieties (\cite{GHK_birational},\cite{BFMN18}).
Many concepts from the toric world have natural analogues in the cluster world.
For instance, recall that torus invariant irreducible divisors of a toric variety correspond to cocharacters of the defining torus.
For a cluster variety $V$, we replace the cocharacter lattice $\bL$ in toric geometry with the {\it{integral tropicalization}} of $V$, denoted $V^\trop(\Z)$.
As a set, $ V^\trop(\Z) = \bL $. In fact, by definition $V$ is endowed with an open cover of algebraic tori\footnote{See Section~\ref{sec:variety} for a more precise description.} (usually an infinite number) and every torus in the cover gives an identification of $V^\trop(\Z)$ with $\bL$. However, as the identifications differ, $V^\trop(\Z)$ does not have a natural group structure.
Morally, it encodes the logarithmic geometry of $V$.
Every cluster variety is endowed with a canonical volume form $\Omega$, 
and $V^\trop(\Z)$ is the set of divisors on varieties birational to $V$ along which $\Omega$ has a pole, up to a certain equivalence.\footnote{More precisely, it is a set of {\it{divisorial discrete valuations}}. The equivalence identifies divisors of different birational models if they give the same valuation on the field of rational functions of $V$, such as divisors associated to a ray shared by different fans in the toric case.}
Taking coefficients for divisors in $\R$ rather than $\Z$ gives a closely related space $V^\trop(\R)$, which generalizes the vector space $\bL\otimes \R$ from toric geometry.
Unlike $\bL\otimes \R$, the space $V^\trop(\R)$ has only a {\emph{piecewise}} linear structure\footnote{More precisely, $V^\trop(\R)$ possesses the structure of a singular integral affine manifold. As a topological space it is simply homeomorphic to $\R^{\text{rank}(\bL)} $.}, with each domain of linearity invariant under $\R_{>0}$ scaling.\footnote{This scaling invariance of domains of linearity is a very special feature of cluster varieties, and it is crucial to the arguments in this paper.}
Next, the cocharacter lattice of a torus is a natural basis for the regular functions of the dual torus. 
Cluster varieties also have duals, and a pair of dual cluster varieties are built out of dual tori.
Analogously to the toric case, the integral tropical points of a cluster variety $V$ (satisfying certain technical assumptions) give a basis for the regular functions of the dual $V^\vee$.
These basis elements-- the {\it{theta functions}}-- are defined using the machinery of {\it{scattering diagrams}} and {\it{broken lines}} described for cluster varieties in \cite{GHKK}.
A cluster scattering diagram for $V^\vee$ is a collection of codimension 1 cones ({\it{walls}}) in $V^\trop\lrp{\R}$, each decorated with a scattering function. The walls divide $V^\trop\lrp{\R}$ into chambers which are precisely the domains of linearity of $V^\trop\lrp{\R}$.
Broken lines are piecewise linear rays drawn in this piecewise linear manifold which are allowed to bend in a prescribed fashion at the walls.
Each linear segment is decorated with a Laurent monomial in $\C[\bL]$, and allowed bendings at a wall are determined by the scattering function and the decorating monomial.
The theta functions are expressed in local coordinates by summing decorating monomials over broken lines with a given initial direction and end point.
We will review these notions in greater detail in Section \ref{sec:basic}.
For now we simply provide an example on the left side of Figure~\ref{fig:A2Scat} to help orient the reader.

%%%%%%%%%%%%%%%%%%%%%%%%%%%%%%%%%   A2 Scattering Diagram

\noindent
\begin{center}
\begin{minipage}{.85\linewidth}
\captionsetup{type=figure}
\begin{center}
\begin{tikzpicture}[scale=.85]

    \def\x{1}
    \def\d{1}
    \def\l{3}
    \def\op{0}

    \path (-\l,0) coordinate (3) --++ (\l,0) coordinate (0) --++ (\l,0) coordinate (1);
    \path (0,\l) coordinate (2) --++ (0,-2*\l) coordinate (4) --++ (\l,0) coordinate (5);
    \path (5) --++ (0,2*\l) coordinate (tr) --++ (-2*\l,0) coordinate (tl) --++ (0,-2*\l) coordinate (bl);

    \draw[thick, ->] (3) -- (1);
    \draw[thick, ->] (2) -- (4);
    \draw[thick, ->] (0) -- (5);

    \node at (.9,\l) {$1+z^{(0,1)}$};
    \node at (-2.3,-.35) {$1+z^{(-1,0)}$};
    \node at (2.75,-1.5) {$1+z^{(-1,1)}$};

    \node [circle, fill, inner sep = 1.5pt, color = orange](Q) at (\x+.25,\d) {};

    \draw [thick, color=orange] (-\l,\d)--(Q);
    \draw [thick, color=orange] (-\l,\x+\d+.25)--(0,\x+\d+.25)--(Q);

    \path (Q)--(0,\d) node [pos=.4, sloped, below] {\textcolor{black}{$z^{(-1,0)}$}};
    \path (0,\x+\d+.25)--(Q) node [pos=.5, sloped, above] {\textcolor{black}{$z^{(-1,1)}$}};

    \path (-\l,\d)--(0,\d) node [pos=.35, sloped, below] {\textcolor{black}{$z^{(-1,0)}$}};
    \path (-\l,\x+\d+.25)--(0,\x+\d+.25) node [pos=.35, sloped, above] {\textcolor{black}{$z^{(-1,0)}$}};
    
\begin{scope}[xshift=7.5cm]

    \def\d{2}
    \def\l{3}
    \def\op{.25}

    \path (-\l,0) coordinate (3) --++ (\l,0) coordinate (0) --++ (\l,0) coordinate (1);
    \path (0,\l) coordinate (2) --++ (0,-2*\l) coordinate (4) --++ (\l,0) coordinate (5);
    \path (5) --++ (0,2*\l) coordinate (tr) --++ (-2*\l,0) coordinate (tl) --++ (0,-2*\l) coordinate (bl);

    \draw[thick, ->] (3) -- (1);
    \draw[thick, ->] (2) -- (4);
    \draw[thick, ->] (0) -- (5);

    \node [circle, fill, inner sep = 1.5pt, color = blue] (v1) at (\d,0) {};
    \node [circle, fill, inner sep = 1.5pt, color = blue] (v2) at (0,\d) {};
    \node [circle, fill, inner sep = 1.5pt, color = blue] (v3) at (-\d,0) {};
    \node [circle, fill, inner sep = 1.5pt, color = blue] (v4) at (0,-\d) {};
    \node [circle, fill, inner sep = 1.5pt, color = blue] (v5) at (\d,-\d) {};
    \node [circle, fill, inner sep = 1.5pt, color = blue] (v0) at (0,0) {};

    \draw[blue, thick, fill= blue, fill opacity=\op] (v1.center) -- (v2.center) -- (v3.center) -- (v4.center) -- (v5.center) -- cycle;

\end{scope}
\end{tikzpicture}

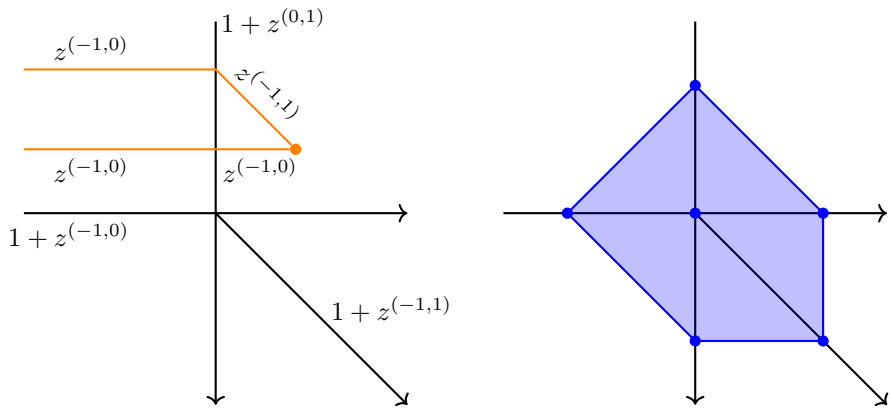
\captionof{figure}{\label{fig:A2Scat} On the left: Example of a scattering diagram and broken lines contributing to the theta function ${\tf_{\lrp{-1,0}}= z^{(-1,1)}+z^{(-1,0)}}$.
In this example $V$ is the complement of an anticanonical cycle of 5 lines in the del Pezzo surface of degree 5.
The scattering diagram contains 3 walls-- the coordinate axes and the ray in direction $(1,-1)$.
Each of the 5 chambers is a domain of linearity of $V^\trop(\R)$.
On the right: The set pictured and its dilations give the anticanonical section ring of the smooth del Pezzo surface of degree 5. (This is an instance where $V\cong V^\vee$.) See \cite[Example~8.31]{GHKK}.
} 

\end{center}
\end{minipage}
\end{center}

%%%%%%%%%%%%%%%%%%%%%%%%%%%%%%%%%

An important insight of \cite{GHKK} is that certain subsets of $V^\trop\lrp{\R}$, with integral points corresponding to theta functions on the dual cluster variety $V^\vee$, can be used to (partially) compactify $V^\vee$ in much the same way that convex polyhedra, with integral points corresponding to characters of a torus, can be used to define a toric variety.
To describe this, we need to discuss multiplication of theta functions.
Since theta functions form a basis for regular functions on $V^\vee$, and regular functions form a ring,
we must be able to decompose a product of theta functions as a sum of theta functions.
That is, there are structure constants $\alpha(p,q,r)$ such that
\eqn{\tf_p \cdot \tf_q = \sum_{r} \alpha(p,q,r) \tf_r.}
A combinatorial description of these structure constants in terms of counts of broken lines appears in \cite[Definition-Lemma~6.2]{GHKK}, and we will review it in Section~\ref{sec:multiplication}.
Motivated by the polytope construction of projective toric varieties, we would like to define graded rings in terms of dilations of a set $S \subset V^\trop(\R)$. 
Namely, for $a,b \in \N$ and integral points $p \in a S$ and $q \in b  S$, we would like the structure constant $\alpha(p,q,r)$ to be non-zero only when $r$ is in $(a+b)  S$. 
Such sets are called {\it{positive}} (\cite[Definition~8.6]{GHKK}). We state the definition more precisely in \thref{def:positive_set}.
A complete discussion of how positive sets in $V^\trop(\R)$ give (partial) compactifications of $V^\vee$ via a Proj construction (relative Proj construction if $S$ is unbounded) appears in \cite[Section~8.5]{GHKK}.
An example of a positive set is provided on the right side of Figure~\ref{fig:A2Scat}.

Observe that if $V$ is just an algebraic torus $\bL\otimes \C^*$,
a closed subset of $V^\trop(\R) = \bL \otimes \R$ is positive if and only if it is convex.
More generally, when working with cluster varieties we argue that broken lines provide a natural generalization of convexity:

\begin{definition}\thlabel{def:blc_intro}
A closed subset $S$ of $V^\trop(\R)$ is {\it{broken line convex}} 
if for every pair of rational points\footnote{We explain the choice of rational points in \thref{rem:Qpoints}.} $s_1$, $s_2$ in $S$,
every segment of a broken line with endpoints $s_1$ and $s_2$ is entirely contained in $S$.  
\end{definition}

Observe that positivity is equivalent to convexity in the toric case.
Since broken line convexity in $V^\trop(\R)$ directly generalizes convexity in $\bL \otimes\R$,
it would be rather nice to find that for cluster varieties positivity is equivalent to broken line convexity.
This is precisely the main result of the paper:
\begin{theorem}\thlabel{MainThmInt}
Let $V$ be a cluster variety and $S$ a closed subset of $V^\trop\lrp{\R}$.
Then $S$ is positive if and only if $S$ is broken line convex.
In particular $S$ defines a partial compactification of the dual $V^\vee$, which is an actual compactification if $S$ is bounded.\footnote{In the final sentence of \thref{MainThmInt}, there are certain assumptions on $V$ and a certain amount of ambiguity in writing ``$V^\vee$'' that we will clarify in Section~\ref{sec:variety}. \change{Furthermore, we would like to remark that we actually prove a more general result that also holds for certain quotients of $\cA$-varieties and fibers of $\cX$-varieties. See \thref{MainThm} for the precise statement.}}
\end{theorem}

\thref{MainThmInt} provides a tractable way to construct positive subsets of $V^{\trop}(\R)$, or to check positivity of a given subset. 
We will need to review some properties of scattering diagrams for cluster varieties before we can adequately explain this process,
so we defer the explanation to \thref{rk:compute}.
However, the manageable approach to positive subsets $S$ of $V^\trop(\R)$ afforded by \thref{MainThmInt} leads to equally manageable (partial if $S$ is unbounded) compactifications of $V^\vee$,
without any reference to ``frozen variables''-- certain distinguished theta functions that are simply characters when restricted to any torus in the atlas used to define $V^\vee$'s cluster structure.
To emphasize this point, most familiar partial compactifications of cluster varieties (decorated flag varieties, affine cones over Grassmannians, etc.) arise by allowing certain frozen variables of the cluster variety to vanish, or by taking a quotient of such a construction (flag varieties, Grassmannians, etc.).
Using \thref{MainThmInt}, it is easy to give compactifications of cluster varieties without any reference to frozen variables. 
It would be interesting to study which projective varieties can be described as compactified cluster varieties in this way.

The proof of \thref{MainThmInt} has some interesting features that might be of independent interest. It consist of two parts. First, we introduce an algorithmic construction whose input is a pair of broken lines $(\gamma^{(1)}, \gamma^{(2)})$ balanced at $r$ and a pair of positive integers $a,b\in \N$; the output is a broken line segment whose endpoints are $\frac{I( \gamma^{(1)})}{a}$ and $\frac{I( \gamma^{(2)})}{b}$ passing through $\frac{r}{a+b}$. We use this to show that every broken line convex set is positive. Second, we introduce an algorithmic construction that produces a balanced pair of broken lines out of a broken line segment and a rational point in its relative interior. We use this to show that every positive set is broken line convex. Since the balancing condition is used to define the structure constants of the theta basis it is interesting to study further implications of these constructions. The reader is kindly invited to consult \thref{def:genbroken}, \thref{def:balanced} and the body of the text for the terminology used in this paragraph.

\red{Next, we would like to remark that a result of Keel and Yu gives a rich source of examples of broken line convex polytopes.  Gross-Hacking-Keel-Kontsevich made the following conjecture (\cite[Conjecture~8.11]{GHKK}):
\begin{center}
    {\it{If $0 \neq f$ is a regular function on a log Calabi–Yau manifold $V$
with maximal boundary, then $f^\trop : V^\trop (\R) \to \R$ is min-convex. }}    
\end{center}
See \cite[Definition~8.2]{GHKK} for the definition of {\it{min-convex}}.  
By \cite[Lemmas~8.4~and~8.9]{GHKK}, if $f^\trop$ is min-convex, the set 
\eqn{ \lrc{\left. x \in V^\trop(\R) \right| f^\trop (x) \geq -1 } }
is positive.
Keel and Yu prove this conjecture in \cite{KeelYu}.\footnote{There is a slight terminology clash between \cite{KeelYu} and the current paper. Keel and Yu refer to {\it{min-convexity}} from \cite{GHKK} as {\it{broken line convexity}}, and they refer to  \cite[Conjecture~8.11]{GHKK} as the {\it{broken line convexity conjecture}}. See \thref{def:blc} for what we call as {\it{broken line convexity}}.}}

Finally, we would like to indicate further motivation for introducing \thref{def:blc_intro} and investigating its connection to positivity.
First, we came up with the notion while pursuing a cluster-varieties-generalization of Batyrev duality for toric Fano varieties.
\thref{def:blc_intro} allows us to give what we call a {\it{theta function analogue of a Newton polytope}}, which will play a key role in the construction. 
This will be explored further in \cite{BatFin}.
Next, while attempting to finish the forthcoming \cite{Grass},
we noticed a potential connection to Newton-Okounkov bodies for cluster varieties.
In particular, Rietsch-Williams associate Newton-Okounkov bodies for a distinguished anti-canonical divisor of the Grassmannian to certain combinatorial objects known as {\it{plabic graphs}}.  See \cite{RW}.
They show that for certain nice plabic graphs, the associated Newton-Okounkov body has a rather simple description-- it is just the convex hull of the valuations of Pl\"ucker coordinates.
However, this simple description fails to describe the Newton-Okounkov body for other plabic graphs.
An immediate consequence of \thref{MainThmInt} is that this issue has an easy fix.
If we just replace the {\it{convex hull of the valuations of Pl\"ucker coordinates}} with what we will call the {\it{broken line convex hull of the valuations of Pl\"ucker coordinates}}, we obtain a simple description of the Newton-Okounkov body that applies to all plabic graphs.
This will be explored further in \cite{Grass}, and we believe broken line convexity provides a natural intrinsic version of Newton-Okounkov bodies for cluster varieties more generally.

\subsection*{Structure of the paper}
In Section~\ref{sec:variety}, we review some basic definitions.  We state precisely our definition of {\it{cluster variety}} and describe the relevant partial compactifications of cluster varieties.
Next, in Section~\ref{sec:basic} we review the notions of scattering diagrams, broken lines, and theta functions.
Original material begins in Section~\ref{sec:blc}, where we describe non-generic broken lines and state carefully what we mean by {\it{broken line convexity}}.

Our construction begins in Section~\ref{sec:key_lemma}. 
Given a pair of broken lines $\lrp{\gamma^{(1)},\gamma^{(2)}}$ balanced at a point $r$ and a pair of positive integers $(a,b)$, we construct a broken line segment $\tilde{\gamma}$ with endpoints $\frac{I(\gamma^{(1)})}{a}$ and $\frac{I(\gamma^{(2)})}{b}$ that passes though $\frac{r}{a+b}$.

In Section~\ref{sec:in_alg}, we perform a reverse construction: given a broken line segment with rational endpoints $\tilde{p}$ and $\tilde{q}$ and a rational point $\tilde{r}$ in the interior of this segment, we construct a pair of broken lines $\lrp{\gamma^{(1)},\gamma^{(2)}}$ balanced at a point $r$, where for some pair of positive integers $(a,b)$ we have  $I(\gamma^{(1)}) = a \tilde{p}$, $I(\gamma^{(2)}) = b \tilde{q}$, and $r= (a+b) \tilde{r}$.

Finally, in Section \ref{sec:MainThm} we give our main theorem-- a set $S$ is positive if and only if it is broken line convex. We finish with an example related to the $\cA$ cluster variety of type $G_2$. 

\subsubsection*{Acknowledgements} 
The authors would like to thank Ben Morley for pointing us in the direction of jagged paths, and Mark Gross, Travis Mandel for helpful discussions.
\red{The authors would also like to thank the anonymous referees for edits and suggestions they have provided.}
M. Cheung is supported by NSF grant DMS-1854512. 
T. Magee \change{was supported by EPSRC grant EP/P021913/1 during the preparation of the first version of this article, and by Royal Society grant RGF\textbackslash EA\textbackslash 181078 while preparing the final version}. A. N\'ajera Ch\'avez is supported by CONACyT grant CB2016 no. 284621. 

\section{Background} 

\subsection{Cluster varieties} \label{sec:variety}
The main result of this paper is a simple combinatorial criterion that decides precisely which closed subsets $S$ of the real tropicalization of a cluster variety $V$ define graded rings with a filtered vector space basis of theta functions.
When $V$ is sufficiently close to affine (various sufficient conditions are given in \cite{GHKK}), 
theta functions form a basis of regular functions on the dual $V^\vee$, and $S$ satisfying our broken line convex criterion give partial compactifications of $V^\vee$, up to a certain ambiguity in codimension 2 loci that we are about to describe.
In the literature there are a few slightly different notions of what {\it{cluster variety}} means.
As hinted in the introduction, all definitions involve in some way a union of algebraic tori glued via birational mutation maps,
and in all cases the resulting scheme is {\it{log Calabi-Yau}}.  See \cite[Definition~1.1]{GHK_birational} for the definition of a log Calabi-Yau scheme.
Gross, Hacking and Keel carefully describe how to construct in this way the pair of schemes they denote by $\cA$ and $\cX$ in \cite[Section~2]{GHK_birational}. 
This pair of schemes is defined in terms of fixed data $ \Gamma$ and an initial seed ${\bf s}$.
We will review the fixed data $\Gamma$ and initial seed $\bf{s}$ in the following subsection~\ref{sec:basic}; 
the reader is encouraged to consult \cite[Section~2]{GHK_birational} for a description of $\cA$ and $\cX$.
Our definition of {\it{cluster variety}} is based upon this and \cite[Lemma~1.4.(ii)]{GHK_birational}, reproduced below.

\begin{lemma}[{\cite[Lemma~1.4.(ii)]{GHK_birational}}]
    Let $\mu: U  \dashrightarrow V$ be a birational map between smooth varieties which is an isomorphism outside codimension two subsets of the domain and range. Then $U$ is log Calabi-Yau if and only if $V$ is.
\end{lemma}

With this in mind we make the following definition.

\begin{definition}\thlabel{cluster}
A smooth scheme $V$ is a {\it{cluster variety of type $\cA$}} if there is a birational map $ \mu: V \dashrightarrow \cA$ which is an isomorphism outside codimension two subsets of the domain and range, 
where $\cA$ is a union of tori glued by $\cA$-mutation as in \cite[Section~2]{GHK_birational}.
We define a {\it{cluster variety of type $\cX$}} analogously. 
\end{definition}

Observe that $\cAp$ is itself constructed like $\cA$, so it is included in the above definition.
\change{Next, Gross-Hacking-Keel also discuss schemes arising as a quotient of $\cA$ by a natural torus action and, dually, schemes arising as fibers in a natural fibration of $\cX$ over a torus.
We will review this construction in subsection~\ref{sec:RedDim}.
Although we mainly focus on cluster varieties, it is worth pointing out that the results of this paper apply to schemes isomorphic up to codimension two to these quotients of $\cA$ and fibers of $\cX$ (see Section \ref{sec:RedDim} below) as well.}

Finally, we indicate the sorts of compactifications we are interested in, again following \cite{GHK_birational}.

\begin{definition}[{\cite[Page~142]{GHK_birational}}]\thlabel{pmm}
A partial compactification $V\subset Y$ is a {\it{partial minimal model}} for the cluster variety $V$ if the volume form $\Omega$ on $V$ has a pole on every irreducible divisorial component of the boundary $Y\setminus V$.  
\end{definition}

We include this information only to orient the reader.
The main result \thref{MainThm} deals with subsets of a tropical space, not explicitly with compactifications of cluster varieties.
Its interest stems from these compactifications though, and we include \thref{cluster} and \thref{pmm} to clarify the context of this result.
As the main result is logically independent from these definitions, we will not dwell on this further.
We invite the interested reader to see \cite[Section~8.5]{GHKK} for a complete description of the construction of partial minimal models for cluster varieties from positive sets.

\subsection{Scattering diagrams and theta functions} \label{sec:basic}
In this section we review the formulation of cluster scattering diagrams, following \cite{GHKK}.
This section is included to make the paper reasonably self-contained, but greater detail is provided in the original reference and the interested reader is encouraged to read \cite[Section~1]{GHKK} for a more complete treatment.
After this review, we give a slightly modified definition of broken lines. 
Finally, we discuss theta functions and their multiplication, and we motivate constructions which will appear in the coming sections.

Let us begin by defining the \emph{fixed data} $\Gamma$ for a pair of cluster varieties $\lrp{\cA,\cX}$, which consists of the following:
\begin{itemize}
\setlength\itemsep{0em}
    \item a finite set $I$ of \emph{directions} with a subset of \emph{unfrozen directions} $\Iuf$; 
    \item a lattice $N$ of rank $ |I|$;
    \item a saturated sublattice $ \Nuf \subseteq N$ of rank $|I_{\text{uf}}|$;
    \item a skew-symmetric bilinear form $\lbrace \cdot , \cdot \rbrace: N\times N \to \Q$;
    \item a sublattice $N^{\circ}\subseteq N$ of finite index satisfying
    $\lbrace N_{\text{uf}}, N^{\circ}\rbrace \subset \Z  $ and $  \lbrace N,N_{\text{uf}}\cap N^{\circ}\rbrace \subset \Z$;
    \item a tuple of positive integers $(d_i:i\in I)$ with greatest common divisor 1;
    \item $M=\Hom(N,\Z)$ and $M^{\circ}=\Hom(N^{\circ},\Z)$.
\end{itemize}

A \emph{seed} is a tuple $\seed =(e_i\in N :i\in I)$ such that $\{e_i:i\in I\}$ is a basis of $N$, 
$\{ e_i : i\in \Iuf\} $ is a basis of $\Nuf$, 
and $\{d_i e_i : i\in I\}$ is a basis $N^{\circ}$.
We denote the dual basis by $\{e_1^*, \dots , e_n^*\}$, and we set $f_i = d_i^{-1}e_i^*$. 

The skew form induces the following map of lattices:
\eqn{p_1^*:\Nuf&\to M^\circ\\
        n&\mapsto \sk{n}{\ \cdot\ } .}
\red{Let $\sigma\subset M^\circ_\R$ be a full dimensional strictly convex rational polyhedral cone containing $\lrc{p_1^*(e_i): i\in \Iuf}$. 
Denote the monoid of integral points of $\sigma$ by $P$,} and its maximal ideal of non-invertible elements by $J$.
In order to define the scattering diagram, Gross-Hacking-Keel-Kontsevich require $p_1^*$ to be injective.\cite[Section~1.1:~``Injectivity~assumption'']{GHKK}
This assumption allows them to build the scattering diagram through an iterative procedure and provides a notion of convergence.
If $p_1^*$ is injective, then $J = P \setminus \lrc{0}$, but if $p_1^*$ fails to be injective $P$ may have non-zero invertible elements. 
Scattering functions for cluster varieties are elements of the completion $\widehat{\C\lrb{P}}$ of $\C\lrb{P}$ with respect to $J$.
If $P$ has some non-zero invertible element $m$, then a series $1 + \sum_{k=1}^\infty c_k z^{k m}$ is not in $\widehat{\C\lrb{P}}$, and we have no way to describe convergence for this series.
If a scattering function were of that form, it would not be congruent to $1 \mod J$, and the iterative procedure of building a consistent scattering diagram-- where intermediate scattering diagrams are consistent modulo powers of $J$-- would break down.

When the injectivity assumption fails, the situation can be fixed by introducing {\it{principal coefficients}}.
This replaces $\cA$ with a higher dimensional space $\cAp$ whose unfrozen directions are naturally identified with $\Iuf$.
By construction, the injectivity assumption is satisfied for $\cAp$, so it is always possible to build a consistent scattering diagram for this space.
Moreover, $\cAp$ is itself a $(\C^*)^n$-family of deformations of $\cA$, with $\cA$ the fiber over 1, while $\cX$ is a quotient of $\cAp$ by a torus action.
So, we can use $\cAp$ to study both $\cA$ and $\cX$; failure of the injectivity assumption really does not present an insurmountable problem.  This is illustrated emphatically in \cite{GHKK}, where some of the main results in cluster theory are established using scattering diagrams for $\cAp$ and in turn drawing conclusions about $\cA$ and $\cX$. 
Given fixed data $\Gamma$ for the pair $\lrp{\cA,\cX}$, 
the fixed data $\Gamma_{\prin}$ encoding $\cAp$ is given by

\begin{itemize}
\setlength\itemsep{0em}
    \item $\wN:= N \oplus M^{\circ}$ with the skew-symmetric bilinear form 
    \[
    \{ (n_1, m_1) , (n_2, m_2) \}  = \{n_1, n_2 \} + \langle n_1, m_2 \rangle - \langle n_2 , m_1 \rangle. 
    \]
    \item $\widetilde{\Nuf}:= \Nuf \oplus 0$; $\wN^{\circ} := N^{\circ} \oplus M$. 
    \item The new index set $\tilde{I}$ is the disjoint union of two copies of $I$. The set of unfrozen indices $\Iuft$ is the original $\Iuf$ viewed as a subset of the first copy of $I$.
\end{itemize}
Given an initial seed $\seed =(e_i\in N :i\in I)$, the corresponding seed for $\cAp$ is \eqn{\widetilde{\seed} =\left( (e_1, 0), \dots, (e_n, 0) , (0,f_1) , \dots, (0, f_n) \right).} 

Fix the seed data $\textbf{s} := ( e_i | i \in I)$ and define 
\[    \wN^{+} := \wN^+_{\textbf{s}} := \left\{ \sum_{i \in\Iuf } a_i (e_i,0) \ \bigg| \  a_i \geq 0, \sum a_i >0 \right\}. \]
Now let $\wpp_1^*: \widetilde{\Nuf} \rightarrow \wM^{\circ}$ be the map  given by $\wpp_1^*(n ) = \{n,\ \cdot\ \}$.
Observe that $\wpp_1^*$ is automatically injective since $\sk{(e_i,0)}{(0,f_j)}= \lra{e_i, f_j}=  \delta_{i,j}$.
Let $\sigma$ be the cone generated by $p_1^*\left( (e_i, 0) \right)$ in $\wM^\circ_{\R}$. 
Let $\widetilde{P} = \sigma \cap  \wM^{\circ}$ be the associated monoid, and $\widetilde{J} := \widetilde{P} \setminus \{ 0\}$. 
Abusing notation slightly, we also write $\widetilde{J}$ for the associated monomial ideal in the monoid ring $\C[\widetilde{P}]$, 
and we denote the completion of $\C[\widetilde{P}]$ with respect to $\widetilde{J}$ by $\widehat{\C[\widetilde{P}]}$.

\begin{definition} \label{def:wall+scat}
	A {\it{wall}} in $\wM_{\R}^\circ$ is a pair $(\wall, f_{\wall})$ where
\begin{itemize}
\setlength\itemsep{0em}
\item $\wall \subseteq \wM_{\R}^\circ$, the {\it{support}} of the wall, is a convex rational polyhedral cone of codimension one, contained in $n^{\perp}$ for some $n \in \wN^+$, and 
\item $f_{\wall} \in \widehat{\C[\widetilde{P}]}$, the {\it{scattering function}}, is of the form $ f_{\wall} = 1+ \sum_{k \geq 1} c_k z^{k\wpp_1^*(n)}$, where $c_k \in \C$. 
\end{itemize}
A wall $(\wall, f_{\wall})$ is called \emph{incoming} if $\wpp_1^* (n) \in \wall$. Otherwise it is called \emph{outgoing}.
We call $-p_1^*(n)$ the \emph{direction} of the wall. 
A \emph{scattering diagram} $\scat$ is a collection of walls such that, for each $k \geq  0$, the set
$
	{\{ (\wall, f_{\wall}) \in \scat\, |\, f_{\wall} \neq 1 \bmod (z^{\wpp_1^* (e_i)})^k \}}
$
 is finite. The \emph{support} of a scattering diagram, $\Supp(\scat)$, is the union of the supports of its walls. 
We write
	\eqn{\Sing (\scat) = {\bigcup_{\wall \in \scat} \partial \wall} \quad \cup {\bigcup_{\substack{\wall_1, \wall_2 \in \scat \\  \dim \wall_1 \cap \wall_2 = n-2}} \wall_1 \cap \wall_2},} 
for the \emph{singular locus} of the scattering diagram.
\end{definition}

Now consider a smooth immersion $\gamma: [0,1] \rightarrow \widetilde{M}_{\R}^\circ \setminus \Sing(\scat)$, with endpoints not contained in the support of $\scat$, that only crosses walls transversely.
For each $k>0$, $\gamma$ will cross only a finite number $s_k$ of walls, and we label them by $\lrc{\left.\lrp{\wall_i,f_{\wall_i} }\right| 1\leq i \leq s_k }$, where $\wall_i \subseteq n_0^\perp$ for some $n_0 \in \widetilde{N}^+ $ and $i<j$ if $\gamma$ crosses $\wall_i$ before $\wall_j$.
Each wall $\wall_i$ determines an automorphism $\frakp_{f_{\wall_i}}$ of $\widehat{\C[\widetilde{P}]}$ by
\eqn{\frakp_{f_{\wall_i}}\lrp{z^m}=  z^m f_{\wall_i}^{\epsilon \langle n_0', m \rangle}}
where $n_0'$ is the generator of the monoid $\R_{\geq 0}\cdot n_0 \cap \widetilde{N}^{\circ}$ and $\epsilon = - \sgn(\langle n_0, \gamma' (t_i) \rangle)$.\footnote{Note that with this choice of $\epsilon$, there is never any need to divide.  In the coming sections, we will often incorporate $\epsilon$ into $n_0$ by not requiring $n_0$ to be in $N^+$.}
By composing these automorphisms, we can define for each $k>0$ the {\it{path ordered product}} $\frakp^k_{\gamma, \scat} = \frakp_{f_{\wall_{s_k}}} \circ \cdots \circ \frakp_{f_{\wall_{1}}}$.
The astute reader may notice a potential pitfall in this definition: we have not said how to order walls that are crossed at the same time-- walls whose supports $\wall_i$ and $\wall_j$ are contained in the same hyperplane $n^\perp$.
However, in this case the associated automorphisms $\frakp_{f_{\wall_i}}$ and $\frakp_{f_{\wall_j}}$ of $\widehat{\C[\widetilde{P}]}$ commute, so there is no issue. See \cite[Section~1.1]{GHKK} for details. We then take
$	\frakp_{\gamma, \scat} = \lim_{k \rightarrow \infty} \frakp ^k_{\gamma, \scat}$.

Now observe that any homotopy of $\gamma$ that leaves its endpoints and intersection points with walls fixed will not affect $\frakp_{\gamma, \scat}$.
Note also that this definition can easily be extended to piecewise smooth paths $\gamma$, provided that the path always crosses a wall if it intersects it.
Two scattering diagram $\scat$ and $\scat'$ are said to be \emph{equivalent} if $\frakp_{\gamma, \scat} = \frakp_{\gamma, \scat'}$ for all paths $\gamma$ for which both are defined.
A scattering diagram is \emph{consistent} if $\frakp_{\gamma, \scat}$ only depends on the endpoints of $\gamma$ for any path $\gamma$ for which $\frakp_{\gamma, \scat}$ is defined.

To define a cluster scattering diagram, Gross-Hacking-Keel-Kontsevich first take
\[
\scat_{in, \seed}^{\cAp} := \lrc{\left.\left( (e_i, 0)^{\perp} , 1+z^{ \wpp_1^*\left( (e_i,0) \right)}\right) \right| \  i \in \Iuf }. 
\]
Observe that all scattering functions in $\scat_{in, \seed}^{\cAp}$ are congruent to $1 \mod \widetilde{J}$, so trivially $\scat_{in, \seed}^{\cAp}$ is consistent modulo $\widetilde{J}$. Gross-Hacking-Keel-Kontsevich then introduce new walls in a systematic iterative process (see \cite[Appendix~C]{GHKK}).
All walls introduced are outgoing, and the $k^{\text{th}}$ scattering diagram $\scat_k$ of the sequence is consistent modulo $\widetilde{J}^k$.
By taking the $k\to\infty$ limit, they obtain a consistent scattering diagram $\scat_{\seed}^{\cAp}$ that contains $\scat_{in, \seed}^{\cAp} $:
\begin{proposition} \cite[Theorem 1.12, Theorem 1.13]{GHKK}
There is a consistent scattering diagram $\scat_{\seed}^{\cAp} $ such that $\scat_{in, \seed}^{\cAp} \subset \scat_{\seed}^{\cAp}$, and $\scat_{\seed}^{\cAp} \setminus \scat_{in, \seed}^{\cAp} $ consists only of outgoing walls. 
Furthermore $\scat_{\seed}^{\cAp}$ is equivalent to a scattering diagram all of whose scattering functions are of the form ${ f_{\wall} = (1+ z^{\wpp_1^*(n)})^c}$, for some $n \in \wN^+$, and $c$ a positive integer. 
\red{The consistent scattering diagram is unique up to equivalence.}
\end{proposition}
The details of this construction can be found in \cite[Section~1.2 and Appendix~C]{GHKK}.

\begin{definition} \thlabel{def:genbroken}

{\sloppy{Let $\scat$ be a scattering diagram, $\mono \in \wM^{\circ} \setminus \{0\}$ and $x_0 \in \wM^\circ_{\R} \setminus \Supp(\scat)$. 
A \emph{generic broken line} with \emph{initial exponent} $\mono$ and \emph{endpoint} $x_0$ is a piecewise linear continuous proper path $\gamma : ( - \infty , 0 ] \rightarrow \wM^\circ_{\mathbb{R}} \setminus \Sing (\scat)$ bending only at walls, with a finite number of domains of linearity $L$ and a monomial $c_L z^{\mono_L} \in \C[\wM^\circ]$ for each of these domains. The path $\gamma$ and the monomials $c_L z^{\mono_L}$ are required to satisfy the following conditions:
\begin{itemize}
\setlength\itemsep{0em}
    \item $\gamma(0) = x_0$.
    \item If $L$ is the unique unbounded domain of linearity of $\gamma$, then $c_L z^{\mono_L} = z^{\mono}$.
    \item For $t$ in a domain of linearity $L$, $\gamma'(t) = -\mono_L$.
    \item Suppose $\gamma$ bends at a time $t$, passing from the domain of linearity $L$ to $ L'$, and set ${\scat_t = \lrc{\left.(\wall, f_{\wall}) \in \scat \right| \gamma (t) \in \scat }}$. Then $c_{L'}z^{\mono_{L'}}$ is a term in $\frakp_{{\gamma}|_{(t-\epsilon,t+\epsilon)},\scat_t} (c_L z^{m_L}) $.
\end{itemize}
Further, we refer to $\mono_L$ as the \emph{slope} or \emph{exponent vector} of $\gamma$ at $L $ and set 
\begin{itemize}
    \item $I(\gamma) = \mono$;
    \item $\text{Mono} (\gamma) = c(\gamma)z^{F(\gamma)}$
to be the monomial $c_L z^{\mono_L}$ attached to the unique domain of linearity $L$ of $\gamma$ having $x_0 $ as an endpoint. 
\end{itemize}}}
\end{definition}

\red{The reader can refer to Figure~\ref{fig:numbering} for a schematic of a broken line and to Figure~\ref{fig:gamma} for a precise example in type $A_2$.}

\begin{remark}
\thlabel{rem:motivation}
We are deviating slightly from the terminology of \cite{GHKK} here.
What we call a {\it{generic broken line}} is simply called a {\it{broken line}} in \cite{GHKK}.
The reason for the different terminology is that we will need to allow endpoints $Q$ that lie in $\Supp\lrp{\scat}$ and paths $\gamma$ that pass through $\Sing\lrp{\scat}$.
We will use these non-generic broken lines in Section~\ref{sec:MainThm} when $\alpha(p,q,r) \neq 0 $, $r$ lies in a wall, and we wish to construct a broken line segment passing through a given multiple of $r$ and having endpoints given multiples of $p$ and $q$.
To indicate the need for such non-generic broken lines, it is easy to see that a finite set of distinct points in an outgoing wall is not positive.
However, since there is no generic broken line between them, \thref{def:genbroken} would vacuously be satisfied for this set if we were to restrict to generic broken lines.   
We give our definition of {\it{non-generic broken lines}} in \thref{def:broken}.
\end{remark}

\begin{remark}
Even though a broken line consists of the piecewise linear ray $\gamma$ together with the collection of monomials $c_L z^{\mono_L}$ attached to the domains of linearity, we systematically abuse language and refer to a broken line simply by $\gamma$. 
\end{remark}

\begin{remark} \label{rk:incoming}
Now that we have the notion of a broken line, for {\red{dimension 2}} scattering diagrams,
the terms \emph{incoming wall} and \emph{outgoing wall} in Definition~\ref{def:wall+scat} do carry a literal interpretation. 
\change{Walls in dimension 2 scattering diagrams are lines passing through the origin or rays emanating from the origin. 
In this remark, we decompose a line as two rays emanating from the origin.
The scattering function $f_{\wall}$ of the wall $\wall$ would be of the form $1+ \sum_{k \geq 1} c_k z^{kv}$ from some $c_k \in \C$ and $v$ such that $\wall \subset \R v$. 
A wall is called incoming if $v \in \wall$. Otherwise it is called outgoing.}

In the definition of broken lines, the direction of travel in a domain of linearity is the opposite of the direction of exponent vector in that domain.
\change{Having decomposed lines into pairs of rays, }if the wall $\wall$ contains the point $v$, the broken lines bending over $\wall$ will bend towards the origin. This justifies the term incoming. 
Similarly, broken lines bending at outgoing walls must bend away from the origin.\footnote{\change{  If we had not decomposed lines into rays, then bends at incoming walls could be toward or away from the origin, but bends at outgoing walls would always be away from the origin.}}
\end{remark}

Broken lines are used to define {\it{theta functions}}.
\begin{definition}\thlabel{def:theta}
Let $\scat$, $\mono$, and $x_0$ be as in Definition~\ref{def:genbroken}. 
We define
\[ \tf_{x_0, \mono} = \sum_{\gamma} \text{Mono} (\gamma), \]
where the sum is over all broken lines $\gamma$ with $I(\gamma)=\mono$ and $\gamma(0)=x_0$.
For $\mono = 0$, for any endpoint $x_0$ we define  $\vartheta_{x_0, 0} =1$.
\end{definition}
The point $x_0$ in \thref{def:theta} picks out a coordinate system, and we would like to interpret $\tf_{x_0,\mono}$ as the expression for a canonical function $\tf_\mono$ in this coordinate system.
For this to make sense, if $y_0$ is another point in $\widetilde{M}^\circ_\R\setminus \Supp(\scat)$ and $\eta$ is a path connecting $x_0$ and $y_0$ and defining a wall-crossing automorphism $\frakp_{\eta,\scat}$,
we need to have $\tf_{y_0,\mono} = \frakp_{\eta,\scat}\lrp{ \tf_{x_0,\mono} }$.
But this is \cite[Theorem~3.5]{GHKK}.
From now on, when we write $\tf_\mono$ (with no indication of endpoint $x_0$), we will mean the canonical function that can be expressed in different coordinate systems by choosing an endpoint $x_0$ and following \thref{def:theta}.

\subsubsection{The \texorpdfstring{$\cX$}{X} case}\label{sec:Xscat}
To describe the scattering diagram for $\cX$, Gross-Hacking-Keel-Kontsevich view $\cX$ as a quotient of $\cAp$.
A choice\footnote{See \cite[Section~2]{GHK_birational} for a discussion of $p^*$ maps.  We simply comment here that $\left.p^*\right|_{\Nuf} = p_1^*$.} of $p^*:N \to M^\circ$ 
\change{induces a short exact sequence}
\[
\begin{tikzcd}
0 \arrow[r] & N \arrow[r,"\widetilde{p}^*"] & \widetilde{M}^\circ= M^{\circ}\oplus N \arrow[r,"\psi^*"] & M^\circ \arrow[r] & 0  
\end{tikzcd}
\]
\change{where $\widetilde{p}^*$ is the map $n\mapsto (p^*(n),n)$ and $\psi^*$ is the map  $(m,n)\mapsto m- p^*(n)$.
This yields an exact sequence of tori:
}
\[
\begin{tikzcd}
1 \arrow[r] & T_{N^\circ} = \Spec\lrp{\C[M^\circ]} \arrow[r,"\psi"] & T_{\widetilde{N}^\circ}=\Spec\lrp{\C[\widetilde{M}^\circ]} \arrow[r,"\widetilde{p}"] & T_M= \Spec\lrp{\C[N]} \arrow[r] & 1  
\end{tikzcd}
\]
\change{Since $p^*$ commutes with mutation, this gives a $T_{N^\circ}$-action on $\cAp = \bigcup_\seed {T_{\widetilde{N}^\circ,\seed}}/\sim$ and realizes $\cX= \bigcup_\seed {T_{M,\seed}}/\sim$ as the quotient of $\cAp$ by this $T_{N^\circ}$-action. }

Functions on $\cX$ are $T_{N^\circ}$-weight $0$ functions on $\cAp$.
Restricting to tori $T_{\widetilde{N}^\circ}$ in $\cAp$,
a character $z^{\lrp{m,n}}$ has weight $m-p^*(n)$ under the $T_{N^\circ}$-action.
As the scattering functions are all series in $z^{\wpp_1^*(n)}=z^{\lrp{p^*(n),n}}$ for $n\in \Nuft$,
wall crossing does not affect $T_{N^\circ}$-weight.
A broken line whose initial decoration monomial has weight $0$ and whose endpoint lies in the weight $0$ slice 
\eqn{ \lrc{ \left.(m,n) \in \widetilde{M}^\circ_\R \right| m=p^*(n)} \subset \lrp{\cAp^\vee}^\trop(\R)\cong \widetilde{M}^\circ_\R } 
must be contained entirely in the weight $0$ slice.
The scattering diagram for $\cX$ (denoted $\scat^\cX_\seed$) is just the weight $0$ slice of the scattering diagram $\scat^{\cAp}_\seed$.

\subsubsection{The \texorpdfstring{$\cA$}{A} case}\label{sec:Ascat}
The situation for $\cA$, or its deformations $\cA_{\vb{t}}$, is a bit more subtle.
Rather than being quotients of $\cAp$ under a $T_{N^\circ}$ action,
these spaces are fibers of the natural projection ${\cAp \to T_M}$.
In this case, on the level of the tropical space of the mirror, we have a natural quotient 
\eqn{\lrp{\cAp^\vee}^\trop(\R) \to \lrp{\cA^\vee}^\trop(\R).}
The image inherits a wall structure from $\scat^{\cAp}_{\seed}$,
with attached functions simply determined by specializing coefficients to $\vb{t}$ (for $\cA= \cA_1$, $\vb{t}= 1$) in the $\scat^{\cAp}_\seed$ scattering functions.
However, the result is not technically a scattering diagram as these attached functions may not lie in $\widehat{\C[P]}$.
Nevertheless, this structure is entirely sufficient for our needs.
Gross-Hacking-Keel-Kontsevich define broken lines in $\lrp{\cA^\vee}^\trop(\R)$ to be images of broken lines in $\lrp{\cAp^\vee}^\trop(\R)$ under the natural projection,
with decoration monomials for broken lines in $\lrp{\cA^\vee}^\trop(\R)$ obtained by applying the derivative of the projection map to decoration monomials for broken lines in $\lrp{\cAp^\vee}^\trop(\R)$.
We will be able to directly recover the equivalence of positivity and broken line convexity for subsets of $\lrp{\cA^\vee}^\trop(\R)$ from the analogous statement for $\cAp$.
For more details on the construction of broken lines for $\cA$ and $\cX$ via broken lines for $\cAp$, please see \cite[Construction~7.11]{GHKK}.

\change{
\subsubsection{Quotients of \texorpdfstring{$\cA$}{A} and fibers of \texorpdfstring{$\cX$}{X}} \label{sec:RedDim}
Cluster $\cA$-varieties are often naturally equipped with a torus action, while $\cX$ varieties often naturally fiber over a torus.
This occurs when $p^*:N \to M^\circ$ is not injective.
Let $K= \ker(p^*)$ and $K^\circ = K\cap N^\circ$.
Now let $H$ be a saturated sublattice of $K$ and $H^\circ = H\cap N^\circ$.
The inclusion $H^\circ \hookrightarrow N^\circ$ obviously induces an action of $T_{H^\circ}$ on $T_{N^\circ}$.
This extends to an action of $T_{H^\circ}$ on $\cA = {\bigcup_\seed T_{N^\circ,\seed}}/ \sim$ as $H^\circ \subset \ker(p^*)$ and $p^*$ commutes with mutation.
Meanwhile, the inclusion $H\hookrightarrow N$ gives a map $$ T_{M}= \Spec(\C[N]) \longrightarrow T_{M/H^\perp}= \Spec(\C[H]).$$
Since $H\subset \ker(p^*)$, $p^*$ commutes with mutation, and $\cX = \bigcup_\seed T_{M,\seed} / \sim$, this extends to a map $ \cX  \longrightarrow T_{M/H^\perp}$.
Up to the subtlety about scattering functions for the $\cA$ ``scattering diagram'' discussed in Section~\ref{sec:Ascat}, we can obtain scattering diagrams and broken lines for $\cA/T_{H^\circ}$ and for the fibers $\cX_\phi$ of $ \cX \longrightarrow T_{M/H^\perp}$ in precisely the same way scattering diagrams and broken lines for $\cX$ and $\cA$ are obtained from those of $\cAp$.

{\bf{The $\cA/T_{H^\circ}$ case:}}
Let $\scat^{\cA_{\vb{t}}}_\seed$ be the ``scattering diagram'' for $\cA_{\vb{t}}$ described above.
Functions on $\cA_{\vb{t}}/T_{H^\circ}$ are $T_{H^\circ}$-weight $0$ functions on $\cA_{\vb{t}}$,
so $\big((\cA_{\vb{t}}/T_{H^\circ})^\vee\big)^\trop(\R)$ is the $T_{H^\circ}$-weight $0$ slice of $(\cA^\vee_{\vb{t}})^\trop(\R)$ (the ambient space of $\scat^{\cA_{\vb{t}}}_\seed$) to obtain $\scat^{\cA_{\vb{t}}/T_{H^\circ}}_\seed$. 
To see that this makes sense, we simply note that the scattering functions are themselves $T_{H^\circ}$-weight 0 functions.
Recall that we may take all scattering functions of $\scat^{\cAp}_\seed$ to be of the form $f_\wall = \lrp{1+ z^{\widetilde{p}_1^*(n)}}^c$ for some $n\in \widetilde{N}^+ \subset \widetilde{N}_{\mathrm{uf}}$ and $c \in \mathbb{N}$.
For $h \in H^\circ$ and $n \in \Nuf$, we have \eqn{\lra{(h,0),\widetilde{p}_1^*((n,0))}_{\widetilde{N} \times \widetilde{M} } &= \lra{h,p_1^*(n)}_{N \times M }\\
&=-\lra{n,p^*(h)}_{N \times M }\\
&=0.} 
As a result, a broken line whose initial decoration monomial has $T_{H^\circ}$-weight $0$ and whose endpoint lies in the $T_{H^\circ}$-weight $0$ slice of $(\cA^\vee_{\vb{t}})^\trop(\R)$
must be contained entirely in the $T_{H^\circ}$-weight $0$ slice.

{\bf{The $\cX_\phi$ case:}}
The ambient space of the $\cX$ scattering diagram is $(\cX^\vee)^\trop(\R)$. 
(Recall that $\cX^\vee$ is the $\cA$ variety associated to the Langlands dual data.)
We freely use \cite[Proposition~A.5]{GHKK} to identify $(N^\vee)^\circ$ with $N$ and $(H^\vee)^\circ$ with $H$.
Then precisely as discussed above, $\cX^\vee$ comes with a $T_H$ action.
The action of $T_H$ on $\cX^\vee$ tropicalizes to an action of $H= T_H^\trop(\Z)$ on $(\cX^\vee)^\trop(\Z)$, extending to an action of $H\otimes \R = T_H^\trop(\R)$ on $(\cX^\vee)^\trop(\R)$.
The scattering diagram $\scat^{\cX}_{\seed}$ may be chosen such that the support of each wall is $H$-invariant.\footnote{This may be shown by induction using the algorithmic construction of cluster scattering diagrams in \cite[Appendix~C]{GHKK}.}
Then the quotient $(\cX_\phi^\vee)^\trop(\R) = (\cX^\vee)^\trop(\R)/ H\otimes \R$ inherits a wall structure $ \scat^{\cX_\phi}_\seed $ from $ \scat^{\cX}_\seed $, with attached scattering functions obtained by specializing $\C[H]$-coefficients to $\phi$ in the $\scat^\cX_\seed$ scattering functions. 
Broken lines in $(\cX_\phi^\vee)^\trop(\R)$ are defined to be images of broken lines in $(\cX^\vee)^\trop(\R)$ under the projection, with the decoration monomials obtained by applying the derivative of the projection map to the decoration monomials of the broken lines in $(\cX^\vee)^\trop(\R)$.

The case in which $H$ has corank 2 in $N$ is of particular interest.
As remarked in the Acknowledgements of \cite{GHK_birational} and explained in Section~5 of the same paper, in this case the two dimensional schemes we have just reviewed are (up to the codimension two ambiguity addressed in \cite[Lemma~1.4.(ii)]{GHK_birational}) exactly the log Calabi-Yau surfaces of \cite{ghk}.
}

\subsubsection{Scattering diagram examples}
We conclude this subsection with additional examples of the scattering diagrams. 
\noindent
\begin{center}
\begin{minipage}{\linewidth}
\captionsetup{type=figure, width=.85\linewidth}
\begin{center}
\begin{tikzpicture}[scale=.85]
    \def\x{3}
    \def\d{2.5}
    \def\l{4.5}
    \def\op{0}
    
\begin{scope}
    
    \clip (-\d,-\l) rectangle (\l,\d);

    \coordinate (0) at (0,0); 

    \draw[->,thick] (-\d,0)--(\l,0) node [pos = .15, sloped, above] {$1+z^{(-2,0)}$};     
    \draw[->,thick] (0,\d)--(0,-\l) node [pos = .03, right] {$1+z^{(0,2)}$};     

    \path[name path=r1] (0)--(2*\x,-\x) node [pos = .56, sloped, above] {$1+z^{(-4,2)}$}; 
    \path[name path=r2] (0)--(3*\x,-2*\x); 
    \path[name path=r3] (0)--(4*\x,-3*\x); 

    \path[name path=r4] (0)--(2*\x,-2*\x); 

    \path[name path=r5] (0)--(3*\x,-4*\x); 
    \path[name path=r6] (0)--(2*\x,-3*\x); 
    \path[name path=r7] (0)--(\x,-2*\x) node [pos = .56, sloped, below] {$1+z^{(-2,4)}$}; 

    \draw (3.6,-3.1) node {$\vdots$};
    \draw (3.15,-3.6) node {$\cdots$};

    \path[name path global=bbox] (current bounding box.south west) rectangle 
    (current bounding box.north east);
    \path [name intersections={of=bbox and r1,by=v1}];
    \path [name intersections={of=bbox and r2,by=v2}];
    \path [name intersections={of=bbox and r3,by=v3}];
    \path [name intersections={of=bbox and r4,by=v4}];
    \path [name intersections={of=bbox and r5,by=v5}];
    \path [name intersections={of=bbox and r6,by=v6}];
    \path [name intersections={of=bbox and r7,by=v7}];

    \draw[->, thick] (0)--(v1); 
    \draw[->, thick] (0)--(v2); 
    \draw[->, thick] (0)--(v3); 
    \draw[->, thick] (0)--(v4); 
    \draw[->, thick] (0)--(v5); 
    \draw[->, thick] (0)--(v6); 
    \draw[->, thick] (0)--(v7); 

\end{scope}

\begin{scope}[xshift=8.2cm]
    
    \clip (-\d,-\l) rectangle (\l,\d);

    \coordinate (0) at (0,0); 

    \draw[->,thick] (-\d,0)--(\l,0) node [pos = .17, sloped, above] {$1+z^{(-2,0,0,1)}$};     
    \draw[->,thick] (0,\d)--(0,-\l) node [pos = .03, right] {$1+z^{(0,2,1,0)}$};     

    \path[name path=r1] (0)--(2*\x,-\x) node [pos = .53, sloped, above] {$1+z^{(-4,2,1,2)}$}; 
    \path[name path=r2] (0)--(3*\x,-2*\x); 
    \path[name path=r3] (0)--(4*\x,-3*\x); 

    \path[name path=r4] (0)--(2*\x,-2*\x); 

    \path[name path=r5] (0)--(3*\x,-4*\x); 
    \path[name path=r6] (0)--(2*\x,-3*\x); 
    \path[name path=r7] (0)--(\x,-2*\x) node [pos = .53, sloped, below] {$1+z^{(-2,4,2,1)}$}; 

    \draw (3.6,-3.1) node {$\vdots$};
    \draw (3.15,-3.6) node {$\cdots$};

    \path[name path global=bbox] (current bounding box.south west) rectangle 
    (current bounding box.north east);
    \path [name intersections={of=bbox and r1,by=v1}];
    \path [name intersections={of=bbox and r2,by=v2}];
    \path [name intersections={of=bbox and r3,by=v3}];
    \path [name intersections={of=bbox and r4,by=v4}];
    \path [name intersections={of=bbox and r5,by=v5}];
    \path [name intersections={of=bbox and r6,by=v6}];
    \path [name intersections={of=bbox and r7,by=v7}];

    \draw[->, thick] (0)--(v1); 
    \draw[->, thick] (0)--(v2); 
    \draw[->, thick] (0)--(v3); 
    \draw[->, thick] (0)--(v4); 
    \draw[->, thick] (0)--(v5); 
    \draw[->, thick] (0)--(v6); 
    \draw[->, thick] (0)--(v7); 

\end{scope}

\end{tikzpicture}

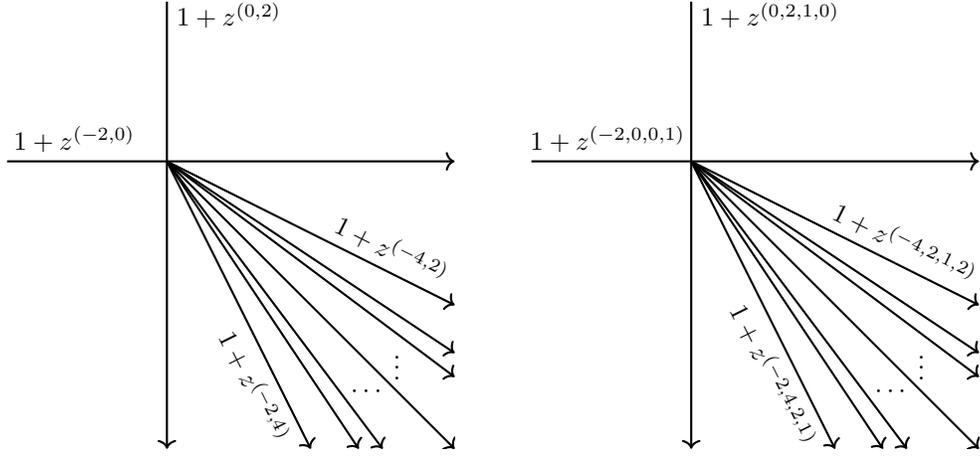
\captionof{figure}{\label{fig:ScatEx} Scattering diagrams associated to the Kronecker quiver ${ 1 \rightrightarrows 2 }$. Note that the injectivity assumption is satisfied in this example. The scattering diagram on the left is for the $\cA$ variety and the scattering diagram on the right is for the $\cX$ variety. Both are two dimensional.  
On the left, there is an outgoing wall through ${(n+1,-n)}$ and ${(n,-n-1)}$ for every ${n\in \Z_{>0}}$.
The corresponding scattering functions are 
${ 1 + z^{(-2 n-2, 2n)}}$ and $ {1 + z^{(-2 n, 2n+2)}}$ respectively.
There is one additional outgoing wall passing through $(1,-1)$ whose scattering function is $\lrp{1-z^{(-2,2)}}^{-2}$.
On the right, the scattering diagram for $\cX$ is the slice of the four dimensional scattering diagram for $\cAp$ where coordinates $(m,n)$ satisfy $m=p^*(n)$.
For each $n\in \Z_{>0}$ there is one outgoing wall though ${(2n+2,-2n,-n,-n-1)}$ and one through ${(2n,-2n-2,-n-1,-n)}$. Their scattering functions are $ {1+z^{(-2n-2,2n,n,n+1)}} $ and $ {1+z^{(-2n,2n+2,n+1,n)}} $ respectively. The last outgoing wall passes through ${(2,-2,-1,-1)}$ and has scattering function $ {\lrp{1+z^{(-2,2,1,1)}}^{-2} }$. See \cite{reineke}, \cite{gp}, and \cite{mandy} for details. }
\end{center}
\end{minipage}
\end{center}

\subsection{Multiplication of theta functions}\label{sec:multiplication}

In the previous subsection we reviewed the construction of theta functions in terms of broken lines.
Broken lines also elegantly encode the multiplication of theta functions. 
That is, 
given a product of arbitrary theta functions $\tf_p \cdot \tf_q$,
we can use broken lines to express the structure constants $\alpha\lrp{p,q,r}$ in the expansion 
\begin{align} \label{eq:product}
    \vartheta_p \cdot \vartheta_q = \sum_r \alpha(p,q,r) \vartheta_r.
\end{align}
We review the construction here.

First, pick a general endpoint $z$ near $r$.
\red{For a broken line $\gamma$, we denote by $I(\gamma)$ and $F(\gamma)$ the initial slope and the final slope respectively of $\gamma$, as in \thref{def:genbroken}.} 
Then define (\cite[Definition-Lemma~6.2]{GHKK})
\eq{
    \alpha_z (p, q, r) := \sum_{\substack{\lrp{\gamma^{(1)}, \gamma^{(2)}} \\ I(\gamma^{(1)})= p,\ I(\gamma^{(2)})= q\\ \gamma^{(1)}(0) = \gamma^{(2)}(0) = z\\
    F(\gamma^{(1)}) + F(\gamma^{(2)}) = r   }}   c(\gamma^{(1)})\ c(\gamma^{(2)}), }{eq:multibrokenline}
where the sum is summing over all pairs of broken lines $\lrp{\gamma^{(1)}, \gamma^{(2)}}$ ending at $z$ with initial slopes $I(\gamma^{(1)}) = p$, $I(\gamma^{(2)}) = q$ and the final slopes satisfying $F(\gamma^{(1)})+F(\gamma^{(2)}) =r$.
Gross-Hacking-Keel-Kontsevich show that for $z$ sufficiently close to $r$, $\alpha_z (p, q, r)$ is independent of $z$ and gives the structure constant $\alpha (p, q, r)$. See \cite[Proposition~6.4]{GHKK}. 
\red{For the moment, we will} say the pairs $\lrp{\gamma^{(1)},\gamma^{(2)}}$ appearing in \eqref{eq:multibrokenline} are {\it{balanced near $r$}}, or if $z=r$ we \red{will} simply say the pair $\lrp{\gamma^{(1)},\gamma^{(2)}}$ is {\it{balanced}}.
\red{We give a more precise definition of balanced broken line in \thref{def:balanced}.}

Now that we have described theta function multiplication, we are prepared to state more carefully the definition of a positive set-- the property required for a set and its dilations to define a graded ring.
\begin{definition}[{\cite[Definition~8.6]{GHKK}}]
\thlabel{def:positive_set}
A closed subset $S$ of $V^\trop(\R)$ is {\it{positive}} if for any non-negative integers $a$, $b$, any $p \in a S(\Z)$, $q \in b S(\Z)$, and any $r$ in $V^\trop(\Z)$ with $\alpha(p,q,r) \neq 0$, we have $r \in \lrp{a+b} S(\Z)$.  
\end{definition}

It will be convenient in the remainder of the text to work with $\cAp$,
where the injectivity assumption is satisfied, and recover results for both $\cA$ and $\cX$ from analogous results for $\cAp$.
\change{In turn, we will then recover results for the quotients of $\cA$ and fibers of $\cX$ discussed in Section~\ref{sec:RedDim} from the analogous results for $\cA$ and $\cX$.}
The following lemmas relate positive sets in $ \lrp{\cAp^\vee}^\trop(\R)$ to positive sets in $ \lrp{\cA^\vee}^\trop(\R)$ and $ \lrp{\cX^\vee}^\trop(\R)$. 
The proofs follow directly from the definitions.

\begin{lemma}\thlabel{lem:Apos}
Let $ \pi:\widetilde{M}^{\circ}_{\R} \to M^{\circ}_{\R}$ be the natural projection.  If $S\subset \widetilde{M}^{\circ}_{\R}$ is positive, then so is $\pi(S)\subset M^{\circ}_{\R}$. Similarly, if  $S\subset M^{\circ}_{\R}$ is positive, then so is $\pi^{-1}(S)\subset \widetilde{M}^{\circ}_{\R}$.
\end{lemma}

\begin{lemma}\thlabel{lem:Xpos}
The subset $(\cX^{\vee})^{\trop}(\R)\subset (\cAp^{\vee})^{\trop}(\R)$ is positive. 
In particular, if the subset ${S\subset (\cAp^{\vee})^{\trop}(\R)}$ is positive then so is $S\cap (\cX^{\vee})^{\trop}(\R)$. 
Moreover, if $S$ is positive as a subset of $(\cX^{\vee})^{\trop}(\R)$, then it is also positive as a subset of $(\cAp^{\vee})^{\trop}(\R)$.
\end{lemma}

\change{
\begin{remark}\thlabel{rem:PosRedDim}
The same reasoning applies for the quotients of $\cA$ and fibers of $\cX$ of Section~\ref{sec:RedDim}.
In detail,
for $\cX_\phi$, let $\pi$ be the quotient map $(\cX^\vee)^\trop(\R) \to (\cX^\vee_\phi)^\trop(\R)$.
If $S\subset (\cX^\vee)^\trop(\R)$ is positive, so is $\pi(S) \subset (\cX^\vee_\phi)^\trop(\R)$.
If $S\subset (\cX^\vee_\phi)^\trop(\R)$ is positive, so is $\pi^{-1}(S) \subset (\cX^\vee)^\trop(\R)$.
Next, for $\cA_{\vb{t}}/T_{H^\circ}$, the subset $\big((\cA_{\vb{t}}/T_{H^\circ})^\vee\big)^\trop(\R) \subset (\cA^\vee_{\vb{t}})^\trop(\R)$ is positive.
If the subset ${S\subset (\cA^{\vee}_{\vb{t}})^{\trop}(\R)}$ is positive then so is $S\cap \big((\cA_{\vb{t}}/T_{H^\circ})^\vee\big)^\trop(\R)$. 
Moreover, if $S$ is positive as a subset of $\big((\cA_{\vb{t}}/T_{H^\circ})^\vee\big)^\trop(\R)$, then it is also positive as a subset of $(\cA^\vee_{\vb{t}})^\trop(\R)$.
\end{remark}
}

The main challenge of relating positivity to broken line convexity is that the theta functions are indexed by initial exponent vectors of broken lines-- so an initial {\emph{direction}}-- while broken line convexity relates to segments connecting specified {\emph{points}}.
However, there is a construction similar to broken lines called {\it{jagged paths}} which, Mark Gross informs us, conceptually predated broken lines.
Unlike broken lines, jagged paths have no unbounded domain of linearity.
Instead, in the jagged path description of theta functions, it is the starting point of the jagged path that indexes the theta function.
Generically, multiplication of theta functions is described in terms of pairs of jagged paths ending at the same point and having final exponent vectors sum to 0.  (Non-generically, this condition on the sum of final exponent vectors may have contributions from a wall as well.)
See \cite[Section~3]{GSTheta} for an introduction to this construction.

This description of the multiplication of theta functions in terms of jagged paths generalizes the following toric description of multiplication of characters.
Let $\bL$ be the character lattice of a torus $T$, and let $p$, $q$ be in $\bL\otimes \Q$.
Then every rational point along the line segment $\overline{pq}$ corresponds to a product of characters of the form $z^{a p} \cdot z^{b q}=z^{a p + b q}$.
See Figure~\ref{fig:WeighedAvg}.

\noindent
\begin{center}
\begin{minipage}{.85\linewidth}
\captionsetup{type=figure}
\begin{center}
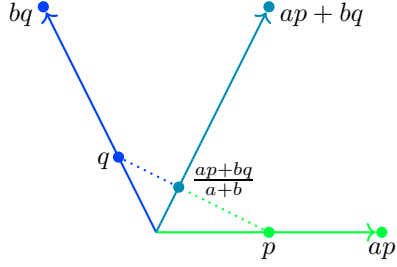

\begin{tikzpicture}

    \def\a{2}
    \def\b{3}
    \def\div{.2}
    \def\px{1.5}
    \def\py{0}
    \def\qx{-.5}
    \def\qy{1}
    \def\op{.25}
    \definecolor{color1}{rgb}{0,1,.25}
    \definecolor{color2}{rgb}{0,.25,1}
    \colorlet{color3}{color1!40!color2}

    \node[circle, fill, inner sep = 1.5pt, color = color1] (p) at (\px,\py) {};
    \node[circle, fill, inner sep = 1.5pt, color = color1] (ap) at (\a*\px,\a*\py) {};
    \node[circle, fill, inner sep = 1.5pt, color = color2] (q) at (\qx,\qy) {};
    \node[circle, fill, inner sep = 1.5pt, color = color2] (bq) at (\b*\qx,\b*\qy) {};
    \node[circle, fill, inner sep = 1.5pt, color = color3] (sum) at (\a*\px+\b*\qx,\a*\py+\b*\qy) {};
    \node[circle, fill, inner sep = 1.5pt, color = color3] (avg) at (\div*\a*\px+\div*\b*\qx,\div*\a*\py+\div*\b*\qy) {};

    \draw[->,thick, color1] (0,0) -- (ap); 
    \draw[->,thick, color2] (0,0) -- (bq); 
    \draw[->,thick, color3] (0,0) -- (sum);
    \draw[dotted,thick,color1] (p)--(avg);
    \draw[dotted,thick,color2] (q)--(avg);
    
    \node at (\px,\py-.25) {$p$};    
    \node at (\a*\px,\a*\py-.25) {$a p$};    
    \node at (\qx-.2,\qy-.05) {$q$};    
    \node at (\b*\qx-.3,\b*\qy-.1) {$b q$};    
    \node at (\a*\px+\b*\qx+.7,\a*\py+\b*\qy-.1) {$a p + b q$};    
    \node at (\div*\a*\px+\div*\b*\qx+.6,\div*\a*\py+\div*\b*\qy+.1) {$\frac{a p + b q}{a+b}$};    

\end{tikzpicture}
\captionof{figure}{\label{fig:WeighedAvg} The rational points along the line segment $\overline{pq}$ can be expressed as weighted averages of $p$ and $q$, given here as $\frac{a p +b q}{a+b}$. Rescaling by $(a+b)$ gives the exponent vector of product $z^{a p}\cdot z^{b q}$.}
\end{center}
\end{minipage}
\end{center}

The way we associate a broken line segment $\tilde{\gamma}$
to a pair of balanced broken lines $\lrp{\gamma^{(1)},\gamma^{(2)}}$ ending at $x_0$ and a pair of positive integers $(a,b)$ in Section~\ref{sec:key_lemma} stems from this construction.
One endpoint of the segment is $\tilde{\gamma}(0)=\frac{I(\gamma^{(1)})}{a}$, the other is $\tilde{\gamma}(T)=\frac{I(\gamma^{(2)})}{b}$, and $\tilde{\gamma}$ passes through $\tilde{x}_0:=\frac{x_0}{a+b}$ at time $\frac{b}{a+b}T$-- the weighted average of $0$ and $T$.
In this sense, $\tilde{x}_0$ is the weighted average of $\frac{I(\gamma^{(1)})}{a}$ and $\frac{I(\gamma^{(2)})}{b}$ along $\tilde{\gamma}$.

\section{Broken line convexity}\label{sec:blc}

Throughout this section, as well as Sections~\ref{sec:key_lemma} and \ref{sec:in_alg},
we take $\scat$ to be the scattering diagram of a cluster variety $V^\vee$ of type $\cA$ for which the injectivity assumption is satisfied, {\it{e.g.}} any $\cAp$ cluster variety.

\subsection{Non-generic broken lines}
\label{sec:non-generic}

To define a good notion of convexity using broken lines we need to consider \emph{non-generic} broken lines which we introduce in \thref{def:broken} below, \emph{c.f.} \thref{rem:motivation}. 
Roughly speaking, these are broken lines that pass through $ \Sing(\scat)$ or have endpoint lying in $\Supp(\scat)$.
In order to control how a non-generic broken line bends we consider a sequences of generic broken lines approximating it. 
To deal with convergence in $V^{\trop}(\R)$ recall that a choice of initial seed gives an identification of $V^{\trop}(\R)$ with $ M^{\circ}_{\R}$. So a choice of seed gives rise to a topology in $V^{\trop}(\R)$: we simply identify $V^{\trop}(\R)$ with $ M^{\circ}_{\R}$ endowed with its usual topology. 
The topology on  $V^{\trop}(\R)$ induced in this way does not depend on the choice of seed as the different identifications are related by piecewise linear bijective maps. 
These maps are open in the topologies given by the different choices of seed and therefore, they are homeomorphisms.

\begin{definition}\thlabel{def:conv_supp}
We say that a sequence of piecewise linear rays $\left( \gamma_k:(-\infty, 0]\to M^{\circ}_{\R} \right)_{k \in \N}$ converges to a piecewise linear ray $\gamma:(-\infty, 0]\to M^{\circ}_{\R}$ if for every $t \in (-\infty, 0]$ the sequence $(\gamma_k(t))_{k \in \N}$ converges to $\gamma(t)$. 
\end{definition}

\begin{notation}
\thlabel{not:numbering}
Let $ \gamma $ be a generic broken line and let $x_1, \dots , x_s $ be its bending points. From now on we assume that these points are numbered ``backwards". That is, if $t_i\in (-\infty, 0)$ is the time at which $\gamma $ bends at $x_i$ then we assume that $t_s < \dots <t_1 < 0$. 
We label the bending walls according to this numbering as well:
for each $x_i$ we have a finite collection of bending walls $\lrc{\wall_{i_1}, \dots ,\wall_{i_{r_i}}}$ contributing\footnote{To clarify, $\gamma$ may cross infinitely many walls, but by definition of broken line it may only bend at finitely many.}
to $\gamma$'s bend at $x_i$.
{\red{(Note that $r_1+\dots +r_s$ is the total number of bendings contributing to $\gamma$.)}}
Observe that a bending wall could appear more than once-- we might have $\wall_{{i_1}_{j_1}}=\wall_{{i_2}_{j_2}}$ for ${i_1}_{j_1} \neq {i_2}_{j_2}$. 
Next, the bounded domains of linearity of $ \gamma$ are precisely the line segments $L_{i}$ with endpoints $x_{i}$ and $ x_{i+1}$, for $ i \in \{ 0, \dots , s-1 \}$ and we let $L_s $ be the unique unbounded domain of linearity of $\gamma $. This numbering also induces a numbering on the monomials attached to the domains of linearity, the slopes, the vectors orthogonal to walls and so on. See Figure~\ref{fig:numbering}. Further, we let $\ray_i$ be the ray spanned by $x_i$. In particular, $\ray_i \subset \wall_{i_j}$.  
\noindent
\begin{minipage}{\linewidth}
\captionsetup{type=figure}
\begin{center}
\begin{tikzpicture}

    \node [inner sep=0] (0) at (0,0) {$\bullet$};

    \draw [name path=walls1] (0)--++(150:5) node [above left] {$\lrc{\wall_{s_j}}$}; 
    \draw [name path=walls1m1] (0)--++(130:5) node [above] {$\lrc{\wall_{{s-1}_j}}$}; 
    \draw [name path=wall11] (0)--++(100:5) node [above] {$\lrc{\wall_{{i+1}_j}}$}; 
    \draw [name path=wall12] (0)--++(75:5) node [above] {$\lrc{\wall_{{i}_j}}$}; 
    \draw [name path=walls2m1] (0)--++(40:5)node [above right] {$\lrc{\wall_{2_j}}$}; 
    \draw [name path=walls2] (0)--++(20:5)node [right] {$\lrc{\wall_{1_j}}$}; 
    \draw[->] (0)--++(75:4.5) node [right] {$\ray_i $ }; 
    \draw[->] (75:2.4)--++(165:.5) node [above] {$n_{0,i}$ }; 

    \path (130:4.8)--(100:4.8) node [midway, sloped] {$\cdots$};
    \path (75:4.8)--(40:4.8) node [midway, sloped] {$\cdots$};

    \node [inner sep=0] (p) at (170:4) {}; 
    \node [inner sep=0] (q) at (5:4.35) {$\bullet$}; 

    \path [name path=seg1] (p) --++(30:2);
    \path [name intersections={of=seg1 and walls1,by=xs1}];

    \path [name path=seg2] (xs1) --++(75:2);
    \path [name intersections={of=seg2 and walls1m1,by=xs1m1}];

    \path [name path=seg3] (xs1m1) --++(40:3.5);
    \path [name intersections={of=seg3 and wall11,by=x11}];

    \coordinate (l1) at ($(xs1m1)!.25!(x11)$);
    \coordinate (l2) at ($(xs1m1)!.75!(x11)$);

    \path [name path=seg4] (x11) --++(-10:3.5);
    \path [name intersections={of=seg4 and wall12,by=x12}];

    \path [name path=seg5] (x12) --++(-45:5);
    \path [name intersections={of=seg5 and walls2m1,by=xs2m1}];

    \coordinate (r1) at ($(x12)!.25!(xs2m1)$);
    \coordinate (r2) at ($(x12)!.75!(xs2m1)$);

    \path [name path=seg6] (xs2m1) --++(-30:3.5);
    \path [name intersections={of=seg6 and walls2,by=xs2}];

    \draw[thick, color=orange]  (p)--(xs1) node [midway, sloped, above] {${\color{black}L_{s}}$}--(xs1m1) node [pos=.6, sloped, above] {${\color{black}L_{s-1}}$}--(l1);
    \path (l1)--(l2) node [midway, sloped, color=orange] {$\cdots$};
    \draw[thick, color=orange]  (l2)--(x11)--(x12)--(r1);
%    \draw[thick, color=magenta] (x0)--(x12)--(r1);
    \path (x11)--(x12) node [midway, sloped, below, color=orange] {${\color{black}L_i}$};
    \path (x11)--(x12) node [midway, sloped, above, color=black] {$c_iz^{\mono_i}$};
    \path (r1)--(r2) node [midway, sloped, color=orange] {$\cdots$};
    \draw[thick, color=orange]  (r2)--(xs2m1)--(xs2) node [pos=.3, sloped, below] {${\color{black}L_{1}}$}--(q) node [pos=.6, left] {${\color{black}L_{0}}$} node[pos=.6, right] {${\color{black}c(\gamma)z^{\mono_0}}$};

    \path (p)--++(-160:.3) ;
    \path (q)--++(-10:.3) node {$x_0$};
    \path (xs1)--++(-90:.4) node {$x_s$};
    \path (xs1m1)--++(-5:.5) node {$x_{s-1}$};
    \path (x11)--++(165:.4) node {$x_{i+1}$};
    \path (x12)--++(10:.55) node {$x_{i}$};
    \path (xs2m1)--++(5:.7) node {$x_2$};
    \path (xs2)--++(-15:.5) node {$x_1$};

    \node [circle, fill, inner sep=1pt] at (xs1) {};
    \node [circle, fill, inner sep=1pt] at (xs1m1) {};
    \node [circle, fill, inner sep=1pt] at (x11) {};
    \node [circle, fill, inner sep=1pt] at (x12) {};
    \node [circle, fill, inner sep=1pt] at (xs2m1) {};
    \node [circle, fill, inner sep=1pt] at (xs2) {};

\end{tikzpicture}


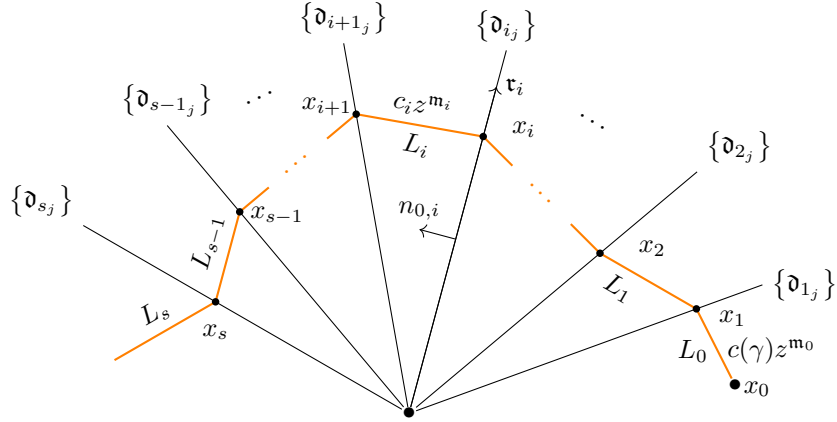
\captionof{figure}{\label{fig:numbering} Labeling the bending points, bending walls, and domains of linearity of $\gamma$.} 
\end{center}
\end{minipage}
\end{notation}

\begin{definition}\thlabel{def:broken}
Let $\mono \in M^\circ \setminus \lrc{0}$ and $x_0 \in M^\circ_\R$. 
A {\it{non-generic broken line}} with initial exponent vector $\mono$ and endpoint $x_0$ is a piecewise linear ray $\gamma: (-\infty , 0] \to M^{\circ}_{\R}$ with $\gamma(0)=x_0$ in $\Supp(\scat)$ or $\text{Im}(\gamma)\cap \Sing(\mathfrak{D})\neq \varnothing$, together with a sequence of generic broken lines $(\gamma_k)_{k \in \N}$ satisfying:
\begin{itemize}
    \item $\lrp{\Supp(\gamma_k)}_{k\in \N}$ converges to $\text{Im}(\gamma)$ (see \thref{def:conv_supp});
    \item $I(\gamma_k) = \mono$ for all $k \in \N$; and
    \item all the broken lines in the sequence bend at the same collection of walls $\lrc{\wall_{i_j}}$ (as in \thref{not:numbering}), in the same order and have the same decorating monomials.
\end{itemize}   
We call $\text{Im}(\gamma)$ the \emph{support} of the non-generic broken line and denote it by $\Supp(\gamma)$. 
The sequence $\lrp{\gamma_k}_{k\in \N}$
induces a sequence of ordered sets of bending points $ \lrp{x_{1;k}, \dots, x_{s;k} }_{k\in \N} $ converging to an ordered set of points $\lrp{x_1, \dots, x_s} $, where $x_i\in \Supp(\gamma) \cap \bigcap_{j=1}^{r_i} \wall_{i_j}$.
We call $x_1, \dots , x_s$ the {\it{bending points}} of the non-generic broken line. 
As in the generic case, for non-zero $x_i$ we let $\ray_i $ be the ray $\R_{\geq 0}x_i$. 
In both the generic and non-generic cases we will say that $\gamma$ bends at $\ray_i $ for each $i \in \{ 1, \dots , s \}$, or bends at the origin if $x_i=0$.
\end{definition}

\red{This notion of {\it{non-generic broken line}} is closely related to Carl-Pumperla-Siebert's {\it{degenerate broken line}} from \cite{CPS}.  We thank the anonymous referee for pointing out this connection.}

\begin{remark}
Notice that we could also describe generic broken lines as limits, {\it{e.g.}} with a constant sequence, and thereby use a single description for all broken lines. 
\end{remark}

We also consider domains of linearity associated to a non-generic broken line $\gamma$. These are the limits of the domains of linearity of the elements in the sequence $(\gamma_k)_{k\in \N}$. Namely, the bounded domains of linearity of $ \gamma $ are the line segments $ L_i$ with endpoints $x_i$ and  $x_{i+1}$ for $i \in \{0, \dots , s-1 \} $. We let $L_s$ be the unique unbounded domain of linearity of $\gamma$. 
Then
\[
\Supp(\gamma)=L_s \cup \dots \cup L_0.
\]
Notice that it might happen that $x_{i}=x_{i+1}$ for some $i$. In this case the domain of linearity $L_i$ is degenerate but we still record it as a domain of linearity. 
In particular, the domains of linearity of $\gamma$ are in bijection with the domains of linearity of any of the $\gamma_k$'s. 
If we had not fixed a sequence of generic broken lines, the non-generic broken line $\gamma$ may not have had a well defined collection of bending walls since it is allowed to intersect $\Sing\lrp{\scat}$.
However, the elements of the sequence $(\gamma_k)_{k \in \N}$ converging to $\gamma $ all have the same collection of bending walls and the same list of attached monomials, say $\lrc{\wall_{i_1}, \dots, \wall_{i_{r_i}} }_{1\leq i \leq s}$  and $ c_0 z^{\mono_0}, \dots ,c_s z^{\mono_s}$, respectively.
When working with non-generic broken lines we will always retain all of this information and attach the monomial $c_i z^{\mono_i}$ to the domain of linearity $L_i$ of $\gamma $.  We define
\[
\Par(\gamma):=\left( (L_s, \mono_s), \dots , (L_0, \mono_0)\right).
\]
In particular,  if a domain of linearity for the elements of $(\gamma_k)_{k \in \N}$ contracts to a point via the limit, {\emph{we still record the exponent vector associated to this degenerate domain of linearity.}} We call $ \mono_i$ the slope or exponent vector of $\gamma$ at $L_i$, even if it is degenerate, and set $I(\gamma)= \mono_s$. Occasionally, we write $\gamma = \left( \Par(\gamma) , c_s, \dots , c_0 \right) $ to codify a parametrization of $\gamma$ and the monomials attached to the domains of linearity of $\gamma$.

\begin{example}
\thlabel{ex:non-generic}
Let $\scat$ be a scattering diagram of type $A_2$ and let $\mono= (0,-1)$ and $x_0=(1,0)$. 
We consider the limit of a sequence of generic broken lines bending first at the wall in $e_2^\perp$, then at the wall in $e_1^\perp$, and ending in the interior of the first quadrant, with endpoints approaching $x_0$.  See Figure~\ref{fig:A2SingBL}. In this case $0$ is the only bending point of $\gamma$.
\noindent
\begin{center}
\begin{minipage}{.85\linewidth}
\captionsetup{type=figure}
\begin{center}
\begin{tikzpicture}

    \def\x{2}
    \def\d{1}
    \def\l{3.5}
    \def\op{.25}
    \def\Qxo{1.5}
    \def\Qyo{2.5}
    \def\Qxi{1.6}
    \def\Qyi{2}
    \def\Qxii{1.8}
    \def\Qyii{1.3}
    \def\Qxiii{1.9}
    \def\Qyiii{.8}
    \definecolor{color1}{rgb}{0,1,.25}
    \definecolor{color2}{rgb}{0,.25,1}
    \colorlet{color3}{color1!50!color2}

    \path (-\l,0) coordinate (3) --++ (\l,0) coordinate (0) --++ (\l,0) coordinate (1);
    \path (0,\l) coordinate (2) --++ (0,-2*\l) coordinate (4) --++ (\l,0) coordinate (5);

    \draw[name path = wall1, thick, ->] (3) -- (1);
    \draw[name path = wall2, thick, ->] (2) -- (4);
    \draw[name path = wall3, thick, ->] (0) -- (5);

    \node at (-.75,\l) {$1+z^{e_2^*}$};
    \node at (-3,.35) {$1+z^{-e_1^*}$};
    \node at (3.25,-2) {$1+z^{e_2^*-e_1^*}$};

    \path [name path = p, postaction={decorate,decoration={text along path,text align={right, right indent=15pt},text={${\dots}${}}}}, postaction={decorate,decoration={text along path,text align={right, right indent=42pt},text={${\dots}${}}}}] plot [smooth] coordinates {(\Qxo,\Qyo) (\Qxi,\Qyi) (\Qxii,\Qyii) (\Qxiii,\Qyiii) (\x,0)}; 

    \path [name path = g01] (\Qxo,\Qyo)--++(-\l,0);
    \path [name intersections={of=g01 and wall2,by=v01}];
    \path [name path = g02] (v01)--++(-\l,-\l);
    \path [name intersections={of=g02 and wall1,by=v02}];

    \path [name path = g11] (\Qxi,\Qyi)--++(-\l,0);
    \path [name intersections={of=g11 and wall2,by=v11}];
    \path [name path = g12] (v11)--++(-\l,-\l);
    \path [name intersections={of=g12 and wall1,by=v12}];

    \path [name path = g21] (\Qxii,\Qyii)--++(-\l,0);
    \path [name intersections={of=g21 and wall2,by=v21}];
    \path [name path = g22] (v21)--++(-\l,-\l);
    \path [name intersections={of=g22 and wall1,by=v22}];

    \path [name path = g31] (\Qxiii,\Qyiii)--++(-\l,0);
    \path [name intersections={of=g31 and wall2,by=v31}];
    \path [name path = g32] (v31)--++(-\l,-\l);
    \path [name intersections={of=g32 and wall1,by=v32}];

    \draw [thick, color=orange] (\Qxo,\Qyo)--(v01) node [pos=.5, sloped, above] {\tc{black}{$z^{-e_1^*}$}}--(v02) node [pos=.5, sloped, above] {\tc{black}{$z^{-e_1^*-e_2^*}$}}--++(0,-\l) coordinate (u);
    \path [postaction={decorate,decoration={text along path,text align=center,text={${z^{-e_{2}^*}}${}},raise=5pt}}] (u)--(v02);

    \draw [thick, color=orange] (\Qxi,\Qyi)--(v11)--(v12) --++(0,-\l);
    \draw [thick, color=orange] (\Qxii,\Qyii)--(v21)--(v22) --++(0,-\l);

    \draw [ultra thick, color=orange] (\x,0)--(0) node [pos=.4, sloped, below] {\tc{black}{$z^{-e_1^*}$}} node [pos=.5, sloped, above] {\tc{black}{$L_0$}}--++(0,-\l) coordinate (u2);
   \path [postaction={decorate,decoration={text along path,text align=center,text={${z^{-e_{2}^*}}${}},raise=-12pt}}] (u2)--(0);
   \path [postaction={decorate,decoration={text along path,text align=center,text={${L_2}${}},raise=7pt}}] (u2)--(0);

    \node [circle, fill, inner sep = 1.5pt, color = black] at (\Qxo,\Qyo) {};
    \node [circle, fill, inner sep = 1.5pt, color = black] at (\Qxi,\Qyi) {};
    \node [circle, fill, inner sep = 1.5pt, color = black] at (\Qxii,\Qyii) {};
    \node [circle, fill, inner sep = 1.5pt, color = black] at (\x,0) {};
    
    \node at (\Qxo+.45,\Qyo+.2) {$\gamma_1(0)$};
    \node at (\Qxi+.475,\Qyi+.15) {$\gamma_2(0)$};
    \node at (\Qxii+.5,\Qyii+.1) {$\gamma_k(0)$};
    \node at (\x+.33,-.3) {$x_0$};

    \coordinate (ex) at ($(0)!.5!(v22)$); 
    \path (ex)--++(0,-.75*\l) node  {$\cdots$};

    \node at (-1.35*\x,-\l) {$\gamma_1$};
    \node at (-1.1*\x,-\l) {$\gamma_2$};
    \node at (-.75*\x,-\l) {$\gamma_k$};
    \node at (-.23,-\l) {$\gamma$};
    
    \coordinate (ul) at ($(v21)!0.5!(v22)$);
    \coordinate (ll) at ($(v31)!0.5!(v32)$);
    \path (ul)--(ll) node [pos=-.6,sloped] {$\cdots$};

    \coordinate (left) at ($(v12)!.5!(v22)$); 
    \path (left)--++(0,-.75*\l) node  {$\cdots$};
    
    \path (-.65,0) -- (0,.65) node [pos=.5,sloped] {$L_1$};

\end{tikzpicture}

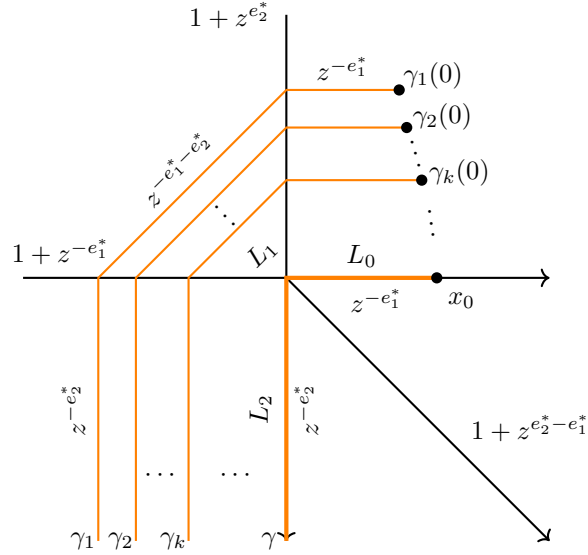
\captionof{figure}{\label{fig:A2SingBL} A non-generic broken line with a degenerate domain of linearity.} 

\end{center}
\end{minipage}
\end{center}

\end{example}

\begin{remark}
\thlabel{abuse}
By abuse of language, we frequently say $\gamma$ is a non-generic broken line without explicitly mentioning the sequence of generic broken lines converging to it. When we say $\gamma$ is a broken line we refer to either a generic one or a non-generic one. 
Moreover, we systematically identify a domain of linearity $L\subseteq (-\infty, 0]$ of a broken line $ \gamma $ with its image $\gamma(L) $.
\end{remark}

\begin{remark}
\thlabel{determined}
As in the generic case, if we know $\Par(\gamma) $, then the non-generic broken line $\gamma$ is completely determined by the condition $c_{s}=1$. 
\end{remark}

\begin{definition}
\thlabel{def:balanced}
\red{Given two generic broken lines $\gamma^{(1)}$, $\gamma^{(2)}$, we say that 
$\lrp{\gamma^{(1)}, \gamma^{(2)}}$ is a {\it{based pair of generic broken lines}} if 
$\gamma^{(1)}(0)=\gamma^{(2)}(0)$.
Next, if $\gamma^{(1)}$, $\gamma^{(2)}$ are {\emph{any}} broken lines (possibly non-generic), we say that $\lrp{\gamma^{(1)}, \gamma^{(2)}}$ is a {\it{based pair of broken lines}} if every $\lrp{ \gamma^{(1)}_k, \gamma^{(2)}_k }$ in the defining sequence is a based pair of generic broken lines. 
We say that $(\gamma^{(1)}, \gamma^{(2)})$ is {\it{balanced at $ x_0 $}} if it is a based pair of broken lines with $\gamma^{(1)}(0) = \gamma^{(2)}(0) = x_0 $  and $F(\gamma^{(1)})+F(\gamma^{(2)})=x_0$.
(See Figure~\ref{fig:A2Balanced}). 
In this situation we say that the generic based pairs $\lrp{\gamma^{(1)}_k, \gamma^{(2)}_k}$ in the sequence are {\it{balanced near $x_0$}}.} 
\end{definition}

\noindent
\begin{center}
\begin{minipage}{.85\linewidth}
\captionsetup{type=figure}
\begin{center}
\begin{tikzpicture}

    \def\x{2}
    \def\d{1}
    \def\l{3.5}
    \def\op{.25}
    \def\Qxo{1.5}
    \def\Qyo{2.5}
    \def\Qxi{1.6}
    \def\Qyi{2}
    \def\Qxii{1.8}
    \def\Qyii{1.3}
    \def\Qxiii{1.9}
    \def\Qyiii{.8}
    \definecolor{color1}{rgb}{0,1,.25}
    \definecolor{color2}{rgb}{0,.25,1}
    \colorlet{color3}{color1!50!color2}

    \path (-\l,0) coordinate (3) --++ (\l,0) coordinate (0) --++ (\l,0) coordinate (1);
    \path (0,\l) coordinate (2) --++ (0,-2*\l) coordinate (4) --++ (\l,0) coordinate (5);

    \draw[name path = wall1, thick, ->] (3) -- (1);
    \draw[name path = wall2, thick, ->] (2) -- (4);
    \draw[name path = wall3, thick, ->] (0) -- (5);

    \node at (-.75,\l) {$1+z^{e_2^*}$};
    \node at (-3,.35) {$1+z^{-e_1^*}$};
    \node at (3.25,-2) {$1+z^{e_2^*-e_1^*}$};

    \path [name path = p, postaction={decorate,decoration={text along path,text align={right, right indent=5pt},text={${\dots}${}}}}] plot [smooth] coordinates {(\Qxo,\Qyo) (\Qxi,\Qyi) (\Qxii,\Qyii) (\Qxiii,\Qyiii) (\x,0)}; 

    \path [name path = g01] (\Qxo,\Qyo)--++(-\l,0);
    \path [name intersections={of=g01 and wall2,by=v01}];
    \path [name path = g02] (v01)--++(-\l,-\l);
    \path [name intersections={of=g02 and wall1,by=v02}];

    \path [name path = g11] (\Qxi,\Qyi)--++(-\l,0);
    \path [name intersections={of=g11 and wall2,by=v11}];
    \path [name path = g12] (v11)--++(-\l,-\l);
    \path [name intersections={of=g12 and wall1,by=v12}];

    \path [name path = g21] (\Qxii,\Qyii)--++(-\l,0);
    \path [name intersections={of=g21 and wall2,by=v21}];
    \path [name path = g22] (v21)--++(-\l,-\l);
    \path [name intersections={of=g22 and wall1,by=v22}];

    \path [name path = g31] (\Qxiii,\Qyiii)--++(-\l,0);
    \path [name intersections={of=g31 and wall2,by=v31}];
    \path [name path = g32] (v31)--++(-\l,-\l);
    \path [name intersections={of=g32 and wall1,by=v32}];

    \draw [thick, color=color1] (\Qxo,\Qyo)--(v01) node [pos=.5, sloped, above] {\tc{black}{$z^{-e_1^*}$}}--(v02) node [pos=.5, sloped, above] {\tc{black}{$z^{-e_1^*-e_2^*}$}}--++(0,-\l) coordinate (u);
    \path [postaction={decorate,decoration={text along path,text align=center,text={${z^{-e_{2}^*}}${}},raise=5pt}}] (u)--(v02);

    \draw [thick, color=color1] (\Qxi,\Qyi)--(v11)--(v12) --++(0,-\l) node [pos=1, below] {\tc{black}{$\gamma^{(1)}_2$}};
    \draw [thick, color=color1] (\Qxii,\Qyii)--(v21)--(v22) --++(0,-\l) node [pos=1, below] {\tc{black}{$\gamma^{(1)}_3$}};
    \draw [thick, color=color1] (\Qxiii,\Qyiii)--(v31)--(v32) --++(0,-\l);

  \draw [very thick, color=color1] (\x,0)--(0) node [pos=.4, sloped, below] {\tc{black}{$z^{-e_1^*}$}}--++(0,-\l) coordinate (u2) node [pos=1, below] {\tc{black}{$\gamma^{(1)}$}};
  \path [postaction={decorate,decoration={text along path,text align=center,text={${z^{-e_{2}^*}}${}},raise=-12pt}}] (u2)--(0);

    \draw[thick, color=color2] (\Qxo,\Qyo) -- (\l,\Qyo) node [pos=.7, sloped, above] {\tc{black}{$z^{2 e_1^*}$}} node [right, pos=1] {\tc{black}{$\gamma^{(2)}_1$}};
    \draw[thick, color=color2] (\Qxi,\Qyi) -- (\l,\Qyi) node [right, pos=1] {\tc{black}{$\gamma^{(2)}_2$}};
    \draw[thick, color=color2] (\Qxii,\Qyii) -- (\l,\Qyii) node [right, pos=1] {\tc{black}{$\gamma^{(2)}_3$}};
    \draw[thick, color=color2] (\Qxiii,\Qyiii) -- (\l,\Qyiii) node [right, pos=1] {\tc{black}{$\gamma^{(2)}_4$}};
    \draw[very thick, color=color2] (\x,0) -- (\l,0) node [pos=.6, below] {\tc{black}{$z^{2 e_1^*}$}} node [right, pos=1] {\tc{black}{$\gamma^{(2)}$}};

    \node [circle, fill, inner sep = 1.5pt, color = color3] at (\Qxo,\Qyo) {};
    \node [circle, fill, inner sep = 1.5pt, color = color3] at (\Qxi,\Qyi) {};
    \node [circle, fill, inner sep = 1.5pt, color = color3] at (\Qxii,\Qyii) {};
    \node [circle, fill, inner sep = 1.5pt, color = color3] at (\Qxiii,\Qyiii) {};
    \node [circle, fill, inner sep = 1.5pt, color = color3] at (\x,0) {};
    
   \node at (\x+.1,-.25) {$x_0$};

    \node at (.5*\x,.6*\Qyiii) {$\vdots$};
    \coordinate (ex) at ($(0)!.5!(v32)$); 
    \path (ex)--++(0,-.5*\l) node {$\cdots$};
    \node at (1.5*\x,.6*\Qyiii) {$\vdots$};

    \node at (-1.4*\x,-\l-.35) {$\gamma^{(1)}_1$};
    \node at (-.35*\x,-\l-.35) {$\gamma^{(1)}_4$};

\end{tikzpicture}

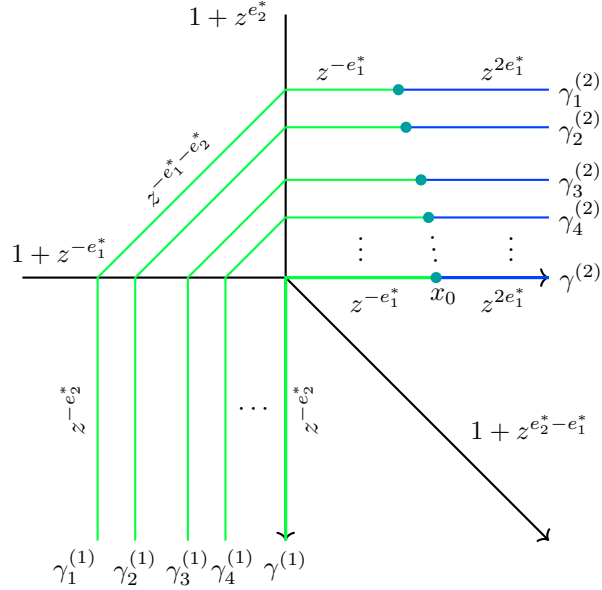
\captionof{figure}{\label{fig:A2Balanced} \red{In this figure, $\lrp{\gamma^{(1)}, \gamma^{(2)}}$ and each $\lrp{\gamma^{(1)}_k, \gamma^{(2)}_k}$ are based pairs of broken lines in the $A_2$ scattering diagram. 
The pair $\lrp{\gamma^{(1)}, \gamma^{(2)}}$ is balanced at $x_0= e_1^*$, and the generic pairs $\lrp{\gamma^{(1)}_k, \gamma^{(2)}_k}$ are balanced near $x_0$.}} 

\end{center}
\end{minipage}
\end{center}

The point of \thref{def:balanced} is that such pairs are used to compute the structure constant $\alpha\lrp{I(\gamma^{(1)}),I(\gamma^{(2)}),x_0}$.
In \cite{GHKK}, the description is very slightly different.  
They do not deal with sequences of broken lines, but require their broken lines to be generic, with endpoints in the complement of $\Supp{\scat}$.
In this setting, they take $x_0$ to be very close to $F(\gamma^{(1)}) + F(\gamma^{(2)}) $, rather than equal to this sum.
To recover the \cite{GHKK} description from \thref{def:balanced}, simply take a \red{based} pair $\lrp{\gamma_k^{(1)},\gamma_k^{(2)}}$ with $k$ sufficiently large.

\subsection{Broken line convexity} \label{sec:brokenconvex}

\begin{definition}\thlabel{def:segment}
Given a broken line $\gamma:(-\infty,0]\to M^\circ_{\R}$ (see \thref{abuse}),
a {\it{segment}} of $\gamma$ is the restriction of $\gamma$ to some closed interval $\lrb{t_0,t_1} \subset (-\infty,0]$.
We will often reparametrize this segment as 
$\tilde{\gamma}:\lrb{0,t_1-t_0} \to M^\circ_{\R}$, where $\tilde{\gamma}(t):= \gamma(t_0+t)$. 
In this case we say that $\tilde{\gamma}$ connects $ \tilde{\gamma}(0)$ and $\tilde{\gamma}(t_1-t_0)$ and refer to these as the endpoints of $\tilde{\gamma}$.
We include the degenerate case in which the segment is a point (that is $t_0=t_1$). 
Clearly, we can also talk about the support $\Supp(\tilde{\gamma})$ and the support with exponent vectors $\Par(\tilde{\gamma})$ of a broken line segment $\tilde{\gamma}$.
\end{definition}

\begin{definition}
\thlabel{well_defined_bend} Let $ \gamma: [0, T] \to M^{\circ}_{\R}\setminus \Sing(\scat)$ be a piecewise linear path with a finite number of domains of linearity. 
Suppose that there is a time $t\in (0,T)$ at which $\gamma$ is not linear and ${ \gamma (t)\in \bigcap_{j\in \mathfrak{J}}\wall_j}$ for some collection of walls $\lrp{\wall_j,f_{\wall_j}}_{j\in \mathfrak{J}}$ of $\scat$ indexed by a set $\mathfrak{J}$ that could be finite or infinite.
Note that by assumption ${ \dim\lrp{\wall_{j_1} \cap \wall_{j_2} } = \rank(M^\circ)-1}$ for any $j_1, j_2 \in \mathfrak{J}$, and there is some primitive $n \in \Nuf$ such that $\bigcup_{j \in \mathfrak{J}} \wall_j \subset n^\perp$.
Possibly after replacing $\scat$ with an equivalent scattering diagram, all associated scattering functions $f_{\wall_j}$ are of the form $\lrp{1+z^{a_j p^*(n)}}^{c_j}$ for some positive integers $a_j$ and $c_j$.
Let $F_{\mathfrak{J}} = \prod_{j\in \mathfrak{J}} f_{\wall_j}$, which is potentially a series in $\widehat{C[\widetilde{P}]}$.
Suppose that $ \gamma $ passes from the domain of linearity $L$ to the domain of linearity $L'$ (that is, $L$ corresponds to certain domain ending at $t$ and $L'$ to a domain beginning at $t$). Denote by $\mono_L$ (resp. $\mono_{L'}$) the exponent vector of $\gamma$ in the domain of linearity $L$ (resp. $L'$).
We say that the bending at time $t$ is \emph{allowed} if 
\[
\mono_L = \mono_{L'} + m,
\]
for some $m$ an exponent vector of a non-zero summand of $F_{\mathfrak{J}}^{\lra{n,\mono_L}}$.
This condition codifies the possible exponent vectors that a broken line might have after bending (see \thref{def:genbroken}).
\end{definition}

\thref{lem:reverse} below ensures that broken line segments can be traversed in both directions. In other words, allowed bendings are independent of the direction in which we travel through $\gamma$. Its proof is straightforward and follows from the definitions.

\begin{lemma}
\thlabel{lem:reverse}
If $\tilde{\gamma}:[0,T]\to M^{\circ}_{\R}$ is a broken line segment then $\overline{\tilde{\gamma}}:[0,T]\to M^{\circ}_{\R} $ given by ${\overline{\tilde{\gamma}}(t)= \tilde{\gamma}(T-t)}$ is a broken line segment.
\end{lemma}

Now that we have carefully stated what we mean by a broken line segment, we state once again \thref{def:blc_intro} from the introduction.
\begin{definition}\thlabel{def:blc}
A closed subset $S$ of $V^\trop(\R)$ is {\it{broken line convex}} 
if for every pair of rational points $s_1$, $s_2$ in $S$,
every segment of a broken line with endpoints $s_1$ and $s_2$ is entirely contained in $S$.  \end{definition}

\begin{remark}\thlabel{rem:Qpoints}
As alluded to in the introduction, the reader may find it odd that we ask for $s_1$ and $s_2$ to be rational points.
It is in fact very natural after consideration.
First, we are interested in when this set $S$ and its dilations define an algebra, and the irrational points are extraneous to this question.
In essence, from an algebraic perspective, we would really like to consider subsets of $V^\trop(\Q)$ instead of $V^\trop(\R)$.
The established convention however is to work over $\R$.
This alleviates potential issues in defining generic broken lines when endpoints lie in a region with dense walls-- the support of every wall is a {\emph{rational}} polyhedral cone-- and it may additionally be more natural from a symplectic or differential (as opposed to algebraic) viewpoint.
Next, broken lines are decorated with Laurent monomials in each domain of linearity, where the exponent vector of this monomial is the negative of the velocity vector along the line.
Wall crossing governs the allowed behavior of this broken line at the interface of domains of linearity, and this only applies to integral exponent vectors.
Obviously, a line segment connecting a generic pair of irrational points in a real vector space will have irrational slope, so we have no way of talking about broken line segments connecting arbitrary irrational points-- the notion simply does not make sense. 
\end{remark}

\begin{remark}
After discussing this project with Travis Mandel, we learned that he investigated a notion very similar to what we call broken line convexity in the setting of log Calabi-Yau surfaces in his thesis \cite{Travis}.
He relates this to a different convexity notion called {\it{strong convexity}}.  See \cite[Theorem~5.11]{Travis}.
Additional tropical convexity notions (such as {\it{min-convexity}}) are explored in connection to positivity in \cite[Section~8]{GHKK}.
Their convexity notions are given in terms of functions on $V^\trop(\R)$.
\end{remark}

The following lemmas relate broken line convex sets associated to $ \cA$ and $ \cX$ with those of $\cAp$, \change{ and the remark following the lemmas relates broken line convex sets associated to $\cA/T_{H^\circ}$ to those of $\cA$ and broken line convex sets associated to $\cX_\phi$ to those of $\cX$.} The proofs follows directly from the definitions.

\begin{lemma}
\thlabel{lem:projection_and_convexity}
Let $ \pi:\widetilde{M}^{\circ}_{\R} \to M^{\circ}_{\R}$ be the natural projection.  If $S\subset \widetilde{M}^{\circ}_{\R}$ is broken line convex, then $\pi(S)$ is broken line convex in $M^{\circ}_{\R}$. Similarly, if  $S\subset M^{\circ}_{\R}$ is broken line convex, then $\pi^{-1}(S)$ is broken line convex in $ \widetilde{M}^{\circ}_{\R}$.
\end{lemma}

\begin{lemma}
\thlabel{lem:slice_and_convexity}
The subset $(\cX^{\vee})^{\trop}(\R)\subset (\cAp^{\vee})^{\trop}(\R)$ is broken line convex. In particular, if $S\subset (\cAp^{\vee})^{\trop}(\R)$ is broken line convex then $S\cap (\cX^{\vee})^{\trop}(\R)$ is broken line convex. Moreover, if $S \subset (\cX^{\vee})^{\trop}(\R)$ is broken line convex in $(\cX^{\vee})^{\trop}(\R)$ then it is also broken line convex in $(\cAp^{\vee})^{\trop}(\R)$.
\end{lemma}

\change{
\begin{remark}\thlabel{rem:BLCRedDim}
The same reasoning applies for the quotients of $\cA$ and fibers of $\cX$ of Section~\ref{sec:RedDim}.
In detail,
for $\cX_\phi$, let $\pi$ be the quotient map $(\cX^\vee)^\trop(\R) \to (\cX^\vee_\phi)^\trop(\R)$.
If $S\subset (\cX^\vee)^\trop(\R)$ is broken line convex, then $\pi(S)$ is broken line convex in $(\cX^\vee_\phi)^\trop(\R)$.
If $S\subset (\cX^\vee_\phi)^\trop(\R)$ is broken line convex, then $\pi^{-1}(S)$ is broken line convex in  $(\cX^\vee)^\trop(\R)$.
Next, for $\cA_{\vb{t}}/T_{H^\circ}$, the subset $\big((\cA_{\vb{t}}/T_{H^\circ})^\vee\big)^\trop(\R) \subset (\cA^\vee_{\vb{t}})^\trop(\R)$ is broken line convex.
If the subset ${S\subset (\cA^{\vee}_{\vb{t}})^{\trop}(\R)}$ is broken line convex then so is $S\cap \big((\cA_{\vb{t}}/T_{H^\circ})^\vee\big)^\trop(\R)$. 
Moreover, if $S$ is broken line convex in $\big((\cA_{\vb{t}}/T_{H^\circ})^\vee\big)^\trop(\R)$, then it is also broken line convex as a subset of $(\cA^\vee_{\vb{t}})^\trop(\R)$.
\end{remark}
}
\section{Constructing a broken line segment from balanced broken lines}\label{sec:key_lemma}

In this section we introduce a procedure to construct a broken line segment from a pair of balanced broken lines. 
To be more precise let
($\gamma^{(1)}$,$\gamma^{(2)}$) be a balanced pair of broken lines and $(a,b)$ a pair of positive integers. In \thref{prop:bl2jp} we construct a broken line segment that connects the points $\frac{I(\gamma^{(1)})}{a}$ and $\frac{I(\gamma^{(2)})}{b}$ and bends precisely on the walls in which $\gamma^{(1)}$ and $\gamma^{(2)}$ bend. 
The construction of this broken line segment is motivated by {\it{jagged paths}}, described, for example, in \cite[Section~3]{GSTheta}.

As stated at the beginning of Section~\ref{sec:blc},
we take $\scat$ to be the scattering diagram of a cluster variety of type $\cA$ for which the injectivity assumption is satisfied, {\it{e.g.}} any $\cAp$ cluster variety.

\begin{remark}
\thlabel{lin_indep}
If $\gamma$ is \red{a} generic \red{broken line with bending points $x_1, \dots , x_s$,} then for each $1\leq i \leq s$ the vectors $x_i$ and $\mono_i$ are linearly independent. This is clear since $x_i\in n_{0,i}^{\perp}$ and $\mono_i\notin n_{0,i}^{\perp}$. 
\end{remark}

\subsection{An algorithmic description of the support} \label{sec:algosupport}

In this subsection we describe an algorithm that will allow us to obtain the support of the broken line segment we aim to construct. 
The following result is used repeatedly in the algorithm.

{\sloppy{
\begin{lemma}
\thlabel{slopes_full_gen}
Let $\gamma$ be a generic broken line with $s\geq 1$ bends. Let \red{$\lambda_1 , \lambda_2\in \R_{>0}$ and ${i \in \{0, \dots , s-1 \}}$.} 
Then the line segment connecting $\lambda_1 x_i$ and $\lambda_2\mono_{i} $ intersects the ray $ \ray_{i+1}\red{=\mathbb{R}_{\geq 0} x_{i+1}}$ in a unique point $x$. \red{Moreover, if $\lambda_1 , \lambda_2\in \Q_{>0}$ and $x_i\in M^{\circ}_{\Q}$ then $x \in M^{\circ}_{\Q}$.} 
\end{lemma}
}}

\begin{proof}
Observe that for all $i\in \{0, \dots ,s-1\}$ the points $x_{i+1}$, $x_{i}$ and $0$ are not colinear. 
We consider the plane $\Pi_i $ containing $x_i, x_{i+1}$ and the origin $0$. 
Clearly, both $\ray_{i}$ and $\ray_{i+1}$ are contained in $\Pi_i$, and $\mono_i$ is also a point of $\Pi_i$. 
We have to prove a statement of Euclidean geometry in the plane $\Pi_i$. First recall that $\mono_i$ is the negative of the velocity of $\gamma $ in the domain of linearity $L_i$. Therefore, the vector $\overrightarrow{\mono_i}$ based at $x_i$ with direction $ \mono_i$ lives in $\Pi_i$ and is directed from $x_i$ to the ray $\ray_{i+1}$. 
 Since $\lambda_2>0$ we have that the vector $\lambda_2\overrightarrow{\mono_i}$ based at $x_i$ with direction $\lambda_2\mono_i$ is also directed from $x_i$ to $\ray_{i+1}$ (see Figure \ref{fig:slopes_full_gen}). Observe now that the vector $-\lambda_1\overrightarrow{x_i}$ based at $x_i$ with direction $-\lambda_1 x_i$ is directed from $x_i$ to the origin. 
 Since the origin belongs to $\ray_{i+1}$, we have that the vector $\overrightarrow{v}$  based at $x_i$ with direction $\lambda_2 \mono_i-\lambda_1 x_i$ points from $x_i$ to $\ray_{i+1}$. 
 Finally, since $\lambda_1 >0 $ we must have that the line passing through
 $
 \lambda_1x_1$ and $\lambda_2\mono_i$ is parallel to $\overrightarrow{v}$ and must intersect $ \ray_{i+1} $ in a unique point which we call $x$. 
 Now assume that \red{$\lambda_1 , \lambda_2\in \Q_{>0}$ and} $x_i\in M^{\circ}_{\Q} $. 
 Observe that the conditions $\mono_i\in M^{\circ}$ and $ \wall_i\subset n_{0,i}^{\perp}$ for some $n_{0,i}\in N$ imply that $x$ is in $ M^{\circ}_{\Q}$. Since both $x $ and $ x_{i+1}$ are in the ray $\ray_{i+1}$ then  $x=\beta x_{i+1}$ for some $\beta \in \Q_{>0}$.
\end{proof}

\begin{figure}[!htbp]
    \centering
\begin{tikzpicture}[scale=1]
\draw[dashed,line width=0.2mm, orange] (-1.8,3.6)-- (-3.8,3.5);
\draw[dashed,line width=0.2mm, orange] (1.6,3.2)-- (3.5,3.2);
\draw[fill,rounded corners, green,opacity=.19] (3.5,4.4) -- (3.5,-0.5) -- (-3.8,-0.5) -- (-3.8,4.4);
\draw[-,line width=0.19mm] (0,0) -- (-2.2,4.4);
\draw[-,line width=0.19mm] (0,0) -- (2.2,4.4);
\draw[-,line width=0.3mm, orange] (1.6,3.2) -- (-1.8,3.6);
\draw[-,line width=0.3mm, magenta] (0.7,1.4) -- (-0.42,0.84);
\draw[->, black, line width=0.5mm] (1.6,3.2) -- (0.9,1.8);
\draw[->, orange, line width=0.5mm] (1.6,3.2) -- (-.1,3.4);
\draw[densely dotted, ->, magenta, line width=0.3mm] (1.6,3.2) -- (-0.8,2);
\draw[-, magenta, line width=0.3mm] (-1.7,0.2) --(-0.42,0.84);
\draw[loosely dotted, orange, line width=0.5mm] (0,0) -- (-1.7,0.2);

\node at (2.3,4.1) {$\ray_i $};
\node at (-2.5,4.1) {$\ray_{i+1} $};
\node at (-3.2,0) {\huge $\Pi_i$};
\node at (1.8,2.2) {$ - \lambda_1 \overrightarrow{x_i}$};
\node at (0.9,3.6) {$ \lambda_2\overrightarrow{\mono_i} $};
\node at (0, 3) {$L_i$};
\node at (0.5, 2.4) {$\overrightarrow{v}$};
\node at (-3.5, 3.2) {$ \gamma$};
\node at (-1.7,0.6) {$\lambda_2 \mono_i$};

\node at (-1.7,0.2) {$\bullet$};
\node at (1.6,3.2) {$\bullet$};
\node at (1.8,3) {$x_i$};
\node at (0.7,1.4) {$\bullet$};
\node at (1,1.15) {$\lambda_1 x_i$};
\node at (-1.8,3.6) {$\bullet$};
\node at (-2.05,3.35) {$x_{i+1}$};
\node at (1.6,3.2) {$\bullet$};
\node at (-0.42,0.84) {$\bullet$};
\node at (-0.52,0.54) {$x$};
\node at (0,0) {$\bullet$};
\node at (0,-0.3) {$0$};

\end{tikzpicture}
\caption{The point $x$ in $\ray_{i+1} $ determined by $\lambda_1 x_i$ and $\lambda_2 \mono_i$.}
\label{fig:slopes_full_gen}
\end{figure}

\begin{algo} 
\thlabel{algorithm_1}
The {\bf input} of the algorithm is the following:
\begin{enumerate}
    \item a pair of positive integers $(a,b)$;
\item $ \Par(\gamma)$ of a broken line $\gamma$. 
\end{enumerate}
The {\bf output} is the support $\Supp(\tilde{\gamma})$ of a piecewise linear segment $\tilde{\gamma}$ in $M^{\circ}_{\R}$ with the properties:
\begin{enumerate}
    \item the endpoints of $\Supp(\tilde{\gamma})$ are $\dfrac{x_0}{a+b}$ and $\dfrac{\mono_s}{a}$;
    \item the piecewise linear segment $\Supp(\tilde{\gamma})$ bends at a ray $\ray$ if and only if $\gamma$ bends at $\ray$.
\end{enumerate} 

{\bf The algorithm.} Let $\Par(\gamma)=((L_s,\mono_s), \dots , (L_0,\mono_0))$ and set $\tilde{x}_0:=\dfrac{x_0}{a+b}$. First assume that $\gamma$ is generic. We have two cases $ \tilde{x}_0 = \dfrac{\mono_0}{a}$ and $ \tilde{x}_0 \neq \dfrac{\mono_0}{a}$.\\
\textbf{Case 1}: If $ \tilde{x}_0 = \dfrac{\mono_0}{a}$ then $\gamma$ cannot bend (that is $s=0$). Define
\[
\Supp(\tilde{\gamma})= \left\lbrace \tilde{x}_0\right\rbrace
\]
and stop here. By construction $\Supp(\tilde{\gamma})$ satisfies the desired properties.
\\
\textbf{Case 2}: If $\tilde{x}_0 \neq \dfrac{\mono_0}{a}$ let $l_0$ be the line segment whose endpoints are  $\dfrac{\mono_0}{a}$ and $\tilde{x}_0$ . We have two sub-cases: $s=0$ and $s>0$.

If $s=0$ we set $\tilde{L}_0 :=l_0 $. We define 
\[
\Supp(\tilde{\gamma}):= \tilde{L}_0
\]
and stop here.
Clearly, $\Supp(\tilde{\gamma}) $ satisfies the desired properties.

If $s>0$ then \thref{slopes_full_gen} ensures that $l_0$ intersects $\ray_1$ in a unique point $\tilde{x}_1=\beta_1 x_1$ for some $\beta_1\in \red{\R}_{>0}$.
In this case we define $\tilde{L}_0$ to be the line segment whose endpoints are $ \tilde{x}_0$ and $ \tilde{x}_1$. 
Now suppose that for $i \in \{ 1, \dots , s-1 \}$ we have defined a point $\tilde{x}_i \in \ray_{i}$ that lies in the line segment $l_{i-1}$ connecting $\tilde{x}_{i-1}$ and $\dfrac{\mono_{i-1}}{a}$, and a positive number $\beta_i\in \red{\R}_{>0}$ such that $ \tilde{x}_i=\beta_i x_i$.
We define $\tilde{L}_{i-1}$ to be the line segment whose endpoints are $\tilde{x}_{i-1} $ and $\tilde{x}_i$.
Let $l_{i}$ be the line segment connecting $\tilde{x}_i$ and $\dfrac{\mono_i}{a}$ (here we implicitly use \thref{lin_indep}). We define $\tilde{x}_{i+1}$ as the unique intersection point of $l_i$ and $\ray_{i+1}$ (using \thref{slopes_full_gen}). Moreover, we define $\tilde{L}_i$ to be the line segment whose endpoints are $\tilde{x}_i$ and $\tilde{x}_{i+1}$, see Figure \ref{fig:algorithm_1}. 
\red{In this way we have defined a point $ \tilde{x}_{i+1}\in \ray_{i+1}$ and a line segment $\tilde{L}_i$ for each $0\leq i < s$.}
Let $\tilde{L}_s$ be the line segment connecting $\tilde{x}_s$ and $\dfrac{\mono_s}{a}$. We define the support of $\tilde{\gamma}$ to be
\[
\text{Supp}(\tilde{\gamma}):=\bigcup_{i=0}^{s}\tilde{L}_{i}
\]
and stop here. By construction $\Supp(\tilde{\gamma}) $ satisfies the desired properties.

Now we treat the non-generic case. So assume $\gamma $ is non-generic and let $(\gamma_k)_{k\in \N}$ be the sequence of generic broken lines approaching $\gamma$ (see \thref{lin_indep}). For each $k \in \N$ we construct the support of a piecewise linear segment $\Supp (\tilde{\gamma}_k) $ applying the previous case to $ \gamma_k$. Let $x_{1;k}, \dots , x_{s;k}$ (resp. $\tilde{x}_{1;k}, \dots , \tilde{x}_{s;k}$) be the bending points of $ \gamma_k$ (resp. $\tilde{\gamma}_k$) and let $\tilde{x}_{0;k}=\dfrac{\gamma_k(0)}{a+b}$. Observe that if the sequence $(\tilde{x}_{i-1;k})_{k \in \N} $ is convergent for some $ i\in \{1, \dots, s\} $ then the sequence $(\tilde{x}_{i;k})_{k \in \N} $ is also convergent. 
To see this recall first that $\tilde{x}_{i;k}$ is the unique intersection point of the ray $\ray_{i;k}:=\R_{\geq 0}x_{i;k}$ and the line segment with endpoints $\dfrac{\mono_{i-1}}{a}$ and $\tilde{x}_{i-1;k}$. 
Next, by definition the rays $\ray_{i;k}$ converge to the ray $\ray_{i}$ and thus if the sequence $\tilde{x}_{i-1;k}$ is convergent, then the sequence $(\tilde{x}_{i;k})_{k \in \N} $ must converge to a point in $\ray_i$. Since the sequence $(\tilde{x}_{0;k})_{k\in \N}$ converges to $ \tilde{x}_{0}$ then we must have that the sequence $(\tilde{x}_{i;k})_{k \in \N} $ converges to a point, say $\tilde{x}_{i} $, for all $i \in \{ 0, \dots, s \}$. 
We define $\tilde{L}_{i-1}$ to be the line segment whose endpoints are $\tilde{x}_{i-1} $ and $\tilde{x}_i$ for $i \in \{ 1, \dots, s \}$ and let $\tilde{L}_s$ be the line segment connecting $\tilde{x}_s$ and $\dfrac{\mono_s}{a}$. We define the support of $\tilde{\gamma}$ to be
\[
\text{Supp}(\tilde{\gamma}):=\bigcup_{i=0}^{s}\tilde{L}_{i}
\]
and stop here. By construction $\Supp(\tilde{\gamma}) $ satisfies the desired properties.
\end{algo}

\begin{figure}[H]
    \centering
\begin{tikzpicture}[scale=1]
\draw[-,line width=0.3mm] (0,0) -- (0,5.75);
\draw[-,line width=0.3mm] (0,0) -- (3.8,3.8);
\draw[-,line width=0.3mm] (0,0) -- (-4,4);

\draw[fill,rounded corners,blue,opacity=.19] (-1.5,6.5) -- (3.8,3.85) -- (3.8,-1.9) -- (-1.5,0.75);

\draw[fill,rounded corners, green,opacity=.19] (1.375,6.25) -- (-4,4.295455) -- (-4,-1.454545) -- (1.375,0.5);

\node at (0,0) {\tiny $ \bullet $};
\node at (2.7,-0.5) {\huge ${\Pi_{i-1}}$};
\node at (-3.5,0.7) {\huge ${\Pi_{i}}$};
\node at (0,6) {$\ray_i $};
\node at (4,4) {$\ray_{i-1}$};
\node at (-4.1,4.15) {$\ray_{i+1} $};

\draw[-,line width=0.4mm, orange] (3,3) -- (0,4.5);
\node at (3,3) {\tiny $ \bullet $};
\node at (3.4,2.8) {$x_{i-1}$};
\node at (0,4.5) {\tiny $ \bullet $};
\node at (0.2,4.1) {$x_{i}$};

\draw[line width=0.4mm, loosely dotted, orange] (0,0) -- (-6,3);
\node at (-6,3) {\tiny $ \bullet $};
\node at (-6,3.5) {$\dfrac{\mono_{i-1}}{a}$};
\node at (-4.8, 3.1) {$ l_{i-1}$};

\draw[-,line width=0.4mm, magenta] (2,2) -- (0,2.25);
\draw[dashed] (0,2.25) -- (-4,2.75);
\draw[-] (-4,2.75) -- (-6,3);
\node at (2,2) {\tiny $ \bullet $};
\node at (2.4,1.8) {$\tilde{x}_{i-1}$};
\node at (0,2.25) {\tiny $ \bullet $};
\node at (0.2,1.9) {$\tilde{x}_{i}$};

\draw[-,line width=0.4mm, orange] (0,4.5) -- (-3.3,3.3);
\draw[line width=0.4mm, loosely dotted, orange] (0,0) -- (-2.538461,-0.923076);
\node at (-3.3, 3.3) {\tiny $ \bullet $};
\node at (-3.6, 3.1) {$ x_{i+1} $};
\node at (-2, 4.1) {$c_{L_i}z^{ \mono_i}$};
\node at (-2.7, -0.5) {$ \dfrac{\mono_i}{a}$};
\node at (-2, 0.3) {$ l_i $};

\draw[-,line width=0.4mm, magenta] (0,2.25) -- (-1,1);
\draw[-] (-1,1) -- (-2.538461,-0.923076);
\node at (-0.8, 1.8) {$ \tilde{L}_{i} $};
\node at (-1, 1) {\tiny $ \bullet $};
\node at (-1.5, 1.1) {$ \tilde{x}_{i+1} $};
\node at (-2.538461,-0.923076) {\tiny $ \bullet $};
\node at (0,-0.3) {$0$};
\node at (0,0) {\tiny $\bullet$};
\node at (0,4.5) {\tiny $\bullet$};
\node at (0,2.25) {\tiny $\bullet$};

\node at (1.9,4.2) {$c_{L_{i-1}}z^{\mono_{i-1}}$};
\node at (1.2,2.4) {$\tilde{L}_{i-1}$};

\end{tikzpicture}
\caption{Construction of the domains of linearity $\tilde{L}_{i-1}$ and $\tilde{L}_i$ following \thref{algorithm_1}.}
\label{fig:algorithm_1}
\end{figure}

\begin{remark}
 \thref{algorithm_1} does not provides a parametrization of the piecewise linear segment $\tilde{\gamma}:[0,T]\to M^{\circ}_{\R}$; it only gives its image in $M^{\circ}_{\R}$. In case $\gamma $ is non-generic, for fixed $i \in \{ 1, \dots , s \}$, the line segments with endpoints $\tilde{x}_{i-1;k}$ and $\tilde{x}_{i;k}$ are in general not parallel as $k$ varies.
\end{remark}

\begin{example}
We illustrate the algorithm for a broken line living in the scattering diagram of type $A_2$. Take $\gamma$ to be the broken line illustrated in Figure~\ref{fig:gamma}.

\noindent
\begin{center}
\begin{minipage}{.85\linewidth}
\captionsetup{type=figure}
\begin{center}
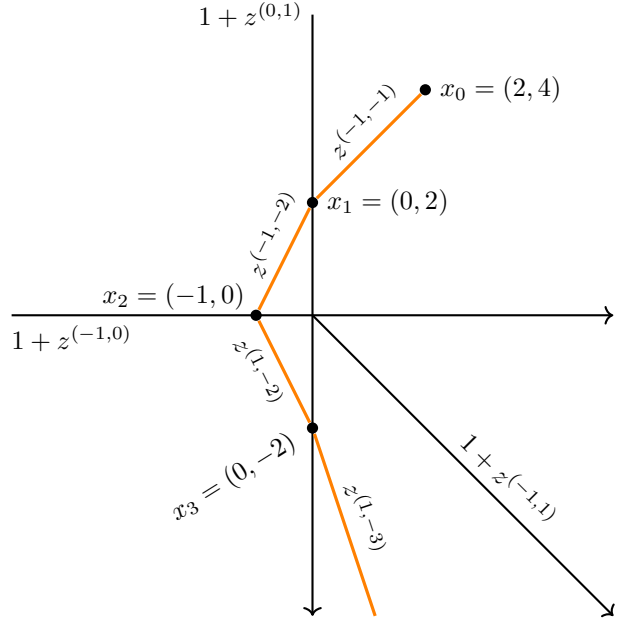

\begin{tikzpicture}

    \def\d{.75}
    \def\l{4}

    \path (-\l,0) coordinate (3) --++ (\l,0) coordinate (0) --++ (\l,0) coordinate (1);
    \path (0,\l) coordinate (2) --++ (0,-2*\l) coordinate (4) --++ (\l,0) coordinate (5);

    \draw[name path = wall1, thick, ->] (3) -- (1) node [pos=0.1, below] {$1+z^{\lrp{-1,0}}$};
    \draw[name path = wall2, thick, ->] (2) -- (4) node [pos=0, left] {$1+z^{\lrp{0,1}}$};
    \draw[name path = wall3, thick, ->] (0) -- (5) node [pos=0.6, sloped, above] {$1+z^{\lrp{-1,1}}$};

    \node [circle, fill, inner sep = 1.5pt]  (x0) at (2*\d,4*\d) {}; 
    \node [circle, fill, inner sep = 1.5pt]  (x1) at (0,2*\d) {}; 
    \node [circle, fill, inner sep = 1.5pt]  (x2) at (-\d,0) {}; 
    \node [circle, fill, inner sep = 1.5pt]  (x3) at (0,-2*\d) {}; 

    \node at (2*\d+1,4*\d) {$x_0= (2,4)$}; 
    \node at (1,2*\d) {$x_1= (0,2)$}; 
    \node at (-1.85,.25) {$x_2= (-1,0)$};     
    \path (-\l,-\l)--(x3) node [pos=.75, sloped] {$x_3= (0,-2)$}; 

    \path[name path = lin] (x3)--++(1,-3);
    \path[name path = bot] (-\l,-\l)--(\l,-\l);
    \path[name intersections={of=lin and bot,by=end}];
    
    \draw[very thick, orange] (x0) -- (x1) node [pos=.4,sloped, above] {\tc{black}{$z^{\lrp{-1,-1}}$}} -- (x2) node [pos=.35,sloped, above] {\tc{black}{$z^{\lrp{-1,-2}}$}} -- (x3) node [pos=.4,sloped, below] {\tc{black}{$z^{\lrp{1,-2}}$}} -- (end) node [pos=.5,sloped, above] {\tc{black}{$z^{\lrp{1,-3}}$}};

\end{tikzpicture}
\captionof{figure}{\label{fig:gamma} The broken line $\gamma$.} 

\end{center}
\end{minipage}
\end{center}

We let $a=1$ and $b=2$. 
Then $\tilde{x}_0= \left( \frac{2}{3}, \frac{4}{3}\right)$.
The first step is to construct the line segment $l_0$ which connects $\dfrac{\mono_0}{a} = (-1,-1)$ and $\tilde{x}_0= \left( \frac{2}{3}, \frac{4}{3}\right)$. We obtain that $\tilde{x}_1= \left( 0, \frac{2}{5}\right)$.  See Figure \ref{fig:l0}.

\noindent
\begin{center}
\begin{minipage}{.85\linewidth}
\captionsetup{type=figure}
\begin{center}
\begin{tikzpicture}

    \def\d{.75}
    \def\l{4}
    \def\op{.25}
    
    \path (-\l,0) coordinate (3) --++ (\l,0) coordinate (0) --++ (\l,0) coordinate (1);
    \path (0,\l) coordinate (2) --++ (0,-2*\l) coordinate (4) --++ (\l,0) coordinate (5);

    \draw[name path = wall1, thick, ->] (3) -- (1);
    \draw[name path = wall2, thick, ->] (2) -- (4);
    \draw[name path = wall3, thick, ->] (0) -- (5);

    \node [circle, fill, inner sep = 1.5pt, opacity=\op]  (x0) at (2*\d,4*\d) {}; 
    \node [circle, fill, inner sep = 1.5pt, opacity=\op]  (x1) at (0,2*\d) {}; 
    \node [circle, fill, inner sep = 1.5pt, opacity=\op]  (x2) at (-\d,0) {}; 
    \node [circle, fill, inner sep = 1.5pt, opacity=\op]  (x3) at (0,-2*\d) {}; 

    \path[name path = lin] (x3)--++(1,-3);
    \path[name path = bot] (-\l,-\l)--(\l,-\l);
    \path[name intersections={of=lin and bot,by=end}];
    
    \draw[very thick, orange, opacity = \op] (x0) -- (x1) -- (x2) -- (x3) -- (end);
    \draw [gray, opacity=\op] (0) -- (.5*\l,\l); 

    \node [circle, fill, inner sep = 1.5pt] (xt0) at (.667*\d,1.333*\d) {};
    \node [circle, fill, inner sep = 1.5pt] (m0) at (-\d,-\d) {};

    \path [name path = L0] (xt0) -- (m0);
    \path [name intersections={of=L0 and wall2,by=xt1}];
    
    \draw [very thick, magenta] (xt0) -- (xt1) node [circle, fill, inner sep = 1.5pt, black, pos=1] {}; 
    \draw (xt1)--(m0) node [pos=.65, sloped, below] {$l_0$};

    \node at (.9,1.333*\d) {$\tilde{x}_0$}; 
    \node at (-.3,.3) {$\tilde{x}_1$}; 
    \node at (-1.1,-1.1) {$\dfrac{\mono_0}{a}$}; 

\end{tikzpicture}

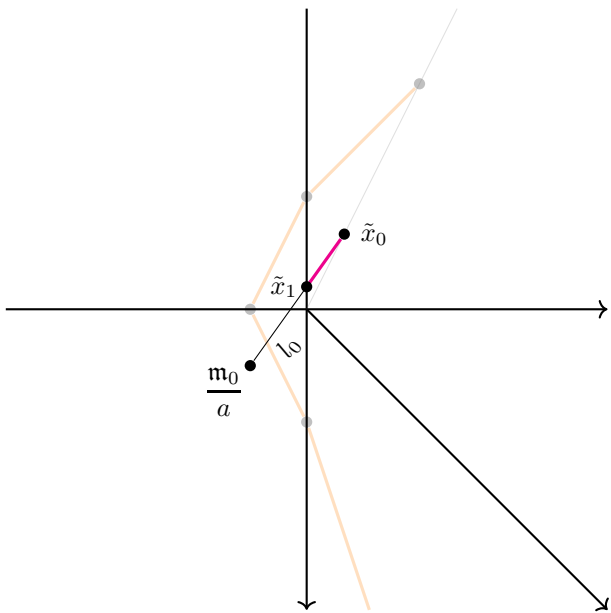
\captionof{figure}{\label{fig:l0} The first step of the algorithm.} 

\end{center}
\end{minipage}
\end{center}

The line $l_0$ intersects the wall $y$-axis at the unique point $\tilde{x}_1 = \left(0,\frac{2}{5} \right)$.
Repeating this process for the remaining exponent vectors, we obtain the broken line segment depicted in Figure \ref{fig:constructsegment}.

\noindent
\begin{center}
\begin{minipage}{\linewidth}
\captionsetup{type=figure}
\begin{center}
\begin{tikzpicture}

    \def\d{.75}
    \def\l{4}
    \def\op{.25}
    
    \path (-\l+1,0) coordinate (3) --++ (\l-1,0) coordinate (0) --++ (\l-.5,0) coordinate (1) --++ (0,-\l+.5) coordinate (5);
    \path (0,\l) coordinate (2) --++ (0,-2*\l) coordinate (4); 

    \draw[name path = wall1, thick, ->] (3) -- (1);
    \draw[name path = wall2, thick, ->] (2) -- (4);
    \draw[name path = wall3, thick, ->] (0) -- (5);

    \node [circle, fill, inner sep = 1.5pt, opacity=\op]  (x0) at (2*\d,4*\d) {}; 
    \node [circle, fill, inner sep = 1.5pt, opacity=\op]  (x1) at (0,2*\d) {}; 
    \node [circle, fill, inner sep = 1.5pt, opacity=\op]  (x2) at (-\d,0) {}; 
    \node [circle, fill, inner sep = 1.5pt, opacity=\op]  (x3) at (0,-2*\d) {}; 

    \path[name path = lin] (x3)--++(.8333,-2.5);
    \path[name path = bot] (-\l+1,-\l)--(\l-.5,-\l);
    \path[name intersections={of=lin and bot,by=end}];
    
    \draw[very thick, orange, opacity = \op] (x0) -- (x1) -- (x2) -- (x3) -- (end);
    \draw [gray, opacity=\op] (0) -- (.5*\l,\l); 

    \node [circle, fill, inner sep = 1.5pt] (xt0) at (.667*\d,1.333*\d) {};
    \coordinate (m0) at (-\d,-\d);

    \path [name path = L0] (xt0) -- (m0);
    \path [name intersections={of=L0 and wall2,by=xt1c}];
    \node [circle, fill, inner sep = 1.5pt] (xt1) at (xt1c) {};

    \draw [very thick, magenta] (xt0) -- (xt1); 

    \node at (.667*\d+1,1.333*\d) {$\tilde{x}_0= \lrp{\frac{2}{3},\frac{4}{3}}$}; 
    \node at (1,.3) {$\tilde{x}_1= \lrp{0,\frac{2}{5}}$}; 

    \node [circle, fill, inner sep = 1.5pt] (m1) at (-\d,-2*\d) {};
    \node at (-2,-1.6) {$\frac{\mono_1}{a}= \lrp{-1,-2}$}; 
    \node [circle, fill, inner sep = 1.5pt] (xt2) at (-.167*\d,0) {};
    \node (x2lab) at (-2*\d,1.3) {$\tilde{x}_2=\lrp{-\frac{1}{6},0}$};

    \draw [very thick, magenta] (xt1)--(xt2);
    \draw (xt2)--(m1);
    
    \node [circle, fill, inner sep = 1.5pt] (m2) at (\d,-2*\d) {};
    \node [circle, fill, inner sep = 1.5pt] (xt3) at (0,-.2857*\d) {};

    \draw [very thick, magenta] (xt2)--(xt3);
    \draw (xt3)--(m2);

    \node [circle, fill, inner sep = 1.5pt] (m3) at (\d,-3*\d) {};
    \draw [very thick, magenta] (xt3)--(m3);

    \node (x3lab) at (2,-.5) {$\tilde{x}_3= \lrp{0,-\frac{2}{7}}$}; 
    \path (m2)--++ (2,-2) node [pos=.4, sloped] {$\frac{\mono_2}{a}= \lrp{1,-2}$}; %%%%%
    \path (m3)--++ (1.75,-1.75) node [pos=.45, sloped] {$\frac{\mono_3}{a}= \lrp{1,-3}$}; %%%%%
    
    \draw[->, darkgray] (x3lab)--(xt3);
    \draw[->, darkgray] (x2lab)--(xt2);
  
%    \draw (current bounding box.north east) -- (current bounding box.north west) -- (current bounding box.south west) -- (current bounding box.south east) -- cycle; 
  
\begin{scope}[xshift=7.5cm]
    \def\d{2.666}
    \def\l{4}
    \def\op{.25}
    \def\fudge{1}
    
    \path (-\l+1,0) coordinate (3) --++ (\l-1,0) coordinate (0) --++ (\l-.5,0) coordinate (1) --++ (0,-\l+.5) coordinate (5);
    \path (0,\l) coordinate (2) --++ (0,-2*\l) coordinate (4); 

    \draw[name path = wall1, thick, ->] (3) -- (1);
    \draw[name path = wall2, thick, ->] (2) -- (4);
    \draw[name path = wall3, thick, ->] (0) -- (5);

    \node [circle, fill, inner sep = 1.5pt] (xt0) at (.667*\d,1.333*\d) {};
    \coordinate (m0) at (-\d,-\d);

    \path [name path = L0] (xt0) -- (m0);
    \path [name intersections={of=L0 and wall2,by=xt1c}];
    \node [circle, fill, inner sep = 1.5pt] (xt1) at (xt1c) {};

    \draw [very thick, magenta] (xt0) -- (xt1); 

    \node at (.667*\d+1,1.333*\d) {$\tilde{x}_0= \lrp{\frac{2}{3},\frac{4}{3}}$}; 
    \node at (1,1) {$\tilde{x}_1= \lrp{0,\frac{2}{5}}$}; 

    \node [circle, fill, inner sep = 1.5pt] (xt2) at (-.167*\d,0) {};
    \node at (-1.5,.3) {$\tilde{x}_2=\lrp{-\frac{1}{6},0}$};

    \draw [very thick, magenta] (xt1)--(xt2);

    \node [circle, fill, inner sep = 1.5pt] (xt3) at (0,-.2857*\d) {};

    \draw [very thick, magenta] (xt2)--(xt3);
    
    \path[name path = lin] (xt3)--++(\fudge,-3.286*\fudge);
    \path[name path = bot] (-\l+1,-\l)--(\l-.5,-\l);
    \path[name intersections={of=lin and bot,by=end}];    
    
    \draw [very thick, magenta] (xt3)--(end);

    \node at (-1.15,-.85) {$\tilde{x}_3= \lrp{0,-\frac{2}{7}}$}; 
        
\end{scope}  
  
\end{tikzpicture}

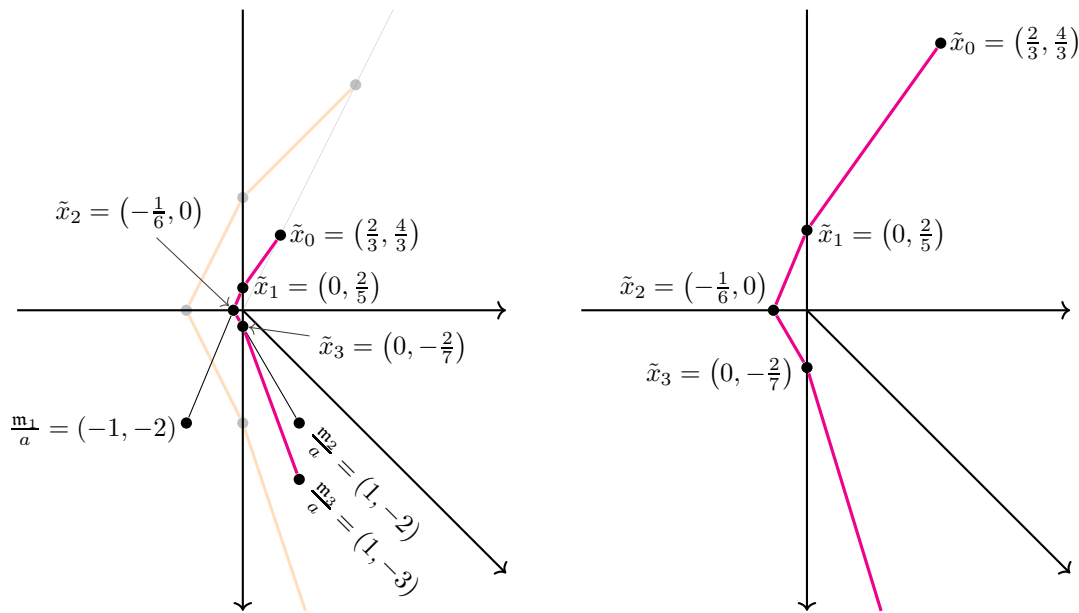
\captionof{figure}{\label{fig:constructsegment} Support of a broken line segment $\Supp\lrp{\tilde{\gamma}}$ associated to the broken line $\gamma$ of Figure~\ref{fig:gamma} and the pair of integers $a=1$, $b=2$, with a zoomed in view provided on the right.} 

\end{center}
\end{minipage}
\end{center}

\end{example}

\subsection{Associating monomials to the support} \label{sec:assocmono}

\thref{algorithm_1} only allows us to construct $ \Supp(\tilde{\gamma}) $. 
In order to define the broken line segment $\tilde{\gamma}=\left(\Par(\tilde{\gamma}), c_{L_0}, \dots , c_{L_s} \right)$, we need to attach a monomial $ c_i z^{\widetilde{\mono}_i} \in \Z_{>0}[ M^{\circ}]$ to every domain of linearity $\tilde{L}_i$ satisfying the axioms of a broken line.
Note that there are infinitely many broken line segments with the same support.
In \thref{key_lemma} below we construct for each $\lambda \in \Z_{>0}$ a broken line segment with the prescribed support.
We first give a few remarks and definitions.

\begin{remark}
\thlabel{velocity}
Suppose that $\gamma $ is generic and that the bending points belong to $M^\circ_{\Q}$. Then the velocity of $\tilde{\gamma}$ at a domain of linearity $\tilde{L}_i$ must be of the form $C_i \left(\dfrac{\mono_i}{a}-\tilde{x}_i\right)$ for some $C_i \in \Q_{>0}$. Indeed, $\tilde{L}_i$ is contained in the line passing through $\tilde{x}_i$ and $\dfrac{\mono_i}{a}$. Moreover,
both $\tilde{x}_i$ and $\dfrac{\mono_i}{a}$ 
belong to $M^{\circ}_{\Q}$ and $\widetilde{\mono}_i \in M^\circ$. 
\end{remark}

\begin{definition}
Let $\gamma$ be a broken line. For $i \in \{ 0, \dots  ,s  \} $ define
\[
\rho_i(\gamma):= \prod_{k=i}^{s-1}\langle  n_{0,k+1}, \mono_k \rangle.
\]
In particular, $\rho_i(\gamma) \in \Z_{>0}$ for all $i$  and
\[\rho_0(\gamma)  M^{\circ} \subset \rho_{1}(\gamma)  M^{\circ} \subset \cdots \subset \rho_{s}(\gamma) M^{\circ}= M^{\circ}.
\]
Moreover, for $1\leq i \leq s$ the following holds
\[
\rho_{i}(\gamma)=\dfrac{\rho_{i-1}(\gamma)}{\langle n_{0,i}, \mono_{i-1} \rangle}.
\]
\end{definition}

\begin{lemma} \thlabel{key_lemma}
\red{Suppose $\gamma$ is a broken line in $M^{\circ}_{\R}$ such that $\gamma(0)=x_0\in M^{\circ}_{\Q}$. Let $a,b $ and $\lambda$ be positive integers.} 
For each $ i \in \{ 0, \dots , s \} $ we can associate a monomial
\[
c_{\tilde{L}_{i}}z^{\widetilde{\mono}_i} \in \Z_{>0}[\rho_i (\gamma)\lambda M^{\circ}] 
\]
\red{to the domain of linearity $\tilde{L}_i$ of $\Supp(\tilde{\gamma}) $ constructed applying \thref{algorithm_1} to $\gamma $ and $(a,b)$ in order to obtain a broken line segment $\left(\Par(\tilde{\gamma}), c_{\tilde{L}_{s}}, \dots , c_{\tilde{L}_{0}} \right)$.}
\end{lemma}

\begin{proof}
\red{
First we note that it is sufficient to consider the case where $x_0 \in M^{\circ}$. 
Indeed, if $x_0 \notin M^{\circ}$ we can pick a positive integer $\beta \in \Z_{>0}$ such that $\beta x_{0}\in M^{\circ}$. 
We then construct an auxiliary broken line $\beta \gamma$ by dilating each domain of linearity of $\gamma$ by $\beta$, while also scaling the corresponding exponent vectors by $\beta$.
Finally, observe that the support of the piecewise linear segment obtained by applying \thref{algorithm_1} to $\beta \gamma$ and the pair of positive integers $(\beta a, \beta b)$ coincides with the support of the piecewise linear segment obtained by applying the algorithm to $\gamma$ and $(a, b)$.
Therefore, from now on we assume $x_0 \in M^{\circ}$. We treat the generic and non-generic cases separately.}

\textbf{Case 1}: Suppose that $\gamma $ is generic.
As in \thref{not:numbering}, for $i \in \{1, \dots ,s \} $ we index the collection of walls containing $x_i$ by $\mathfrak{J}_i$, and we write
\eqn{F_{\mathfrak{J}_i} = \prod_{j \in \mathfrak{J}_i} f_{\wall_j} = 1 +\sum_{k\geq 1} d_k z^{k\, m_{0,i}},}
where $m_{0,i}= p^*(n_{0,i})$ for a primitive $n_{0,i}$ with $x_i \in n_{0,i}^\perp$, and $d_k$ are coefficients in $\Z_{\geq 0}$.\footnote{That coefficients can be taken in $\Z_{\geq 0}$ here is a major result in cluster theory that implies the {\it{positivity of the Laurent phenomenon}}.  See \cite[Theorem~4.10]{GHKK}.}
Since $\gamma$ is a broken line, every bend of $\gamma$ is allowed, and so
\eq{\mono_{i-1}-\mono_{i} = k_i\, m_{0,i},}{eq:bending} 
the exponent vector of a non-zero summand $d_{k_i} z^{^{k_i\, m_{0,i}}}$ of $F_{\mathfrak{J}_i}^{\lra{n_{0,i}, \mono_i}}$.
Our goal is to attach monomials $c_{\tilde{L}_i} z^{\mono_{\tilde{L}_i}}$ to each domain of linearity $\tilde{L}_i$ of $\Supp(\tilde{\gamma})$
in such a way that
$\widetilde{\mono}_{i-1}-\widetilde{\mono}_{i}$
is the exponent vector of a non-zero summand $\tilde{d}_{k_i'} z^{^{k_i'\, m_{0,i}}}$ of $F_{\mathfrak{J}_i}^{\lra{n_{0,i}, \widetilde{\mono}_i}}$.

We claim that each $\widetilde{\mono}_i$ can be taken such that 
\begin{itemize}
    \item $\lra{n_{0,i}, \mono_i}$ divides $ \lra{n_{0,i}, \widetilde{\mono}_i} $
    \item $\displaystyle{\widetilde{\mono}_{i-1}- \widetilde{\mono}_{i} = \frac{\lra{n_{0,i}, \widetilde{\mono}_i}}{\lra{n_{0,i}, \mono_i}} \lrp{\mono_{i-1}-\mono_i} }$.
\end{itemize}
Since coefficients are non-negative, no cancellations will occur when we multiply the functions $f_{\wall_{j}}$, $j\in \mathfrak{J}_i$, and it follows that $\widetilde{\mono}_{i-1}-\widetilde{\mono}_{i}$ is the exponent vector of a summand of 
\eqn{\lrp{F_{\mathfrak{J}_i}^{\lra{n_{0,i}, \mono_i}}}^{ \frac{\lra{n_{0,i}, \widetilde{\mono}_i}}{\lra{n_{0,i},\mono_i}} } = F_{\mathfrak{J}_i}^{\lra{n_{0,i}, \widetilde{\mono}_i}} . }

Now we proceed to show the claim.
By \thref{velocity},
we must be able to write
\eq{\widetilde{\mono}_i = C_i \lrp{ \frac{\mono_i}{a} - \tilde{x}_i }}{eq:Ci}
for some $C_i\in \Q_{>0}$.
Now define
\[
C_0:= a(a+b)\rho_0(\gamma) \lambda
\]
and in turn 
\[
\widetilde{\mono}_0 := C_0\left( \dfrac{\mono_0}{a} - \tilde{x}_0\right).
\]
\red{Since $x_0 $ is in $M^{\circ}$ we have that $\widetilde{\mono}_0\in \rho_0(\gamma) \lambda M^{\circ}$ and one checks that $\langle n_{0,1}, \mono_{0} \rangle $ divides $\langle n_{0,1},\widetilde{\mono}_{0} \rangle $.}
Now suppose we have defined $\widetilde{\mono}_{i-1}\in \rho_{i-1}(\gamma)\lambda M^{\circ} $ for some $1\leq i \leq s$ and that $\widetilde{\mono}_{i-1}$ satisfies 
\begin{equation}
\label{m'_i-1}
    \widetilde{\mono}_{i-1}=C_{i-1}\left( \dfrac{\mono_{i-1}}{a} - \tilde{x}_{i-1}\right)
\end{equation}
for some $C_{i-1}\in \Z_{>0}$.  
Again, $\langle n_{0,i},\mono_{i-1} \rangle $ divides $\langle n_{0,i},\widetilde{\mono}_{i-1} \rangle $ since $\widetilde{\mono}_{i-1}\in \rho_{i-1}(\gamma)\lambda M^{\circ} $. 
We define
\begin{equation}
\label{C_i}
    C_i:=\dfrac{a\langle n_{0,i},\widetilde{\mono}_{i-1}\rangle}{\langle  n_{0,i},\mono_{i-1} \rangle}
\end{equation}
and thus
\begin{equation}
\label{m'_i}
\widetilde{\mono}_{i}:= \dfrac{a\langle  n_{0,i}, \widetilde{\mono}_{i-1}\rangle}{\langle n_{0,i},\mono_{i-1}\rangle}\left(\dfrac{\mono_i}{a}- \tilde{x}_i \right).
\end{equation}
Observe that by construction $ C_i \in \Z_{>0} $.
We will soon see that $\widetilde{\mono}_i\in \rho_i(\gamma) \lambda M^{\circ}$.
First notice that
\begin{equation}
\label{x'_i}
     \tilde{x}_i=-\dfrac{\langle n_{0,i},\tilde{x}_{i-1}\rangle}{\langle n_{0,i},\widetilde{\mono}_{i-1}\rangle}\widetilde{\mono}_{i-1}+\tilde{x}_{i-1}.
\end{equation}
To see this let $x'$ denote the right hand side of the equation. 
We claim that $x'$ is the intersection point of the line $l_{i-1}$ and the walls $\wall_{i_j}$.
Indeed, $x'$ can be thought of as a vector based at $ \tilde{x}_{i-1}$ that runs parallel to $l_{i-1}$ (recall $\widetilde{\mono}_{i-1}  $ is a multiple of $\dfrac{\mono_{i-1}}{a}-\tilde{x}_{i-1} $ and $l_{i-1}$ passes through $\dfrac{\mono_{i-1}}{a}$ and $\tilde{x}_{i-1}$). 
Clearly $ \langle n_{0,i},x' \rangle = 0 $ which implies that $ x'\in  \wall_{i_j}$. Since $\tilde{x}_i$ is the unique point with this property we have $x'=\tilde{x}_i$. 
Now we compute:
\begin{align*}
   \widetilde{\mono}_i & \overset{(\ref{m'_i})}{=} C_i\left(\dfrac{\mono_i}{a}-\tilde{x}_i\right) 
   \\
    & \underset{(\ref{x'_i})}{\overset{(\ref{eq:bending}) }{=}} C_i\left( \dfrac{\mono_{i-1}-k_i\, m_{0,i}}{a}+ \dfrac{\langle n_{0,i},\tilde{x}_{i-1}\rangle}{\langle n_{0,i},\widetilde{\mono}_{i-1}\rangle}\widetilde{\mono}_{i-1}-\tilde{x}_{i-1}\right)
    \\
    & \overset{(\ref{m'_i-1})}{=} C_i\left(\left(\dfrac{1}{C_{i-1}}+ \dfrac{\langle n_{0,i},\tilde{x}_{i-1}\rangle}{\langle n_{0,i},\widetilde{\mono}_{i-1}\rangle} \right)\widetilde{\mono}_{i-1} - \dfrac{k_i}{a}m_{0,i}\right)
    \\     
    & \overset{(\ref{m'_i-1})}{=} C_i\left( \dfrac{\langle  n_{0,i},\mono_{i-1} \rangle}{a\langle n_{0,i},\widetilde{\mono}_{i-1}\rangle} \widetilde{\mono}_{i-1}- \dfrac{k_i}{a}m_{0,i}\right)
    \\
    & \overset{(\ref{C_i})}{=}   \widetilde{\mono}_{i-1} - k_i \dfrac{\langle n_{0,i},\widetilde{\mono}_{i-1}\rangle}{\langle  n_{0,i},\mono_{i-1} \rangle}m_{0,i} .
\end{align*}
So, we have 
\eqn{\widetilde{\mono}_{i-1} - \widetilde{\mono}_{i} = \dfrac{\lra{ n_{0,i},\widetilde{\mono}_{i-1}}}{\lra{n_{0,i},\mono_{i-1} }} k_i\, m_{0,i} 
= \dfrac{\lra{ n_{0,i},\widetilde{\mono}_{i-1}}}{\lra{n_{0,i},\mono_{i-1} }} \lrp{\mono_{i-1}-\mono_i}}
as claimed.
Finally, since $\widetilde{\mono}_{i-1} \in \rho_{i-1}(\gamma)\lambda M^\circ$, we have that $\langle n_{0,i}, \widetilde{\mono}_{i-1}\rangle \in  \rho_{i-1} (\gamma)\lambda \Z$ and thus $ \dfrac{\langle n_{0,i},\widetilde{\mono}_{i-1}\rangle}{\langle n_{0,i},\mono_{i-1}\rangle}\in \rho_i(\gamma) \lambda\Z $. 
In particular, both $ \widetilde{\mono}_{i-1}$ and $ \dfrac{\langle n_{0,i},\widetilde{\mono}_{i-1}\rangle}{\langle n_{0,i},\mono_{i-1}\rangle}k_i\, m_{0,i}$ belong to $ \rho_i(\gamma)\lambda M^\circ$. (Recall $ \rho_{i-1} (\gamma) M^{\circ} \subset \rho_i(\gamma)  M^{\circ}$.) Therefore, $\widetilde{\mono}_i\in \rho_i(\gamma) \lambda M^{\circ}$. This completes the proof in \text{Case 1}.\vspace{2mm}

\textbf{Case 2}: Suppose $\gamma $ is non-generic. We need to construct a sequence of broken line segments $(\eta_k)_{k \in \mathbb N}$ as in \thref{def:broken} \red{converging to $\tilde{\gamma}$}. 
\red{In order to do so, we first apply \thref{algorithm_1} to the broken line $\gamma_k $ for each $k \in \mathbb N$.
As a result we obtain the support of a generic piecewise linear segment $\Supp(\tilde{\gamma}_k)$ bending at the same walls as $ \gamma$.
Note that in general the domains of linearity of $\tilde{\gamma}_k$ are not parallel to the domains of linearity of $\tilde{\gamma}$.
However, the sequence $(\Supp(\tilde{\gamma}_k))_{k \in \mathbb{N}}$ does converge to $\Supp(\tilde{\gamma})$ and we will use this fact to construct $\Supp(\eta_k)$ from $\Supp(\tilde{\gamma}_k)$.} 
For $i\in \{1, \dots , s \}$ let $ \tilde{x}_{i;k}$ be the intersection point of $\tilde{\gamma}_k$ with the walls $\lrc{\wall_{i_j}}_{j\in \mathfrak{J}_i} $ and set $\tilde{x}_{0;k}:= \dfrac{\gamma_k(0)}{a+b}$. 
Let $\widehat{L}_{i;k}$ be the line segment with endpoints $\tilde{x}_{i;k}$ and $\tilde{x}_i$. 
Since both $\tilde{x}_{i;k}$ and $ \tilde{x}_i$ are contained in $ \wall_{i_j}$, which is convex, it follows that $\widehat{L}_{i;k}\subset \wall_{i_j}$. 
We can actually say something stronger: since $\tilde{x}_{i;k}\notin \Sing(\scat)$, 
if $\tilde{x}_i \notin \Sing(\scat)$ then $\widehat{L}_{i;k}$ is contained in
$\wall_{i_j} \setminus \lrp{ \Sing\lrp{\scat} \cap \wall_{i_j} }$, 
and if $\tilde{x}_i \in \Sing(\scat)$ then
$\widehat{L}_{i;k} \cap \Sing(\scat)= \{ \tilde{x}_i \}$.

We proceed to construct $\Par(\eta_k) $ for arbitrary $k \in \N$. 
Let $ C_0:= a(a+b)\rho_0(\gamma) \lambda$ and let $ {\widetilde{\mono}_0 := C_0\left( \dfrac{\mono_0}{a} - \tilde{x}_0\right)}$. For $i\in \{1, \dots , s \}$ we construct $C_i$ and $\widetilde{\mono}_{i}$ as in \text{Case 1}. 
These are going to be the exponent vectors of the monomials associated to the domains of linearity of the $\eta_k$'s. 
To describe $\Supp(\eta_k)$ we first make a few observations. 
By construction the points $\dfrac{\mono_{i-1}}{a}, \tilde{x}_{i;k}$ and $\tilde{x}_{i-1;k}$ 
(resp. $\dfrac{\mono_{i-1}}{a}, \tilde{x}_{i}$ and $\tilde{x}_{i-1}$) are colinear for all $ i \in \{1, \dots , s \}$ (see Figure~\ref{fig:generic_approx}). 
Let $\widehat{\Pi}_{i-1;k}$ be the plane containing $\dfrac{\mono_{i-1}}{a} , \tilde{x}_{i;k}$ and $\tilde{x}_{i}$. 
The line segments $\widehat{L}_{i;k}$ and $\widehat{L}_{i-1;k}$ are contained in the plane $\widehat{\Pi}_{i-1;k}$ as the endpoints of the segments are contained therein.  
Recall that we have denoted by $l_{i-1}$ the line passing through $\dfrac{\mono_{i-1}}{a}$ and $\tilde{x}_{i-1}$, hence $l_{i-1} \subset \widehat{\Pi}_{i-1;k}$. 
As shown in Figure~\ref{fig:generic_approx}, take the triangle defined by the points $\dfrac{\mono_{i-1}}{a}$, $\tilde{x}_{i-1}$, and $\tilde{x}_{i-1;k}$.
This contains the triangle defined by the points $\dfrac{\mono_{i-1}}{a}$, $\tilde{x}_{i}$, and $\tilde{x}_{i;k}$.
Observe that any line in $\widehat{\Pi}_{i-1;k}$ parallel to the edge $l_{i-1} $ and intersecting the edge $\widehat{L}_{i;k}$
must also intersect the edge $\widehat{L}_{i-1;k}$.
Moreover, if this line misses the point $\tilde{x}_i$, it must also miss $\tilde{x}_{i-1} $ and hence 
must intersect $\wall_{{i-1}_j}\setminus \Sing(\scat)$. 

First we construct the domain of linearity $ L^{(\eta_k)}_{s-1}$ associated to $\eta_k$: let $\hat{l}_{s-1;k}$ be the line parallel to $\tilde{L}_{s-1}$ passing through $\tilde{x}_{s;k}$.
Let $ y_{s-1;k}$ be the intersection point of $\hat{l}_{s-1;k}$ and $\widehat{L}_{s-1;k}$. 
Define $ L^{(\eta_k)}_{s-1}$ to be the line segment whose endpoints are $\tilde{x}_{s;k}$ and $y_{s-1;k}$.
Now assume we have defined a point $y_{i;k} \in \widehat{L}_{i;k}\setminus \{ \tilde{x}_i \}$ for some $i \in \{ s-1, \dots, 1\}$. 
We define the domain of linearity $ L^{(\eta_k)}_{i-1}$ as follows: let $\hat{l}_{i-1;k}$ be the line parallel to $\tilde{L}_{i-1}$ passing through $y_{i;k}$.
Let $ y_{i-1;k}$ be the intersection point of $\hat{l}_{i-1;k}$ and $\widehat{L}_{i-1;k}$. 
Define $ L^{(\eta_k)}_{i-1}$ to be the line segment whose endpoints are $y_{i;k}$ and $y_{i-1;k}$. 
\red{As we want $\eta_k$ to be generic, in the event that $y_{0;k}\in \Supp\lrp{\scat}$, we will lengthen or shorten $L^{(\eta_k)}_0$ by a small irrational amount, thereby preventing the endpoint of $\eta_k$ from lying in a wall.}
Finally, the ray $\tilde{x}_{s;k}+\R_{\geq 0}\widetilde{\mono}_s$ is parallel to to the domain of linearity $\tilde{L}_s$ of $\tilde{\gamma}$. 
Therefore, we can choose a point $y_{s+1;k}$ in this ray whose distance to $\dfrac{\mono_s}{a}$ is less or equal to the distance from $\tilde{x}_{s;k}$ to $\tilde{x}_s$.
We let $L^{(\eta_k)}_s$ be the line segment whose endpoints are $\tilde{x}_{s;k}$ and $y_{s+1;k}$.
We define
\[
\Par(\eta_k)=\left((L^{(\eta_k)}_s,\widetilde{\mono}_s),\dots ,(L^{(\eta_k)}_0,\widetilde{\mono}_0) \right).
\]
By construction each domain of linearity $ L^{(\eta_k)}_i$ has slope $\widetilde{\mono}_i$, all bendings are allowed and the broken line segment $\eta_k$ is generic and completely determined by setting $c_{L^{(\eta_k)}_s}=1$. 
Moreover, since the sequence $(\tilde{x}_{s;k})_{k \in \N}$ converges to $\tilde{x}_s$ then the sequence of  broken line segment $(\eta_k)_{k \in \N}$ converges to $ \tilde{\gamma}$. Indeed, if $y_{i;k}$ converges to $\tilde{x}_{i} $ then $y_{i-1;k}$ converges to $\tilde{x}_{i-1}$ for all $i \in \{ s, \dots , 1 \}$ (taking $\tilde{x}_{s;k}=y_{s;k}$).

\begin{figure}[!htbp]
    \centering
\begin{tikzpicture}[scale=1]

\draw[fill,rounded corners, blue,opacity=.19] (-3,-1) -- (5,-1) -- (5,4) -- (-3,4);
\draw[-,line width=0.19mm] (-2,0) -- (3,0) -- (4,3) -- (-2,0);
\node at (-2,0) {$\bullet$};
\node at (-2.3,-0.4) {$\dfrac{\mono_{i-1}}{a}$};
\node at (3.3,-0.3) {$\tilde{x}_{i-1}$};
\node at (4.3,3.3) {$\tilde{x}_{i-1;k}$};
\draw[-,line width=0.19mm] (1,0) -- (2,2);
\node at (2,2.3) {$\tilde{x}_{i;k}$};
\draw[-, magenta,line width=0.35mm] (5/3,4/3) -- (31/9,4/3);
\draw[-, magenta,line width=0.35mm] (1,0) -- (3,0);
\draw[-, magenta,line width=0.35mm] (2,2) -- (4,3);
\node at (3,0) {$\bullet$};
\node at (2,2) {$\bullet$};
\node at (4,3) {$\bullet$};
\node at (5/3,4/3) {$\bullet$};
\node at (1.2,1.1) {$y_{i;k}$};
\node at (2,-0.3) {${\color{magenta}\tilde{L}_{i-1}}$};
\node at (2.6,1.65) {${\color{magenta}L^{(\eta_k)}_{i-1}}$};
\node at (31/9,4/3) {$\bullet$};
\node at (4,1.1) {$y_{i-1;k}$};
\node at (1,0) {$\bullet$};
\node at (0.7,-0.3) {$\tilde{x}_i$};
\node at (1.7,0.5) {$\widehat{L}_{i;k}$};
\node at (3.8,0.3) {$\widehat{L}_{i-1;k}$};
\node at (-0.5,-0.3) {$l_{i-1}$};
\node at (-2,3.3) {\huge $\widehat{\Pi}_{i-1;k}$};

\end{tikzpicture}
\caption{Construction of the line segment $ L_{i-1}^{(\eta_k)}$ of Case 2 in the proof of \thref{key_lemma}. This segment is contained in the line $\hat{l}_{i-1;k}$ passing through $y_{i;k}$ and parallel to $l_{i-1}$.}
\label{fig:generic_approx}
\end{figure}

\end{proof}

\red{While we have allowed $x_0$ to be any rational point in \thref{key_lemma}, we are primarily interested in the integral case $x_0 \in M^\circ$.
This is because the broken line segments we will construct in \thref{prop:bl2jp} are intrinsically related to structure constants of theta function multiplication.
Using broken lines (as opposed to broken line segments), 
pairs of broken lines $\lrp{\gamma^{(1)},\gamma^{(2)}}$ balanced at a point $x_0$ determine the structure constant $\alpha \lrp{I(\gamma^{(1)}), I(\gamma^{(1)}), x_0}$.
By definition, if $\lrp{\gamma^{(1)},\gamma^{(2)}}$ is balanced at $x_0$, then $x_0$ must be in $M^\circ$.}

\subsection{The main construction} \label{sec:mainconstruction}

Given a broken line $\gamma$  and a pair of integers $(a,b)$, we constructed in \thref{key_lemma} a broken line segment $\tilde{\gamma}:[\tau,0 ]\to M^{\circ}_{\R}$ with one endpoint $\tilde{\gamma}(0) = \frac{x_0}{a+b}$ and the other $\tilde{\gamma}(\tau) = \frac{I(\gamma)}{a}$ for some $\tau<0$. 
In \thref{prop:bl2jp} we will need an explicit description of $\tau$.
With this in mind, let $t_i$ be the time at which $\tilde{\gamma}$ bends at ray $\ray_i$, {\it{i.e.}} $\tilde{\gamma}\lrp{t_i} = \tilde{x}_i$.

\begin{lemma}\thlabel{lem:tildetime}
For each $i\in \{0, \dots ,s \}$ the time $\tau\in (-\infty, 0)$ satisfies the following equation:
\eq{\frac{\mono_i}{a}- {\tilde{x}_i} = \lrp{t_i -\tau}{\widetilde{\mono}_i}}{eq:T1lem}

\end{lemma}

\begin{proof}

First observe that we can simultaneously tackle the generic and non-generic cases. Indeed, if we have the generic case, then we apply it to the sequence $(\eta_k)_{k \in \N}$ of Case 2 in \thref{key_lemma} to obtain it for non-generic $\tilde{\gamma}$. Now notice that by \eqref{m'_i}, \eqref{eq:T1lem} holds if and only if
\eq{t_i -\tau = \frac{\lra{n_{0,i},\mono_{i-1}}}{a \lra{n_{0,i},\widetilde{\mono}_{i-1}}   }.}{eq:time}
Next, note that 
\eq{\tilde{x}_{i+1}-\tilde{x}_i = \lrp{t_i - t_{i+1}} \widetilde{\mono}_i }{eq:xdiff}
and
\eq{\frac{\mono_{s}}{a}-\tilde{x}_{s} = \lrp{t_{s} - \tau} \widetilde{\mono}_{s}. }{eq:basecase}
Then \eqref{eq:T1lem} holds for $i=s$.
Assume it holds for all $j>i$.
Then we have
\eqn{\lrp{t_i -\tau}\widetilde{\mono}_i &\overset{\phantom{\eqref{eq:time}}}{=} \lrp{t_i -t_{i+1}+t_{i+1}-\tau}\widetilde{\mono}_i\\
&\overset{\phantom{\eqref{eq:time}}}{=}\lrp{t_i -t_{i+1}}\widetilde{\mono}_i +\lrp{t_{i+1}-\tau}\lrp{\widetilde{\mono}_{i+1} + k'_{i+1} m_{0,i+1} }\\
&\overset{\eqref{eq:xdiff}}{=}\tilde{x}_{i+1}-\tilde{x}_{i} + \frac{\mono_{i+1}}{a}-\tilde{x}_{i+1} + \lrp{t_{i+1}-\tau} k'_{i+1} m_{0,i+1}\\
&\overset{\eqref{eq:time}}{=} -\tilde{x}_{i} + \frac{\mono_{i+1}}{a} + \frac{k_{i+1}}{a} m_{0,i+1}\\
&\overset{\eqref{eq:bending}}{\underset{\phantom{\eqref{eq:time}}}{=} } \frac{\mono_i}{a} - \tilde{x}_{i}, 
}
establishing \eqref{eq:T1lem}.
\end{proof}

In what follows we consider many broken lines simultaneously. To avoid confusions the slope of the $i^{\text{th}}$ domain of linearity of a broken line (segment) $\eta $ is denoted by $\mono_i(\eta) $.

\begin{theorem}\thlabel{prop:bl2jp}
\red{Let $(\gamma^{(1)},\gamma^{(2)})$ be a pair of broken lines balanced at $ x_0\in M^{\circ}$ and $(a,b)$ a pair of positive integers.} 
Then there exist a broken line segment $\tilde{\gamma}:\lrb{0,T}\to M^\circ_\R$ such that $\tilde{\gamma}(0)=\dfrac{I(\gamma^{(1)})}{a}$, $\tilde{\gamma}(T)=\dfrac{I(\gamma^{(2)})}{b}$, 
and $\tilde{\gamma}\lrp{\dfrac{b}{a+b} T}= \dfrac{x_0}{a+b}=:\tilde{x}_0$.
\end{theorem}

\begin{proof}
For $\epsilon \in \{ 1, 2 \}$ let 
$ \lrc{\wall^{(\epsilon)}_{{s_{\epsilon}}_j}}, \dots , \lrc{\wall^{(\epsilon)}_{1_j}}$ 
be the walls in which  $\gamma^{(\epsilon)}$ bends; 
$ L^{(\epsilon)}_{s_{\epsilon}}, \dots ,L^{(\epsilon)}_{0} $ the domains of linearity of $\gamma^{(\epsilon)}$ and 
$ x^{(\epsilon)}_{s_{\epsilon}}, \dots , x^{(\epsilon)}_0$
the bending points. Since we will consider many broken lines we will denote by $ \mono_i(\eta)$ the exponent vector associated to the $i^{\text{th}}$ domain of linearity of the broken line (segment) $\eta$. The case in which $\frac{\mono_0(\gamma^{(1)})}{a} = \tilde{x}_0$ or $\frac{\mono_0(\gamma^{(2)})}{b} = \tilde{x}_0$ reduces to \thref{key_lemma}. So let us assume we are not in these cases.

\textbf{Case 1}: Suppose $\gamma^{(1)} $ and $ \gamma^{(2)}$ are generic. We apply \thref{key_lemma} with $ \lambda = \rho_0(\gamma^{(2)})$ to construct a broken line segment $\tilde{\gamma}^{(1)}:[\tau_1,0] \to M^{\circ}_{\R}$ with the following properties:
\begin{enumerate}
    \item $\tilde{\gamma}^{(1)}(0)= \tilde{x}_0$;
    \item $\tilde{\gamma}^{(1)}(\tau_1)= \dfrac{I(\gamma^{(1)})}{a}$;
    \item \label{mono01}$\mono_0(\tilde{\gamma}^{(1)})=a(a+b)\rho_0(\gamma^{(1)})\rho_0(\gamma^{(2)})\left(\dfrac{\mono_0(\gamma^{(1)})}{a}-\tilde{x}_0\right)$.
\end{enumerate}

Similarly, we use \thref{key_lemma} with $ \lambda = \rho_0(\gamma^{(1)})$ to construct a broken line segment ${\tilde{\gamma}^{(2)}:[\tau_2,0] \to M^{\circ}_{\R}}$ with the following properties:
\begin{enumerate}
    \item $\tilde{\gamma}^{(2)}(0)= \tilde{x}_0$;
    \item $\tilde{\gamma}^{(2)}(\tau_2)= \dfrac{I(\gamma^{(2)})}{b}$; 
    \item \label{mono02}$\mono_0(\tilde{\gamma}^{(2)})=b(a+b)\rho_0(\gamma^{(1)})\rho_0(\gamma^{(2)})\left(\dfrac{\mono_0(\gamma^{(2)})}{b}-\tilde{x}_0\right)$.
\end{enumerate}

We claim that the broken line segments $\tilde{\gamma}^{(1)} $ and $\tilde{\gamma}^{(2)}$  can be glued together to obtain the broken lines segment $\tilde{\gamma}$ connecting $\dfrac{I(\gamma^{(1)})}{a}$ with $\dfrac{I(\gamma^{(2)})}{b}$.
The idea is that we can extend $ \tilde{\gamma}^{(1)}$ by traveling along $\tilde{\gamma}^{(2)}$ in the opposite direction (using \thref{lem:reverse}) in order to hit $\dfrac{I(\gamma^{(1)})}{a}$. With this in mind let us first describe the support of $\tilde{\gamma} $. 
Since the pair $(\gamma^{(1)},\gamma^{(2)})$ is balanced, we have that $\mono_0(\gamma^{(1)})+\mono_0(\gamma^{(2)})=x_0$. 
This implies that $\mono_0(\tilde{\gamma}^{(1)})=-\mono_0(\tilde{\gamma}^{(2)})$. 
In particular, the three different points $\tilde{x}^{(1)}_1,\tilde{x}_0$ and $\tilde{x}^{(2)}_1$ are co-linear. 
Let $L_0$ be the line segment whose endpoints are $\tilde{x}^{(1)}_1 $and $\tilde{x}^{(2)}_1$. Let
\[
\text{Supp}(\tilde{\gamma}):=\bigcup_{i=1}^{s_1}\tilde{L}^{(1)}_{i}\cup L_0 \cup \bigcup_{i=1}^{s_2}\tilde{L}^{(2)}_{i}.
\]
We now attach a monomial to every domain of linearity in $ \Supp(\tilde{\gamma})$. 
Let $ c_{\tilde{L}^{(1)}_i}z^{\mono_i(\tilde{\gamma}_1)}$ be the monomial of $\tilde{\gamma}^{(1)}$ associated to $\tilde{L}^{(1)}_i $ for $ 0 \leq i \leq s_1$. 
Similarly, let $ c_{\tilde{L}^{(2)}_i}z^{\mono_i(\tilde{\gamma}_i)}$ be the monomial of $ \tilde{\gamma}^{(2)}$ associated to $\tilde{L}^{(2)}_i $ for $1 \leq i \leq s_2$. 
\begin{itemize}
    \item  The monomial of $ \tilde{\gamma}$ attached to $\tilde{L}^{(1)}_i$ for $ i\in \{ 1,\dots , s_1\}$ is the same as the one attached to $\tilde{\gamma}^{(1)}$:
\[
c_{\tilde{L}^{(1)}_i}z^{\mono_i(\tilde{\gamma}^{(1)})}.
\]
\item The monomial of $ \tilde{\gamma}$ attached to $L_0$ is also the same as the one attached to $ \tilde{\gamma}^{(1)}$ in $\tilde{L}_0^{(1)}$: 
\[
c_{\tilde{L}^{(1)}_0}z^{\mono_0(\tilde{\gamma}^{(1)})}
\] 
\item The monomial of $ \tilde{\gamma}$ attached to the domain of linearity $\tilde{L}^{(2)}_i$ for $i\in \{ 1,\dots , s_2\}$ is of the form
\[
c_{i}z^{-\mono_i(\tilde{\gamma}^{(2)})}
\]
for some $c_i \in \Z_{>0}$. More precisely, observe that the numbers $c_1, \dots , c_{s_2}$ are completely determined by specifying that  $c_{\tilde{L}^{(2)}_0}z^{\mono_0(\tilde{\gamma}^{(2)})}$ is the monomial of $\tilde{\gamma}$ attached to $L_0$ and the fact that broken line segments can be traversed in both directions (see \thref{lem:reverse} and \thref{determined}). 
So we travel along $\tilde{\gamma}^{(2)}$ is opposite direction with this new monomial associated to $\tilde{L}^{(2)}_0$.
\end{itemize}
This information completely describes 
the broken line segment $\tilde{\gamma}:[0,T] \to M^{\circ}_{\R}$. 
It only remains to show that for this choice of parametrization the equation 
\[
\tilde{\gamma}\lrp{\frac{b}{a+b} T}= \tilde{x}_0
\]
holds. Let $\tilde{\gamma}^{(1)}:[\tau_1, 0]\to M^{\circ}_{\R}$ and $\tilde{\gamma}^{(2)}:[\tau_2, 0]\to M^{\circ}_{\R}$ be parametrizations for $\tilde{\gamma}^{(1)}$ and $\tilde{\gamma}^{(2)}$. The parametrizations we have chosen imply that 
\eqn{\tilde{\gamma}(t) = \begin{cases}
\tilde{\gamma}^{(1)}\lrp{\tau_1+t}  & \text{for }t\in 
\lrb{0,-\tau_1}\\
\tilde{\gamma}^{(2)}\lrp{-\tau_1-t}  & \text{for }t\in 
\lrb{-\tau_1,T}.
\end{cases}
}
In particular, $\tilde{\gamma}\lrp{-\tau_1} = \frac{x_0}{a+b}$ and $T= -\tau_1-\tau_2$.
We need to verify that $-\tau_1 = \frac{b}{a+b}T$. 
On the one hand $ \mono_0(\tilde{\gamma}^{(1)})= a(a+b)\rho_0(\gamma^{(1)})\rho_0(\gamma^{(2)}) \left( \dfrac{\mono_0(\gamma^{(1)})}{a}-\tilde{x}_0\right)
$.
On the other hand, \thref{lem:tildetime} applied to $\tilde{\gamma}^{(1)}$ asserts that $\left( \dfrac{\mono_0(\gamma^{(1)})}{a}-\tilde{x}_0\right)=-\tau_1 \mono_0(\tilde{\gamma}^{(1)})$.
Since $ \dfrac{\mono_0(\gamma^{(1)})}{a}\neq \tilde{x}_0$ we have that
\[
-\tau_1=\dfrac{1}{a(a+b)\rho_0(\gamma^{(1)})\rho_0(\gamma^{(2)})}.
\] 
An analogous computation gives 
\[
-\tau_2=\dfrac{1}{b(a+b)\rho_0(\gamma^{(1)})\rho_0(\gamma^{(2)})}.
\] 
Together these facts give $\dfrac{-\tau_1}{T}=\dfrac{-\tau_1}{-\tau_1-\tau_2} = \dfrac{b}{a+b}$ as desired.

\textbf{Case 2}: Suppose $\gamma^{(1)} $ or $ \gamma^{(2)}$ is non-generic. 
We treat here the case in which both are non-generic, from which it is easy to recover the case in which only one is non-generic.
We use \thref{key_lemma} with $ \lambda = \rho_0(\gamma^{(2)})$ to construct a non-generic broken line segment $\tilde{\gamma}^{(1)}$, together with the required sequence of broken line segments $(\eta^{(1)}_k)_{k \in \N}$ whose limit is $\tilde{\gamma}^{(1)}$. 
Similarly, we use \thref{key_lemma} with $ \lambda = \rho_0(\gamma^{(1)})$ to construct a non-generic broken line segment $\tilde{\gamma}^{(2)}$ and a sequence of broken line segments $(\eta^{(2)}_k)_{k \in \N}$ limiting to $\tilde{\gamma}^{(2)}$.
Our choices of $\lambda$ were taken to ensure that $\mono_0(\tilde{\gamma}^{(1)})=-\mono_0(\tilde{\gamma}^{(2)}) $. 
Now we would like to proceed as in Case 1 and glue the broken line segments $\tilde{\gamma}_1$ and $\tilde{\gamma}_2$ to produce the desired broken line segment $\tilde{\gamma}$.
Since $\mono_0(\tilde{\gamma}^{(1)})=-\mono_0(\tilde{\gamma}^{(2)}) $ we have that $\Supp(\tilde{\gamma}^{(1)})$ and $\Supp(\tilde{\gamma}^{(2)})$ can indeed be glued to obtain $\Supp(\tilde{\gamma})$. 
However, we still have to define a sequence of broken line segments $(\zeta_k)_{k \in \N}$ whose limit is $\tilde{\gamma}$.
Once again it would be desirable to proceed as in Case 1 and glue the broken line segments of the sequence $(\eta^{(1)}_k)_{k \in \N}$ with those of the sequence $(\eta^{(2)}_k)_{k \in \N}$ to obtain $(\zeta_k )_{k\in \N}$.
The problem that arises here is that the endpoint of $\eta^{(1)}_k$ might not coincide with the endpoint of $\eta^{(2)}_k$ preventing us from gluing these broken line segments. The idea to overcome this problem is to extend $\eta^{(1)}_k$ by running parallel to $\eta^{(2)}_k $ in order to construct $\zeta_k$.

Let $y^{(2)}_{i;k} $ be the intersection point of $\eta_k^{(2)}$ with the walls $\lrc{\wall_{i_j}^{(2)}}_{j\in \mathfrak{J}_i^{(2)}}$ for $i \in \{ 1, \dots , s_2\}$ and set ${y^{(\epsilon)}_{0;k}:=  \eta^{(\epsilon)}_k(0) }$ for $\epsilon \in \{ 1,2\}$.
Observe that  $\tilde{x}_0,y^{(1)}_{0;k}$ and $ y^{(2)}_{0;k}$ are colinear. 
To see this we need to focus our attention on the sequence of pairs $(\gamma^{(1)}_{k}$,$\gamma^{(2)}_{k})_{k \in \N}$ of broken lines balanced near $x_0$ converging to $(\gamma^{(1)}, \gamma^{(2)})$. 
By construction both $y^{(1)}_{0;k}$ and $ y^{(2)}_{0;k}$ lie in the line segment whose endpoints are $\tilde{x}_0$ and $\dfrac{\gamma^{(1)}_k(0)}{a+b}=\dfrac{\gamma^{(2)}_k(0)}{a+b}$. 
So in the problematic case $y^{(1)}_{0;k} \neq y^{(2)}_{0;k} $ we can assume without loss of generality that $ y^{(1)}_{0;k}$ is closer to $\tilde{x}_0$ than $ y^{(2)}_{0;k}$.
So by construction the line passing through $ y^{(1)}_{0;k}$ with slope $\mono_0(\tilde{\gamma}^{(1)}) $ intersects the line segment whose endpoints are  $y^{(2)}_{1;k}$ and $ \tilde{x}^{(2)}_1$ in a unique point $w_{1;k}$. Moreover, we see that
$ w_{1;k} \in \wall_{1_j}^{(2)} \setminus \lrp{\Sing(\scat)\cap \wall_{1_j}^{(2)}}$ (using the same argument as in Case 2 of \thref{key_lemma}).
Assume that for $i \in \{ 1, \dots , s_{2-1} \}$ we have defined a point $w_{i;k}\in \wall^{(2)}_{i_j}\setminus \lrp{\Sing(\scat)\cap\wall^{(2)}_{i_j}}$ contained in the line segment whose endpoints are $ y^{(2)}_{i;k}$ and $ \tilde{x}^{(2)}_i$. 
We define $w_{i+1;k}$ to be the intersection point of the line with slope $\mono_i(\tilde{\gamma}^{(2)})$ passing through $ w_{i;k} $ and the line segment whose endpoints are $y^{(2)}_{i+1;k}$ and $\tilde{x}^{(2)}_{i+1} $.
Finally, the ray $w_{s_2;k}+\R_{\geq 0}\mono_{s_2}(\tilde{\gamma}^{(2)})$ is parallel to to the domain of linearity $\tilde{L}^{(2)}_{s_2}$ of $\tilde{\gamma}^{(2)}$. 
Therefore, we can choose a point $w_{s_2+1;k}$ in this ray whose distance to $\dfrac{\mono_{s_2}(\tilde{\gamma}^{(2)})}{b}$ is less or equal to the distance from $w_{s_2;k}$ and $ \tilde{x}_{s_2}^{(2)}$.

We are ready to describe $\Supp(\zeta_k) $. 
For $i \in \{1, \dots , s_2 \}$ let $L_{2,i}^{(\zeta_k)} $ be the line segment whose endpoints are $ w_{i;k}$ and $w_{i+1;k}$. 
For $i\in \{ 1, \dots , s_1 \}$ let $L_{1,i}^{(\zeta_k)}$ coincide with the $i^{\text{th}}$ domain of linearity  $L^{(\eta_k^{(1)})}_i$ of $\eta_k^{(1)}$. Let $L_0^{(\zeta_k)}$ be the line segment whose endpoints are $ w_{1;k}$ and the intersection point of $\eta^{(1)}_k $ with $\wall^{(1)}_{1_j}$. Define
\[
\Supp(\zeta_k)= \bigcup_{i=1}^{s_1} L_{1,i}^{(\zeta_k)} \cup L_0^{(\zeta_k)} \cup \bigcup_{i=1}^{s_2} L_{2,i}^{(\zeta_k)}.
\]
By construction, every bending of $\Supp(\zeta_k)$ is allowed. The monomial associated to $L^{(\zeta_k)}_{1,i} $ is the same as the monomial associated to the domain of linearity $L^{(\eta_k^{(1)})}_i$ for $\eta_k^{(1)}$ for $i \in \{ 1, \dots , s_1\}$. 
The monomial associated to $L_0^{(\zeta_k)} $ is the same as the monomial associated to the domain of linearity $L^{(\eta_k^{(1)})}_0$ for $\eta_k^{(1)}$. Since the bendings concerning the domains of linearity of the form $\bigcup_{i=1}^{s_2} L_{2,i}^{(\zeta_k)}$ are allowed, their monomials are completely  determined by the the monomial associated to $L_0^{(\zeta_k)}$.
If we have a parametrization $\tilde{\gamma}:[0,T]\to M^{\circ}_{\R} $ the fact $ \tilde{\gamma}\left( \dfrac{b}{a+b}T\right)= \tilde{x}_0$ follows from Case 1 applied by to parametrizations of the form $\zeta_k:[0,T_k]\to M^{\circ}_{\R}$ (notice that we must have $T_k$ converge to $T$).
\end{proof}

\section{Constructing balanced broken lines from a broken line segment}\label{sec:in_alg}

In the previous section we began with a pair of broken lines $(\gamma^{(1)},\gamma^{(2)})$ \red{balanced at $x_0 \in M^\circ$} and a pair of integers $(a,b)$.
From this data, we constructed a broken line segment
$\tilde{\gamma}$ connecting $\dfrac{p}{a}$ to $\dfrac{q}{b}$, and passing through $\dfrac{x_0}{a+b}$, where $I(\gamma^{(1)}) = p$ and  $I(\gamma^{(2)}) = q$. 
In this section we instead start with 
\begin{itemize}
    \item a broken line segment $\tilde{\gamma}=\lrp{\Par(\tilde{\gamma});c_s,\dots,c_0}$ parametrized as $\tilde{\gamma}:[0,T] \to M^\circ_\R$ and having endpoints $\tilde{p}:=\tilde{\gamma}(0)$ and $\tilde{q}:= \tilde{\gamma}(T)$ in $M^\circ_{\Q}$, and
    \item a rational point $\tilde{r}:= \tilde{\gamma}(\tau) \in M^\circ_{\Q}$ for some $\tau\in \lrp{0,T}\cap \Q$. 
\end{itemize}
In this case we will construct a balanced pair of broken lines
$(\gamma^{(1)}, \gamma^{(2)})$
with 
\begin{itemize}
    \item $I(\gamma^{(1)})= a \tilde{p}$, $I(\gamma^{(2)})= b \tilde{q}$ for a pair of positive integers $a$, $b$ satisfying 
    \eq{\frac{b}{a+b}= \frac{\tau}{T},}{eq:t0}
    \item $\mono_0(\gamma^{(1)}) + \mono_0(\gamma^{(2)}) = \lrp{a+b}\tilde{r} $, and
    \item $\gamma^{(1)}(0) = \gamma^{(2)}(0) = \lrp{a+b}\tilde{r}$.  
\end{itemize}
Roughly speaking, at the level of supports, 
this ``inverts'' the construction of \thref{prop:bl2jp}.
We say {\it{roughly}} and put scare quotes around {\it{inverts}} because composition of these constructions will 
radially translate the base point $\gamma^{(1)}(0)=\gamma^{(2)}(0)$ (or equivalently, apply a dilation to $\Supp(\gamma^{(\epsilon)})$). 
 
From now on we fix a broken line segment $\tilde{\gamma}:[0,T]\to M^{\circ}_{\R}$ with endpoints $\tilde{\gamma}(0)=\tilde{p}$ and $\tilde{\gamma}(T)=\tilde{q}$, 
and let $\tilde{r}$ be a rational point in the interior of this broken line segment. 
Moreover, for non-generic $\tilde{\gamma}$ we fix a sequence  $(\tilde{\gamma}_k)_{k \in \N}$ of generic broken line segments limiting to $\tilde{\gamma}$, as in \thref{def:broken}.
As in \thref{not:numbering}, we let $\tilde{x}_1, \dots , \tilde{x}_s$ be the bending points of $\tilde{\gamma}$
and denote the time of this bend by $t_i$, {\it{i.e.}} $\tilde{\gamma}(t_i) = \tilde{x}_i$.
Recall that if $\tilde{\gamma}$ is non-generic, we could have $x_i=x_j$ and $t_i=t_j$ for distinct $i$ and $j$.
A depiction of the generic case is provided in Figure~\ref{fig:SegNotationNew}.

\noindent
\begin{minipage}{\linewidth}
\captionsetup{type=figure}
\begin{center}
\begin{tikzpicture}

    \node [inner sep=0] (0) at (0,0) {$\bullet$};

    \draw [name path=walls1] (0)--++(150:5) node [above left] {$\lrc{\wall_{1_j}}$}; 
    \draw [name path=walls1m1] (0)--++(130:5) node [above] {$\lrc{\wall_{2_j}}$}; 
    \draw [name path=wall11] (0)--++(100:5) node [above] {$\lrc{\wall_{k_j}}$}; 
    \draw[->] (0)--++(100:4.5) node [right] {$\ray_k $ }; 
    \draw [name path=wall12] (0)--++(75:5) node [above] {$\lrc{\wall_{{k+1}_j}}$}; 
    \draw [name path=walls2m1] (0)--++(40:5)node [above right] {$\lrc{\wall_{{s-1}_j}}$}; 
    \draw [name path=walls2] (0)--++(20:5)node [right] {$\lrc{\wall_{s_j}}$}; 

    \path (130:4.8)--(100:4.8) node [midway, sloped] {$\cdots$};
    \path (75:4.8)--(40:4.8) node [midway, sloped] {$\cdots$};

    \node [inner sep=0] (p) at (170:4) {$\bullet$}; 
    \node [inner sep=0] (q) at (5:4.35) {$\bullet$}; 

    \path [name path=seg1] (p) --++(30:2);
    \path [name intersections={of=seg1 and walls1,by=xs1}];

    \path [name path=seg2] (xs1) --++(75:2);
    \path [name intersections={of=seg2 and walls1m1,by=xs1m1}];

    \path [name path=seg3] (xs1m1) --++(40:3.5);
    \path [name intersections={of=seg3 and wall11,by=x11}];

    \coordinate (l1) at ($(xs1m1)!.25!(x11)$);
    \coordinate (l2) at ($(xs1m1)!.75!(x11)$);

    \path [name path=seg4] (x11) --++(-10:3.5);
    \path [name intersections={of=seg4 and wall12,by=x12}];

    \path [name path=seg5] (x12) --++(-45:5);
    \path [name intersections={of=seg5 and walls2m1,by=xs2m1}];

    \coordinate (r1) at ($(x12)!.25!(xs2m1)$);
    \coordinate (r2) at ($(x12)!.75!(xs2m1)$);

    \path [name path=seg6] (xs2m1) --++(-30:3.5);
    \path [name intersections={of=seg6 and walls2,by=xs2}];

    \node [inner sep=0] (x0) at ($(x11)!.65!(x12)$) {$\bullet$};

    \draw[thick, color=magenta]  (p)--(xs1) node [midway, sloped, above] {$\tilde{L}_{0}$}--(xs1m1) node [pos=.6, sloped, above] {$\tilde{L}_{1}$}--(l1);
    \path (l1)--(l2) node [midway, sloped, color=magenta] {$\cdots$};
    \draw[thick, color=magenta]  (l2)--(x11)--(x0);
    \draw[thick, color=magenta] (x0)--(x12)--(r1);
    \path (x11)--(x12) node [midway, sloped, below, color=magenta] {$\tilde{L}_k$};
    \path (r1)--(r2) node [midway, sloped, color=magenta] {$\cdots$};
    \draw[thick, color=magenta]  (r2)--(xs2m1)--(xs2) node [pos=.3, sloped, below] {$\tilde{L}_{s-1}$}--(q) node [pos=.6, left] {$\tilde{L}_{s}$};
        
    \path (p)--++(-160:.3) node {$\tilde{p}$};
    \path (q)--++(-10:.3) node {$\tilde{q}$};
    \path (xs1)--++(-90:.4) node {$\tilde{x}_{1}$};
    \path (xs1m1)--++(-5:.5) node {$\tilde{x}_{2}$};
    \path (x11)--++(165:.4) node {$\tilde{x}_{k}$};
    \path (x0)--++(90:.3) node {$\tilde{r}$};
    \path (x12)--++(10:.55) node {$\tilde{x}_{k+1}$};
    \path (xs2m1)--++(5:.7) node {$\tilde{x}_{s-1}$};
    \path (xs2)--++(-15:.5) node {$\tilde{x}_{s}$};

    \node [circle, fill, inner sep=1pt] at (xs1) {};
    \node [circle, fill, inner sep=1pt] at (xs1m1) {};
    \node [circle, fill, inner sep=1pt] at (x11) {};
    \node [circle, fill, inner sep=1pt] at (x12) {};
    \node [circle, fill, inner sep=1pt] at (xs2m1) {};
    \node [circle, fill, inner sep=1pt] at (xs2) {};

\end{tikzpicture}


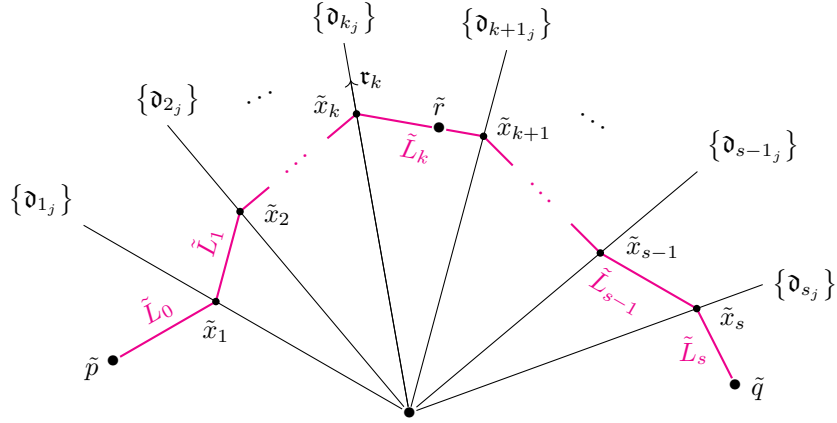
\captionof{figure}{\label{fig:SegNotationNew} Broken line segment $\tilde{\gamma}$ from $\tilde{p}$ to $\tilde{q}$, with interior point $\tilde{r}$.  Depicted for generic $\tilde{\gamma}$ and $\tilde{r}$.
} 

\end{center}
\end{minipage}

We will use the broken line sub-segment of $\tilde{\gamma}$ having endpoints $\tilde{p}$ and $\tilde{r}$ to construct a broken line $\gamma^{(1)}$ with initial exponent vector a multiple of $\tilde{p}$, 
and we will use the broken line sub-segment of $\tilde{\gamma}$ having endpoints $\tilde{r}$ and $\tilde{q}$ to construct a broken line $\gamma^{(2)}$ with initial exponent vector a multiple of $\tilde{q}$.
The indexing in Figure~\ref{fig:SegNotationNew} will be a bit awkward for describing this pair, so we introduce the following relabelling.

\noindent
\begin{minipage}{\linewidth}
\captionsetup{type=figure}
\begin{center}
\begin{tikzpicture}

    \node [inner sep=0] (0) at (0,0) {$\bullet$};

    \draw [name path=walls1] (0)--++(150:5) node [above left] {$\lrc{\wall_{{s_1}_j}^{(1)}}$}; 
    \draw [name path=walls1m1] (0)--++(130:5) node [above] {$\lrc{\wall_{{s_1-1}_j}^{(1)}}$}; 
    \draw [name path=wall11] (0)--++(100:5) node [above] {$\lrc{\wall_{1_j}^{(1)}}$};  \draw[->] (0)--++(100:4.5);
    \node at (95:4.55) {$\ray_1^{(1)} $ }; 
    \draw [name path=wall12] (0)--++(75:5) node [above] {$\lrc{\wall_{1_j}^{(2)}}$}; 
    \draw [name path=walls2m1] (0)--++(40:5)node [above right] {$\lrc{\wall_{{s_2-1}_j}^{(2)}}$}; 
    \draw [name path=walls2] (0)--++(20:5)node [right] {$\lrc{\wall_{{s_2}_j}^{(2)}}$}; 

    \path (130:4.8)--(100:4.8) node [midway, sloped] {$\cdots$};
    \path (75:4.8)--(40:4.8) node [midway, sloped] {$\cdots$};

    \node [inner sep=0] (p) at (170:4) {$\bullet$}; 
    \node [inner sep=0] (q) at (5:4.35) {$\bullet$}; 

    \path [name path=seg1] (p) --++(30:2);
    \path [name intersections={of=seg1 and walls1,by=xs1}];

    \path [name path=seg2] (xs1) --++(75:2);
    \path [name intersections={of=seg2 and walls1m1,by=xs1m1}];

    \path [name path=seg3] (xs1m1) --++(40:3.5);
    \path [name intersections={of=seg3 and wall11,by=x11}];

    \coordinate (l1) at ($(xs1m1)!.25!(x11)$);
    \coordinate (l2) at ($(xs1m1)!.75!(x11)$);

    \path [name path=seg4] (x11) --++(-10:3.5);
    \path [name intersections={of=seg4 and wall12,by=x12}];

    \path [name path=seg5] (x12) --++(-45:5);
    \path [name intersections={of=seg5 and walls2m1,by=xs2m1}];

    \coordinate (r1) at ($(x12)!.25!(xs2m1)$);
    \coordinate (r2) at ($(x12)!.75!(xs2m1)$);

    \path [name path=seg6] (xs2m1) --++(-30:3.5);
    \path [name intersections={of=seg6 and walls2,by=xs2}];

    \node [inner sep=0] (x0) at ($(x11)!.65!(x12)$) {$\bullet$};

    \draw[thick, color=magenta]  (p)--(xs1) node [midway, sloped, above] {$\tilde{L}_{s_1}^{(1)}$}--(xs1m1) node [pos=.6, sloped, above] {$\tilde{L}_{s_1-1}^{(1)}$}--(l1);
    \path (l1)--(l2) node [midway, sloped, color=magenta] {$\cdots$};
    \draw[thick, color=magenta]  (l2)--(x11)--(x0);
    \draw[thick, color=magenta] (x0)--(x12)--(r1);
    \path (x11)--(x12) node [midway, sloped, below, color=magenta] {$\tilde{L}_0$};
    \path (r1)--(r2) node [midway, sloped, color=magenta] {$\cdots$};
    \draw[thick, color=magenta]  (r2)--(xs2m1)--(xs2) node [pos=.3, sloped, below] {$\tilde{L}_{s_2-1}^{(2)}$}--(q) node [pos=.6, left] {$\tilde{L}_{s_2}^{(2)}$};
        
    \path (p)--++(-160:.3) node {$\tilde{p}$};
    \path (q)--++(-10:.3) node {$\tilde{q}$};
    \path (xs1)--++(-90:.45) node {$\tilde{x}_{s_1}^{(1)}$};
    \path (xs1m1)--++(-5:.65) node {$\tilde{x}_{s_1-1}^{(1)}$};
    \path (x11)--++(165:.4) node {$\tilde{x}_1^{(1)}$};
    \path (x0)--++(90:.3) node {$\tilde{r}$};
    \path (x12)--++(10:.5) node {$\tilde{x}_1^{(2)}$};
    \path (xs2m1)--++(5:.75) node {$\tilde{x}_{s_2-1}^{(2)}$};
    \path (xs2)--++(-15:.6) node {$\tilde{x}_{s_2}^{(2)}$};

    \node [circle, fill, inner sep=1pt] at (xs1) {};
    \node [circle, fill, inner sep=1pt] at (xs1m1) {};
    \node [circle, fill, inner sep=1pt] at (x11) {};
    \node [circle, fill, inner sep=1pt] at (x12) {};
    \node [circle, fill, inner sep=1pt] at (xs2m1) {};
    \node [circle, fill, inner sep=1pt] at (xs2) {};

\end{tikzpicture}


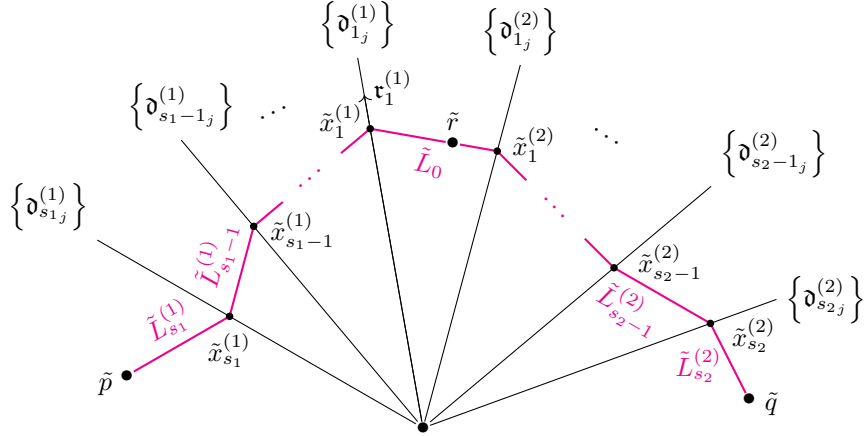
\captionof{figure}{\label{fig:SegNotation} Relabelling of $\tilde{\gamma}$ and bending walls. Compare to Figure \ref{fig:SegNotationNew}.
} 

\end{center}
\end{minipage}
In the case that $\tilde{r}$ is contained in a (not necessarily unique) bending wall, 
we index as if $\tilde{r}$ were replaced by a new point $\tilde{r}_- := \tilde{\gamma}\lrp{\tau - \delta}$ for some $\delta > 0$ sufficiently small that $\tilde{r}$ is in the closure of the domain of linearity containing $\tilde{r}_-$.

Set ${\tilde{\rho_i}}^{(\epsilon)}= \prod_{k=1}^{s_\epsilon-i}\lra{n_{0,k}^{(\epsilon)},{\widetilde{\mono}_{k}}^{(\epsilon)}}$, 
and take $a$, $b$ such that 
$a \tilde{p} \in {\tilde{\rho}_{0}}^{(1)} M^\circ $ and 
$b \tilde{q} \in {\tilde{\rho}_{0}}^{(2)} M^\circ $.
If $\tilde{r}$ lies on a bending wall, we associate the bending wall with side ${}^{(2)}$.

We start by defining the exponent vectors.
We set 
\begin{enumerate}
    \item \label{epsilon1} $\mono_{i}^{(1)}:= a \lrp{\tilde{x}_i^{(1)}+ t_i^{(1)} \widetilde{\mono}_i^{(1)} }$
    \item \label{epsilon2} $\mono_{i}^{(2)}:= b \lrp{\tilde{x}_i^{(2)} - \lrp{T - t_i^{(2)} }\widetilde{\mono}_i^{(2)} }$,
\end{enumerate}
where for the index $0$ we take, {\it{e.g.}} 
$\tilde{x}_0^{(1)}=\tilde{x}_1^{(2)}$ and 
$\tilde{x}_0^{(2)}=\tilde{x}_1^{(1)}$. 

\begin{proposition} \label{prop:velocities}
We have $\mono_{s_1}^{(1)} = a \tilde{p}$, $\mono_{s_2}^{(2)} = b \tilde{q}$, and $\mono_0^{(1)}+\mono_0^{(2)} = \lrp{a+b}\tilde{r}$.
\end{proposition}

\begin{proof}
Since $\widetilde{\mono}_{s_\epsilon}^{(\epsilon)}$ is the negative of the velocity of $\tilde{\gamma}'$ along $\tilde{L}_{s_\epsilon}^{(\epsilon)}$,
we have:
\eqn{\mono_{s_1}^{(1)} &= a \lrp{\tilde{x}_{s_1}^{(1)} +t_{s_1}^{(1)} \widetilde{\mono}_{s_1}^{(1)} }\\
&= a \tilde{p}.}
Similarly,
\eqn{\mono_{s_2}^{(2)} &= b \lrp{\tilde{x}_{s_2}^{(2)} -\lrp{T- t_{s_2}^{(2)}} \widetilde{\mono}_{s_2}^{(2)} }\\
&= b \tilde{q}.}
Next, we compute:
\eqn{
\mono_0^{(1)}+\mono_0^{(2)} &\stackrel{\phantom{\eqref{eq:t0}}}{=} 
a \lrp{\tilde{x}_{1}^{(2)} + t_{1}^{(2)} \widetilde{\mono}_0 } 
+ b \lrp{\tilde{x}_1^{(1)} - \lrp{T - t_{1}^{(1)}}\widetilde{\mono}_0 }\\
&\stackrel{\eqref{eq:t0}}{=} 
  a \lrp{\tilde{x}_{1}^{(2)} + t_{1}^{(2)} \widetilde{\mono}_0 } 
+ b \lrp{\tilde{x}_{1}^{(1)} + t_{1}^{(1)} \widetilde{\mono}_0 }- \lrp{a+b} \tau \widetilde{\mono}_0\\
&\stackrel{\phantom{\eqref{eq:t0}}}{=}
  a \lrp{\tilde{x}_{1}^{(2)} + \lrp{t_{1}^{(2)}-\tau} \widetilde{\mono}_0 } 
+ b \lrp{\tilde{x}_{1}^{(1)} + \lrp{t_{1}^{(1)}-\tau} \widetilde{\mono}_0 }\\
&\stackrel{\phantom{\eqref{eq:t0}}}{=} a \tilde{r}+ b \tilde{r}.
}
\end{proof}

We will now use these exponent vectors to construct a pair $\lrp{\Par\lrp{\gamma^{(1)}}, \Par\lrp{\gamma^{(2)}}}$, 
where 
$\gamma^{(1)}$ bends at the same walls as the broken line sub-segment of $\tilde{\gamma}$ having endpoints $\tilde{p}$ and $\tilde{r}$,
$\gamma^{(2)}$ bends at the same walls as the broken line sub-segment of $\tilde{\gamma}$ having endpoints $\tilde{r}$ and $\tilde{q}$,
and ${\gamma^{(1)}(0)= \gamma^{(2)}(0)= \lrp{a+b}\tilde{r}}$.
(We will deal with non-genericity issues later by perturbing this pair.)
There are two steps.  
First, we can consider a path beginning at $\lrp{a+b}\tilde{r}$ and proceeding with velocity $\mono_0^{(\epsilon)}$. 
Does this path intersect $\ray_1^{(\epsilon)}$?  
If it does, after reaching $\ray_1^{(\epsilon)}$ we can proceed with velocity $\mono_1^{(\epsilon)}$ 
and ask whether the resulting path intersects $\ray_2^{(\epsilon)}$.
Of course, if $\lrp{a+b}\tilde{r}$ is {\emph{already}} in $\ray_1^{(\epsilon)}$, the velocity $\mono_0^{(\epsilon)}$ is simply recorded for a degenerate domain of linearity and we  proceed from $\lrp{a+b}\tilde{r}$  with velocity $\mono_1^{(\epsilon)}$.
We can continue this process until we have either exhausted all exponent vectors $\mono_i^{(\epsilon)}$ or found a bending ray $\ray_i^{(\epsilon)}$ that isn't crossed.
In the former case, we have a parametrized path that crosses all bending walls in order, and we can move on to the second step--
are adjacent exponent vectors $\mono_i^{(\epsilon)}$ and $\mono_{i-1}^{(\epsilon)}$ for this path related by an allowed bending at $\ray_i^{(\epsilon)}$?
We tackle the first step in \thref{lem:path} and the second in \thref{lem:bend}.

\begin{lemma}\thlabel{lem:path}
There is a piecewise linear path $\eta:[0,\infty)$ with endpoint $\eta(0)= \lrp{a+b}\tilde{r}$ intersecting the rays $\ray_k^{(\epsilon)}$ in order and having velocity $\mono_k^{(\epsilon)}$ along the (possibly degenerate) segment between $\ray_{k}^{(\epsilon)}$ and $\ray_{k+1}^{(\epsilon)}$.
The velocity from $\eta(0)$ to $\ray_1^{(\epsilon)}$ is $\mono_0^{(\epsilon)}$, and the velocity in the unbounded domain of linearity is $\mono_{s_\epsilon}^{(\epsilon)}$. 
\end{lemma}

\begin{proof}
First take $\epsilon=1$.
Observe that $a>0$ and $t_i^{(1)}>0$ for all $i$.
Moreover, $\tilde{x}_i^{(1)}$ is in $\ray_i^{(1)}$ 
and $\widetilde{\mono}_{i-1}^{(1)}$ is directed from $\tilde{x}_{i-1}^{(1)}$ toward $\tilde{x}_i^{(1)}$.
It follows that the path intersects $\ray_i^{(1)}$, yielding the claim.
The case of $\epsilon=2$ is similar.
The only subtlety is that the negative sign in front of $\lrp{T-t_i^{(2)}}$ accounts for the path and $\tilde{\gamma}$ crossing $\ray_i^{(2)}$ in opposite directions. 
\end{proof}

In the following lemma we drop the superscript ${}^{(\epsilon)}$ to avoid an aesthetic disaster.
\begin{lemma}\thlabel{lem:bend}
The exponent vectors $\mono_i$ and $\mono_{i-i}$ are related by
\eq{\mono_{i-1}- \mono_{i} = \lrp{\widetilde{\mono}_{i-1}- \widetilde{\mono}_{i} } \frac{\lra{n_{0,i},\mono_{i}}}{\lra{n_{0,i},\widetilde{\mono}_{i}}}. }{mfrommp}
This corresponds to an allowed bending at $\ray_i$.
\end{lemma}

\begin{proof}
Take $\epsilon = 1$ first.
We have
\eq{\tilde{x}_{i-1} = \tilde{x}_{i} - \lrp{t_{i-1} - t_{i} } \widetilde{\mono}_{i-1},
}{eq:xtilde}
so
\eqn{\mono_{i-1}- \mono_{i} &\stackrel{\phantom{\eqref{eq:xtilde}}}{=} a \lrp{ \tilde{x}_{i-1} + t_{i-1} \widetilde{\mono}_{i-1}} - a \lrp{\tilde{x}_{i} + t_{i} \widetilde{\mono}_{i}  }\\
&\stackrel{\eqref{eq:xtilde}}{=} a \lrp{ \tilde{x}_{i} + t_{i} \widetilde{\mono}_{i-1} - \tilde{x}_{i} - t_{i} \widetilde{\mono}_{i}  }\\
&\stackrel{\phantom{\eqref{eq:xtilde}}}{=} a t_i \lrp{\widetilde{\mono}_{i-1} - \widetilde{\mono}_{i}} .}
But 
\eqn{\frac{\lra{n_{0,i}, \mono_i}}{\lra{n_{0,i}, \widetilde{\mono}_i}} &= \frac{\lra{n_{0,i}, a \lrp{\tilde{x}_i+ t_i \widetilde{\mono}_i}}}{\lra{n_{0,i}, \widetilde{\mono}_i}} \\
&= a \frac{\lra{n_{0,i}, t_i \widetilde{\mono}_i}}{\lra{n_{0,i}, \widetilde{\mono}_i}} \\
&= a t_i,}
where the second equality comes from $\lra{n_{0,i}, \tilde{x}_i} = 0$.
This yields \eqref{mfrommp}. 
Since we have taken $a$ such that $a \tilde{p} \in \tilde{\rho}_0^{(1)} M^\circ$, \eqref{mfrommp} ensures that $\mono_i \in \tilde{\rho}_i^{(1)} M^\circ \subset M^\circ$.
Then \eqref{mfrommp} corresponds to an allowed bending at $\ray_i$.

The case $\epsilon = 2$ is proved similarly.
\end{proof}

So, we have a pair $\lrp{\Par\lrp{\gamma^{(1)}},\Par\lrp{\gamma^{(2)}}}$, both having endpoint $\gamma^{(\epsilon)}(0)=\lrp{a+b}\tilde{r}$, where $I(\gamma^{(1)}) = a \tilde{p}$ and $I(\gamma^{(2)}) = b \tilde{q}$.
Furthermore, the final exponent vectors of this pair sum to $\lrp{a+b}\tilde{r}$.
If the pair is generic, then $\lrp{\Par\lrp{\gamma^{(1)}},\Par\lrp{\gamma^{(2)}}}$
is the desired pair of broken lines balanced at $\lrp{a+b}\tilde{r}$.
In the non-generic case, thanks to \thref{lem:bend}, we can take this pair to be the limit of a sequence $\lrp{\Par\lrp{\gamma^{(1)}_k},\Par\lrp{\gamma^{(2)}_k} }_{k\in \N} $ of \red{based} pairs of generic broken lines having common endpoint $x_{0;k}$ on the side of $\lrc{\wall_{1_j}^{(2)}}$ which corresponds to times $t < t_1^{(2)}$.
An element $\lrp{\Par\lrp{\gamma^{(1)}_k},\Par\lrp{\gamma^{(2)}_k} }$ of this sequence is a pair of generic broken lines balanced near $\lrp{a+b}\tilde{r}$, while the limit $\lrp{\Par\lrp{\gamma^{(1)}},\Par\lrp{\gamma^{(2)}}}$ is balanced at $\lrp{a+b}\tilde{r}$.
We summarize the results of this section with the following proposition.

\begin{proposition}\thlabel{prop:jp2bl}
Given the following data:
\begin{itemize}
    \item $\Par\lrp{\tilde{\gamma}}$ the support with exponent vectors of a broken line segment with endpoints $\tilde{\gamma}(0) = \tilde{p}$ and $\tilde{\gamma}(T)=\tilde{q}$ in $M^\circ_{\Q}$
    \item $\tau$ a rational number in $\lrp{0,T}$
    \item $\lrp{a,b}$ a pair of positive integers with $\frac{b}{a+b}=\frac{\tau}{T}$, $a \tilde{p} \in \tilde{\rho}_0^{(1)} M^\circ$, and $b \tilde{q} \in \tilde{\rho}_0^{(2)} M^\circ$
\end{itemize}
we obtain the support with exponent vectors of a pair of broken lines $\lrp{\Par\lrp{\gamma^{(1)}},\Par\lrp{\gamma^{(2)}}}$ balanced at $\lrp{a+b} \tilde{\gamma}(\tau)$ with $I(\gamma^{(1)})= a \tilde{p}$ and $I(\gamma^{(2)}) = b \tilde{q}$. 
\end{proposition}

\begin{example}
We give an example illustrating \thref{prop:jp2bl} for a broken line segment in the $A_2$ scattering diagram.
Take the broken line segment $\tilde{\gamma}$ of Figure~\ref{fig:segment} with $\tilde{\gamma}(0) = \tilde{p}= \lrp{1,-5}$ and $\tilde{\gamma}(5)=\tilde{q}= \lrp{2,4}$.

\noindent
\begin{center}
\begin{minipage}{.85\linewidth}
\captionsetup{type=figure}
\begin{center}
\begin{tikzpicture}

    \def\d{.75}
    \def\l{4}

    \path (-\l,0) coordinate (3) --++ (\l,0) coordinate (0) --++ (\l,0) coordinate (1);
    \path (0,\l) coordinate (2) --++ (0,-2*\l) coordinate (4) --++ (\l,0) coordinate (5);

    \draw[name path = wall1, thick, ->] (3) -- (1) node [pos=0.1, below] {$1+z^{\lrp{-1,0}}$};
    \draw[name path = wall2, thick, ->] (2) -- (4) node [pos=0, left] {$1+z^{\lrp{0,1}}$};
    \draw[name path = wall3, thick, ->] (0) -- (5) node [pos=0.6, sloped, above] {$1+z^{\lrp{-1,1}}$};

    \node [circle, fill, inner sep = 1.5pt]  (q) at (2*\d,4*\d) {}; 
    \node [circle, fill, inner sep = 1.5pt]  (x21) at (0,2*\d) {}; 
    \node [circle, fill, inner sep = 1.5pt]  (x11) at (-\d,0) {}; 
    \node [circle, fill, inner sep = 1.5pt]  (x12) at (0,-2*\d) {}; 
    \node [circle, fill, inner sep = 1.5pt]  (p) at (\d,-5*\d) {}; 
    \node [circle, fill, inner sep = 1.5pt, cyan]  (r) at (-.5*\d,\d) {}; 

    \node at (2*\d+1,4*\d) {$\tilde{q}= (2,4)$}; 
    \node at (1.2,2*\d) {$\tilde{x}_{1}^{(2)}= (0,2)$}; 
    \node (x11lab) at (-2.5,1) {$\tilde{x}_{1}^{(1)}= (-1,0)$}; 
    \node (x12lab) at (-1.5,-2) {$\tilde{x}_{2}^{(1)}= (0,-2)$}; 
    \node (rlab) at (1.5,.5) {$\tilde{r}= \lrp{-\frac{1}{2},1}$}; 
    \node at (\d+1,-5*\d) {$\tilde{p}= (1,-5)$}; 

    \draw[very thick, magenta] (p) -- (x12) node [pos=.4,sloped, above] {\tc{black}{$z^{\lrp{1,-3}}$}} -- (x11) node [pos=.5,sloped, below] {\tc{black}{$z^{\lrp{1,-2}}$}} -- (r) node [pos=1.1,sloped, above] {\tc{black}{$z^{\lrp{-1,-2}}$}} -- (x21) --(q) node [pos=.6,sloped, above] {\tc{black}{$z^{\lrp{-1,-1}}$}};
    
    \draw[->, darkgray] (x11lab)--(x11);
    \draw[->, darkgray] (x12lab)--(x12);
    \draw[->, darkgray] (rlab)--(r);

\end{tikzpicture}

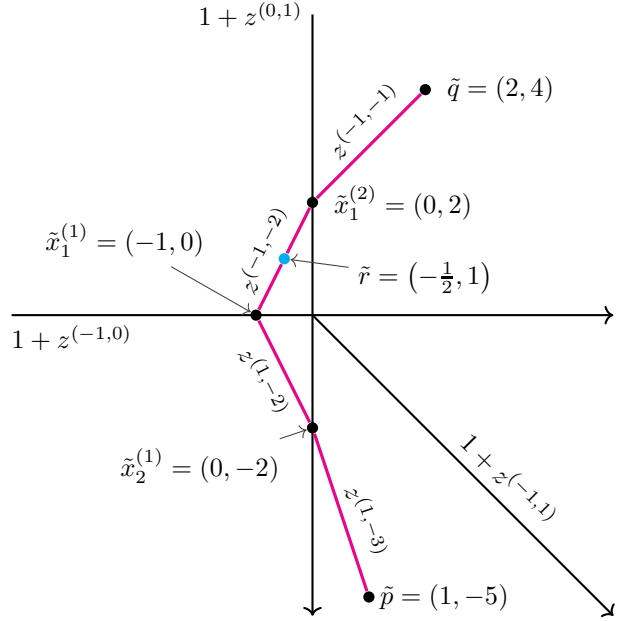
\captionof{figure}{\label{fig:segment} The broken line segment $\tilde{\gamma}$.} 

\end{center}
\end{minipage}
\end{center}
The relevant intermediate times are $t_{2}^{(1)}= 1$, $t_{1}^{(1)}= 2$, $t_0 = 2.5$, and $t_1^{(2)} = 3$.
That is, $\tilde{x}_2^{(1)}=\tilde{\gamma}(1)$, $\tilde{x}_1^{(1)}=\tilde{\gamma}(2)$, $\tilde{r}=\tilde{\gamma}(2.5)$, and $\tilde{x}_1^{(2)}=\tilde{\gamma}(3)$. 
We compute 
$\tilde{\rho}_0^{(1)} = \langle (1,0) ,  (1,-3) \rangle \cdot \langle (0,-1) , (1,-2) \rangle = 2$,
and similarly $\tilde{\rho}_0^{(2)} = 1$. 
For the integers $a$ and $b$, we need $\frac{b}{a+b} = \frac{2.5}{5}= \frac{1}{2}$.
As $\tilde{p}$ and $\tilde{q}$ are already in $M^\circ$ we can take $a=b=1$.
Then the exponent vectors are given by:
\begin{align*}
    m_0^{(1)} &= (0 , 2) + 3 (-1,-2) = (-3,-4) \\
    m_1^{(1)} &= (-1, 0) + 2 ( 1,-2) = (1 ,-4) \\
    m_2^{(1)} &= ( 0,-2) + 1 (-1,-3) = (1 ,-5) \\
    m_0^{(2)} &= (-1, 0) - 3 (-1,-2) = ( 2, 6) \\
    m_1^{(2)} &= ( 0, 2) - 2 (-1,-1) = ( 2, 4).
\end{align*}
Note that $\mono_{2}^{(1)} = a \tilde{p}$, $\mono_{1}^{(2)} = b \tilde{q}$, and $\mono_0^{(1)}+\mono_0^{(2)} = \lrp{a+b}\tilde{r}$ 
in agreement with Proposition \ref{prop:velocities}.
The resulting \red{balanced} pair of broken line can be seen in Figure~\ref{fig:constructedpair}.

\noindent
\begin{center}
\begin{minipage}{.85\linewidth}
\captionsetup{type=figure}
\begin{center}
\begin{tikzpicture}

    \def\d{.3}
    \def\l{4}
    \def\op{.25}
    \definecolor{color1}{rgb}{0,1,.25}
    \definecolor{color2}{rgb}{0,.25,1}
    \colorlet{color3}{color1!50!color2}
    \def\fudgeup{1.3}
    \def\fudgedown{.2}

    \path (-\l,0) coordinate (3) --++ (\l,0) coordinate (0) --++ (\l,0) coordinate (1);
    \path (0,\l) coordinate (2) --++ (0,-2*\l) coordinate (4) --++ (\l,0) coordinate (5);

    \draw[name path = wall1, thick, ->] (3) -- (1) node [pos=0.1, below] {$1+z^{\lrp{-1,0}}$};
    \draw[name path = wall2, thick, ->] (2) -- (4) node [pos=0, left] {$1+z^{\lrp{0,1}}$};
    \draw[name path = wall3, thick, ->] (0) -- (5) node [pos=0.6, sloped, above] {$1+z^{\lrp{-1,1}}$};

    \node [circle, fill, inner sep = 1.5pt, opacity =\op]  (qt) at (2*\d,4*\d) {}; 
    \node [circle, fill, inner sep = 1.5pt, opacity =\op]  (xt21) at (0,2*\d) {}; 
    \node [circle, fill, inner sep = 1.5pt, opacity =\op]  (xt11) at (-\d,0) {}; 
    \node [circle, fill, inner sep = 1.5pt, opacity =\op]  (xt12) at (0,-2*\d) {}; 
    \node [circle, fill, inner sep = 1.5pt, opacity =\op]  (pt) at (\d,-5*\d) {}; 
    \node [circle, fill, inner sep = 1.5pt, cyan, opacity =\op]  (rt) at (-.5*\d,\d) {}; 

    \draw[very thick, magenta, opacity =\op] (pt) -- (xt12) -- (xt11) -- (rt) -- (xt21) -- (qt);
    
    \node [circle, fill, inner sep = 1.5pt, color3] (r) at (-\d,2*\d) {};
    \node [circle, fill, inner sep = 1.5pt] (x11) at (-2.5*\d,0) {};
    \node [circle, fill, inner sep = 1.5pt] (x12) at (0,-10*\d) {};
    \node [circle, fill, inner sep = 1.5pt] (x21) at (0,5*\d) {};
    
    \path[name path = lin1] (x12)--++(\fudgedown,-5*\fudgedown);
    \path[name path = bot] (-\l,-\l)--(\l,-\l);
    \path[name intersections={of=lin1 and bot,by=end1}];    

    \path[name path = lin2] (x21)--++(\fudgeup,2*\fudgeup);
    \path[name path = top] (-\l,\l)--(\l,\l);
    \path[name intersections={of=lin2 and top,by=end2}];    

    \draw [very thick, color1] (r) -- (x11) -- (x12) node [pos=.4,sloped, below] {\tc{black}{$z^{\lrp{1,-4}}$}} -- (end1) node [pos=.6,sloped, above] {\tc{black}{$z^{\lrp{1,-5}}$}};
    \draw [very thick, color2] (r) -- (x21) node [pos=.7,sloped, above] {\tc{black}{$z^{\lrp{2,6}}$}} -- (end2) node [pos=.6,sloped, above] {\tc{black}{$z^{\lrp{2,4}}$}};
    
    \node (m01) at (-2,.4) {$z^{\lrp{-3,-4}}$};    
    \draw[->, darkgray] (m01)--($(x11)!.4!(r)$);

    \node (rlab) at (-2,1.5) {$r=\lrp{-1,2}$};
    \draw[->, darkgray] (rlab)--(r);

    \node at (-.55,-2.6) {$\gamma^{(1)}$};
    \node at (1.65,4) {$\gamma^{(2)}$};
    
\end{tikzpicture}

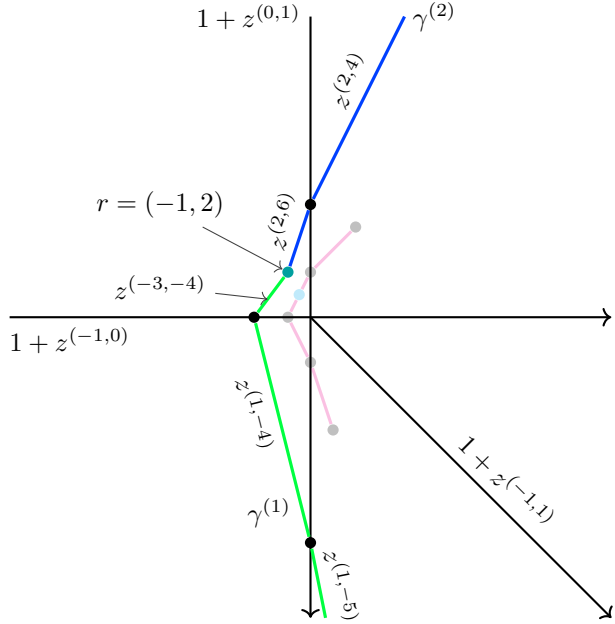
\captionof{figure}{\label{fig:constructedpair} The \red{balanced} pair of broken lines $\lrp{\gamma^{(1)},\gamma^{(2)}}$.} 

\end{center}
\end{minipage}
\end{center}

\end{example}

\section{The main result}\label{sec:MainThm}

\change{Throughout this section we say that a smooth scheme $V$ is a quotient of a cluster $\cA $-variety if it is isomorphic up to codimension 2 to one of the quotients described in Section~\ref{sec:RedDim}. Similarly, we say $V$ is a fiber of a cluster $\cX $-variety if it is isomorphic up to codimension 2 to a variety of the form $ \cX_\phi$ described in Section~\ref{sec:RedDim}. In particular, all log Calabi-Yau surfaces are of this kind, and they are classified in \cite{Travis2}. } 

\begin{theorem}\thlabel{MainThm}
\change{Let $V$ be one of the following: \begin{enumerate}
    \item \label{cv} a cluster variety, or
    \item \label{A/T} a quotient of a cluster $\cA$-variety, or
    \item \label{Xt} a fiber of a cluster $\cX$-variety.
\end{enumerate}
}
Then a closed subset $S$ of $V^\trop\lrp{\R}$ is positive if and only if $S$ is broken line convex.
\end{theorem}

\begin{proof}
In view of \thref{lem:Apos}, \thref{lem:Xpos}, \change{\thref{rem:PosRedDim},} \thref{lem:projection_and_convexity}, \thref{lem:slice_and_convexity},  and \change{\thref{rem:BLCRedDim},} we only need to prove the statement when the dual $V^\vee$ is a cluster variety of type $\cAp$. Since we know the injectivity assumption holds for varieties of the form $\cAp$ we are free to use all the results of Sections~\ref{sec:key_lemma} and \ref{sec:in_alg}.

First assume $S$ is broken line convex in $V^\trop(\R)\cong\widetilde{M}^{\circ}_{\R}$. Let $a$ and $b$ be positive integers and $ p,q,r $ points in ${ V^\trop(\Z) \cong \widetilde{M}^{\circ}}$ such that $ p \in a S, q \in b S$ and $ \alpha(p,q,r)\neq 0$. 
We have to show that $r \in (a+b) S$.
Since $\alpha(p,q,r)\neq 0$ there exists a pair of broken lines $\lrp{\gamma^{(1)}, \gamma^{(2)}}$ balanced at $r$ with  $I(\gamma^{(1)})= p$ and $I(\gamma^{(2)})= q$. 
We use \thref{prop:bl2jp} to construct a broken line segment connecting $\frac{p}{a}\in S$ to $\frac{q}{b}\in S$ and passing through $\frac{r}{a+b}$. Since $S $ is broken line convex we have that $\frac{r}{a+b}\in S$. Therefore, $r\in (a+b) S$ as desired. 
Next, let $a=0$ or $b=0$.
Assume without loss of generality that $a=0$.
Then $p=0 \in V^\trop(\R)$ and $\alpha\lrp{p,q,r}$ is the coefficient of $\tf_r$ in the product $\tf_0 \cdot \tf_q = \tf_q$.
Since $\alpha\lrp{p,q,r} \neq 0$, $r=q$ and we have $r \in (0+b)S$.  We conclude that $S$ is positive.

Now assume $S$ is positive. 
Let $\tilde{p}$,  $\tilde{q}$ be in $S\cap V^\trop(\Q) $ and let $\tilde{\gamma}$ be a broken line segment connecting $\tilde{p}$ and $\tilde{q}$.  Say $\tilde{\gamma}(0) = \tilde{p}$ and $\tilde{\gamma}(T) = \tilde{q}$.
We will first show that every rational point of $\tilde{\gamma}$ is contained in $S$. 
Note that a rational number in $\lrp{0,1}$ can be expressed as $\frac{b}{a+b}$ for some pair of positive integers $a$, $b$.
Then for each $\beta \in \lrp{0,1}\cap \Q$, 
we can use \thref{prop:jp2bl} to construct a pair of broken lines $\lrp{\gamma^{(1)},\gamma^{(2)}}$ balanced at $\lrp{a+b} \tilde{\gamma}\lrp{\beta T} \in V^\trop(\Z)$,
where $a$ and $b$ are positive integers satisfying $\beta = \frac{b}{a+b}$,
$I(\gamma^{(1)}) = a \tilde{p} \in a S\cap V^\trop(\Z)$, and $I(\gamma^{(2)}) = b\tilde{q}\in b S \cap V^\trop(\Z)$.
Since $S$ is positive, this implies  $\tilde{\gamma}\lrp{\beta T}  \in S$.
Since $\beta$ was an arbitrary rational number in $\lrp{0,1}$, all rational points of $\tilde{\gamma}$ are in $S$, and since $S$ is closed $\Supp\lrp{\tilde{\gamma}}$ is entirely contained in $S$.  
\end{proof}

\begin{remark} \thlabel{rk:compute}
\thref{MainThm} provides a tractable way to construct positive subsets of $V^{\trop}(\R)$, or to check positivity of a given subset for \red{two dimensional scattering diagrams}. \change{ Explicitly, these are scattering diagrams for \begin{enumerate} 
\item two dimensional cluster varieties, or
\item quotients of $\cA$ varieties with $H\subset N$ of corank 2, or
\item fibers of $\cX$ with $H\subset N$ of corank 2.
\end{enumerate}}
The key comes from Remark \ref{rk:incoming}: 
\change{in dimension 2,} broken lines bending at outgoing walls necessarily bend away from the origin.
So, to check if a closed subset $S$ of $V^\trop(\R)$ is broken line convex,
we can first ask whether or not $S$ is identified with a convex subset of $\bL\otimes \R$ by a choice of seed, where $\bL$ is the lattice identified with $V^\trop(\Z)$.
If the answer is {\it{yes}},
$S$ stands a chance at being broken line convex.
It is not guaranteed to be broken line convex though,
since a broken line segment may begin at a point in $S$, exit $S$ and bend over a wall, then re-enter $S$.
We illustrate this for the ``type $G_2$'' scattering diagram in \thref{ex:G2}.
In this case, the broken line convex hull of the given vectors is strictly larger that the usual convex hull.

If $S$ contains the origin, bending over an outgoing wall will only push a broken line segment $\tilde{\gamma}$ that has exited $S$ further away from $S$, obstructing re-entry. 
So, if $S$ contains the origin, we only need to consider bends at incoming walls. 
This greatly reduces what remains to check since only initial walls \change{of $\scat^{\cAp}_\seed$ are incoming, and the number of initial walls of $\scat^{\cAp}_\seed$ is just  $\lrm{\Iuf}$.}
In addition to convexity in $\bL\otimes\R$ then, if $S$ contains $0$, we simply have to ask whether a broken line segment bending maximally at initial walls can leave $S$ if its endpoints are in $S$.
If this cannot happen, we know $S$ is positive. 

\red{Moreover, we can say more in this 2 dimensional case.  
It is clear (in any dimension) that each choice of seed will identify a broken line convex subset $S$ of $V^\trop(\R)$ with a convex subset of $\bL\otimes \R$.
However, it is {\emph{not}} clear that $S$ must be broken line convex if it is identified with a convex subset of $\bL\otimes \R$ by every choice of seed.
Mandel showed in \cite{Travis} that this indeed holds in dimension 2.
We recover his result by making the following observation.
If, after fixing a seed $\seed$, a broken line segment bends maximally toward the origin at an incoming wall, then this broken line segment is straight in the affine structure associated to a different choice of seed $\seed'$.
This is because the piecewise linear identification $T_k$ of \cite[Theorem~1.24]{GHKK} precisely straightens a maximal bending toward the origin at the corresponding wall $\wall_k$.}
\end{remark}

We finish with an example.  
As described in \thref{rk:compute}, \thref{MainThm} allows us to easily describe positive subsets of $V^\trop(\R)$ and in turn projective varieties compactifying cluster varieties.
We illustrate this idea below.

\begin{example}\thlabel{ex:G2}
Let $\cA$ be the finite-type cluster variety associated to the $G_2$ root system.
The exchange matrix is 
\eqn{\epsilon = \begin{pmatrix}
0 & 3 \\
-1 & 0
\end{pmatrix}.
}
Note that in this case the injectivity assumption of \cite{GHKK} is satisfied, so we can work directly with $\cA$ itself, as opposed to $\cAp$.
The scattering diagram is provided in \cite[Figure~1.2]{GHKK} and reproduced below for the reader's convenience.

\noindent
\begin{center}
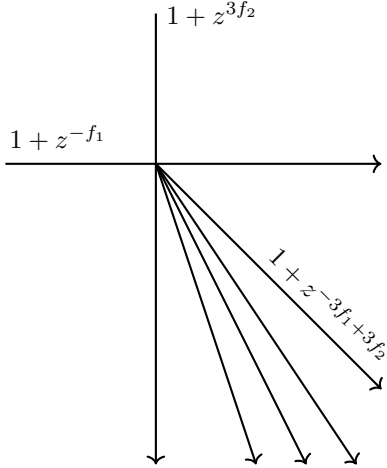

\begin{minipage}{.85\linewidth}
\captionsetup{type=figure}
\centering
\noindent
\begin{tikzpicture}
    \def\x{2}
    \def\d{1}
    \def\l{3}
    \def\op{.25}
    \definecolor{color1}{rgb}{0,1,.25}
    \definecolor{color2}{rgb}{0,.25,1}
    \colorlet{color3}{color1!50!color2}

    \path (-\x,0) coordinate (3) --++ (\x,0) coordinate (0) --++ (\l,0) coordinate (1);
    \path (0,\x) coordinate (2) --++ (0,-\x-\l-\d) coordinate (4) --++ (1.333*\d,0) coordinate (5)--++ (.667*\d,0) coordinate (6)--++ (.667*\d,0) coordinate (7);
    \coordinate (8) at (\l,-\l);

    \draw[name path = wall1, thick, ->] (3) -- (1);
    \draw[name path = wall2, thick, ->] (2) -- (4);
    \draw[name path = wall3, thick, ->] (0) -- (5);
    \draw[name path = wall4, thick, ->] (0) -- (6);
    \draw[name path = wall5, thick, ->] (0) -- (7);
    \draw[name path = wall6, thick, ->] (0) -- (8) node [pos=.7, sloped, above] {$1+z^{-3 f_1+3 f_2}$};

    \node at (.75,\x) {$1+z^{3 f_2}$};
    \node at (-\x+.7*\d,.35) {$1+z^{-f_1}$};

\end{tikzpicture}
\captionof{figure}{\label{fig:G2Scat} Scattering diagram associated to the $G_2$ root system.  Three walls in the fourth quadrant need a scattering function to be specified.  In clockwise order, they are as follows: $1+ z^{-2 f_1+3f_2}$, $1+ z^{-3 f_1+6f_2}$, and $1+ z^{- f_1+3f_2}$. } 
\end{minipage}
\end{center}

The primitive vectors along the rays in Figure~\ref{fig:G2Scat} are the $\gv$-vectors of cluster variables.
We will take the {\it{broken line convex hull}} of these points-- that is, the smallest broken line convex set containing all of the points-- in order to define a natural projective compactification of $\cA$.
As there are only two incoming walls (the coor.  Bossindinate axes) and bends at outgoing walls are always away from the origin, 
it is straightforward to verify that the broken line convex hull of these points is as follows.

\noindent
\begin{center}
\begin{minipage}{.85\linewidth}
\captionsetup{type=figure}
\begin{center}
\begin{tikzpicture}

    \def\x{2}
    \def\d{1}
    \def\l{3}
    \def\op{.25}
    \definecolor{color1}{rgb}{0,1,.25}
    \definecolor{color2}{rgb}{0,.25,1}
    \colorlet{color3}{color1!50!color2}

    \path (-\x,0) coordinate (3) --++ (\x,0) coordinate (0) --++ (\l,0) coordinate (1);
    \path (0,\x) coordinate (2) --++ (0,-\x-\l-\d) coordinate (4) --++ (1.333*\d,0) coordinate (5)--++ (.667*\d,0) coordinate (6)--++ (.667*\d,0) coordinate (7);
    \coordinate (8) at (\l,-\l);

    \draw[name path = wall1, thick, ->] (3) -- (1);
    \draw[name path = wall2, thick, ->] (2) -- (4);
    \draw[name path = wall3, thick, ->] (0) -- (5);
    \draw[name path = wall4, thick, ->] (0) -- (6);
    \draw[name path = wall5, thick, ->] (0) -- (7);
    \draw[name path = wall6, thick, ->] (0) -- (8);
 
    \node [circle, fill, inner sep = 1.5pt, color = color3] (v1) at (\d,0) {};
    \node [circle, fill, inner sep = 1.5pt, color = color3] (v2) at (0,\d) {};
    \node [circle, fill, inner sep = 1.5pt, color = color3] (v3) at (-\d,0) {};
    \node [circle, fill, inner sep = 1.5pt, color = color3] (v4) at (0,-\d) {};
    \node [circle, fill, inner sep = 1.5pt, color = color3] (v8) at (\d,-\d) {};
    \node [circle, fill, inner sep = 1.5pt, color = color3] (v6) at (\d,-2*\d) {};
    \node [circle, fill, inner sep = 1.5pt, color = color3] (v5) at (\d,-3*\d) {};
    \node [circle, fill, inner sep = 1.5pt, color = color3] (v7) at (2*\d,-3*\d) {};

    \coordinate (vnew) at (0,1.5*\d);

    \draw[color3, thick, fill= color3, fill opacity=\op] (v1.center) -- (vnew) -- (v3.center) -- (v5.center) -- (v7.center) -- cycle;
    
    \node at (.6*\d,1.6*\d) {$\lrp{0,\frac{3}{2}}$}; 
    
\end{tikzpicture}

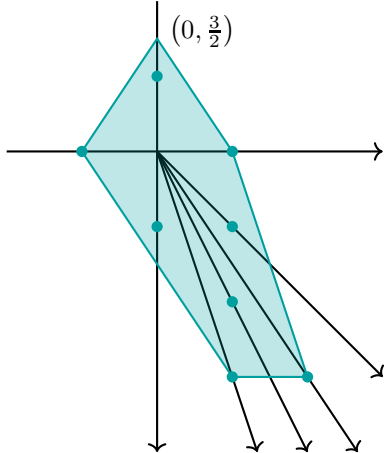
\captionof{figure}{\label{fig:G2Poly} Broken line convex hull $S$ of $\gv$-vectors of cluster variables for $G_2$ cluster variety.
The $\gv$-vector of each cluster variable is displayed as a dot.  Observe that the boundary of $S$ consists of 4 broken line segments, each having a pair of $\gv$-vectors as end points and crossing a single wall.} 

\end{center}
\end{minipage}
\end{center}

Note that in this example, if we were to take the convex hull of these points using the vector space structure on $M^\circ \otimes\R$ given by the choice of initial seed, 
we would get a smaller polygon.

\noindent
\begin{center}
\begin{minipage}{.85\linewidth}
\captionsetup{type=figure}
\begin{center}
\begin{tikzpicture}

    \def\x{2}
    \def\d{1}
    \def\l{3}
    \def\op{.25}
    \definecolor{color1}{rgb}{0,1,.25}
    \definecolor{color2}{rgb}{0,.25,1}
    \colorlet{color3}{color1!50!color2}

    \path (-\x,0) coordinate (3) --++ (\x,0) coordinate (0) --++ (\l,0) coordinate (1);
    \path (0,\x) coordinate (2) --++ (0,-\x-\l-\d) coordinate (4) --++ (1.333*\d,0) coordinate (5)--++ (.667*\d,0) coordinate (6)--++ (.667*\d,0) coordinate (7);
    \coordinate (8) at (\l,-\l);

    \draw[name path = wall1, thick, ->] (3) -- (1);
    \draw[name path = wall2, thick, ->] (2) -- (4);
    \draw[name path = wall3, thick, ->] (0) -- (5);
    \draw[name path = wall4, thick, ->] (0) -- (6);
    \draw[name path = wall5, thick, ->] (0) -- (7);
    \draw[name path = wall6, thick, ->] (0) -- (8);
    
    \node [circle, fill, inner sep = 1.5pt, color = color3] (v1) at (\d,0) {};
    \node [circle, fill, inner sep = 1.5pt, color = color3] (v2) at (0,\d) {};
    \node [circle, fill, inner sep = 1.5pt, color = color3] (v3) at (-\d,0) {};
    \node [circle, fill, inner sep = 1.5pt, color = color3] (v4) at (0,-\d) {};
    \node [circle, fill, inner sep = 1.5pt, color = color3] (v8) at (\d,-\d) {};
    \node [circle, fill, inner sep = 1.5pt, color = color3] (v6) at (\d,-2*\d) {};
    \node [circle, fill, inner sep = 1.5pt, color = color3] (v5) at (\d,-3*\d) {};
    \node [circle, fill, inner sep = 1.5pt, color = color3] (v7) at (2*\d,-3*\d) {};

    \coordinate (vnew) at (0,1.5*\d);

    \draw[color3, thick, fill= color3, fill opacity=\op] (v1.center) -- (v2.center) -- (v3.center) -- (v5.center) -- (v7.center) -- cycle;

    \draw[blue] (v1) -- (vnew) -- (v3);
    
    \node at (.6*\d,1.6*\d) {$\lrp{0,\frac{3}{2}}$}; 
    \node[color= blue] at (-.6*\d,.9*\d) {$\tilde{\gamma}$}; 
    
\end{tikzpicture}

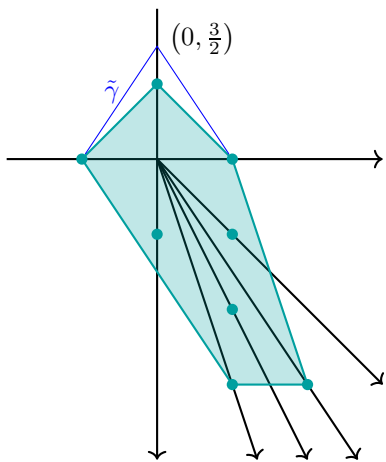
\captionof{figure}{\label{fig:G2ConvHull} The usual convex hull of $\gv$-vectors of cluster variables in $M^\circ \otimes \R$ is not broken line convex as the endpoints of $\tilde{\gamma}$ are contained in this set but $\Supp(\tilde{\gamma})$ is not.} 

\end{center}
\end{minipage}
\end{center}

The missing points prevent the smaller polygon from being positive and defining a projective variety-- $\tf_{\lrp{1,0}} \cdot \tf_{\lrp{-1,0}} =  \tf_{\lrp{0,0}} + \tf_{\lrp{0,3}}$, but $\lrp{0,3}$ is not contained in twice the convex hull of the $\gv$-vectors.
Using the broken line convex hull on the other hand, \thref{MainThm} implies that $S$ and its dilations define a polarized projective variety compactifying $\cA$. 
\end{example}

\bibliographystyle{hep}
\bibliography{bibliography.bib}

\begin{thebibliography}{GHKK18}

\bibitem[BCMN]{Grass}
L.~{Bossinger}, M.-W. {Cheung}, T.~{Magee} and A.~{N\'ajera Ch\'avez},
\newblock On cluster duality, mirror symmetry and toric degenerations of
  {Grassmannians},
\newblock In progress.

\bibitem[BFMN20]{BFMN18}
L.~{Bossinger}, B.~{Fr\'{\i}as-Medina}, T.~{Magee} and A.~{N\'{a}jera
  Ch\'{a}vez}, \textsl{ Toric degenerations of cluster varieties and cluster
  duality},
\newblock Compos. Math. \textbf{ 156}(10), 2149--2206 (2020).

\bibitem[Che19]{mandy}
M.-W. Cheung, \textsl{ Theta functions and quiver Grassmannians},
\newblock arXiv preprint arXiv:1906.12299  (2019).

\bibitem[CM]{BatFin}
M.-W. {Cheung} and T.~{Magee},
\newblock Towards {Batyrev} duality for finite-type cluster varieties,
\newblock In progress.

\bibitem[CPS11]{CPS}
M.~Carl, M.~Pumperla and B.~Siebert, \textsl{ A tropical view of
  Landau-Ginzburg models},
\newblock preprint  (2011),
\newblock Available at
  \url{https://www.math.uni-hamburg.de/home/siebert/preprints/LGtrop.pdf}.

\bibitem[FG06]{FG_Moduli}
V.~Fock and A.~Goncharov, \textsl{ Moduli spaces of local systems and higher
  {T}eichm\"{u}ller theory},
\newblock Publ. Math. Inst. Hautes \'{E}tudes Sci. \textbf{ 103}(1), 1--211
  (2006).

\bibitem[FG09]{FG_Quantization}
V.~V. Fock and A.~B. Goncharov, \textsl{ Cluster ensembles, quantization and
  the dilogarithm},
\newblock Ann. Sci. \'{E}c. Norm. Sup\'{e}r. (4) \textbf{ 42}(6), 865--930
  (2009).

\bibitem[FZ02]{FZ_I}
S.~Fomin and A.~Zelevinsky, \textsl{ Cluster algebras. {I}. {F}oundations},
\newblock J. Amer. Math. Soc. \textbf{ 15}(2), 497--529 (2002).

\bibitem[GHK15a]{GHK_birational}
M.~Gross, P.~Hacking and S.~Keel, \textsl{ Birational geometry of cluster
  algebras},
\newblock Algebr. Geom. \textbf{ 2}(2), 137--175 (2015).

\bibitem[GHK15b]{ghk}
M.~Gross, P.~Hacking and S.~Keel, \textsl{ Mirror symmetry for log
  {C}alabi-{Y}au surfaces {I}},
\newblock Publ. Math. Inst. Hautes \'{E}tudes Sci. \textbf{ 122}, 65--168
  (2015).

\bibitem[GHKK18]{GHKK}
M.~Gross, P.~Hacking, S.~Keel and M.~Kontsevich, \textsl{ Canonical bases for
  cluster algebras},
\newblock J. Amer. Math. Soc. \textbf{ 31}(2), 497--608 (2018).

\bibitem[GP10]{gp}
M.~Gross and R.~Pandharipande, \textsl{ Quivers, curves, and the tropical
  vertex},
\newblock Port. Math. \textbf{ 67}(2), 211--259 (2010).

\bibitem[GS16]{GSTheta}
M.~Gross and B.~Siebert,
\newblock Theta functions and mirror symmetry,
\newblock in \textsl{ Surveys in differential geometry 2016. {A}dvances in
  geometry and mathematical physics}, volume~21 of \textsl{ Surv. Differ.
  Geom.}, pages 95--138, Int. Press, Somerville, MA, 2016.

\bibitem[GSV05]{GSV05}
M.~Gekhtman, M.~Shapiro and A.~Vainshtein, \textsl{ Cluster algebras and
  {W}eil-{P}etersson forms},
\newblock Duke Math. J. \textbf{ 127}(2), 291--311 (2005).

\bibitem[GSV10]{GSV10}
M.~Gekhtman, M.~Shapiro and A.~Vainshtein,
\newblock \textsl{ Cluster algebras and {P}oisson geometry}, volume 167 of
  \textsl{ Mathematical Surveys and Monographs},
\newblock American Mathematical Society, Providence, RI, 2010.

\bibitem[KY19]{KeelYu}
S.~{Keel} and T.~Y. {Yu}, \textsl{ The {F}robenius structure theorem for affine
  log {Calabi-Yau} varieties containing a torus},
\newblock preprint  (2019), {arXiv:1908.09861 [math.AG]}.

\bibitem[{Man}16]{Travis}
T.~{Mandel}, \textsl{ Tropical theta functions and log Calabi–Yau surfaces},
\newblock Sel. Math. New Ser. \textbf{ 22}(3), 1289--1335 (2016).

\bibitem[Man19]{Travis2}
T.~Mandel, \textsl{ Classification of rank 2 cluster varieties},
\newblock SIGMA Symmetry Integrability Geom. Methods Appl. \textbf{ 15}, Paper
  042, 32 (2019).

\bibitem[Rei10]{reineke}
M.~Reineke, \textsl{ Poisson automorphisms and quiver moduli},
\newblock J. Inst. Math. Jussieu \textbf{ 9}(3), 653--667 (2010).

\bibitem[RW19]{RW}
K.~{Rietsch} and L.~{Williams}, \textsl{ {Newton-Okounkov} bodies, cluster
  duality, and mirror symmetry for {Grassmannians}},
\newblock Duke Math. J. \textbf{ 168}(18), 3437--3527 (2019).

\end{thebibliography}

\noindent{\sc{Man-Wai Cheung\\
Department of Mathematics, One Oxford Street, Cambridge, Harvard University, MA 02138, USA}}\\
{\it{e-mail:}} \href{mailto:mwcheung@math.harvard.edu}{mwcheung@math.harvard.edu} \medskip

\change{\noindent{\sc{Timothy Magee\\
Department of Mathematics,
Faculty of Natural \& Mathematical Sciences,
King's College London,
Strand,
London WC2R 2LS,
United Kingdom}}\\
{\it{e-mail:}} \href{mailto:timothy.magee@kcl.ac.uk}{timothy.magee@kcl.ac.uk} }\medskip

\noindent{\sc{Alfredo N\'ajera Ch\'avez\\
CONACyT-Instituto de Matem\'aticas UNAM Unidad Oaxaca, Le\'on 2, altos, Oaxaca de Ju\'arez, Centro Hist\'orico, 68000 Oaxaca, Mexico}}\\
{\it{e-mail:}} \href{mailto:najera@matem.unam.mx}{najera@matem.unam.mx}

\end{document}